\documentclass[11pt]{article}
\usepackage[latin1]{inputenc}
\usepackage[english]{babel}
\usepackage{amsmath,amsfonts,amscd,amssymb}
\usepackage[all]{xy}
\usepackage{makeidx}
\usepackage[active]{srcltx} 

\newtheorem{theorem}{Theorem}[section] 
\newtheorem{lemma}[theorem]{Lemma}     
\newtheorem{corollary}[theorem]{Corollary}
\newtheorem{proposition}[theorem]{Proposition}

\newtheorem{conjecture}[theorem]{Conjecture}  
\newtheorem{definition}[theorem]{Definition}

\newtheorem{remark}[theorem]{Remark}

\newtheorem{example}[theorem]{Example}
\newtheorem{notation}[theorem]{Notation} 
\newtheorem{examples}[theorem]{Examples}

\newsavebox{\proofsavebox}
\sbox{\proofsavebox}
    {\unitlength 1pt\begin{picture}(6.5,6.5)%
     \put(0,0){\framebox(6.5,6.5){}}\end{picture}}
\newcommand\proofbox{\usebox{\proofsavebox}\relax}

\newenvironment{proof}
  {{\sc Proof.-}}
  {\hspace*{\fill}\proofbox\endlist\par}

\let\noi=\noindent

\def\Z{\mathbf{Z}} 
\def\Q{\mathbf{Q}} 
\def\P{\mathcal{P}}

\def\spec{\mbox{Spec}}

\def\gr{\mbox{gr}}
\def\init{\mbox{in}}
\def\notin{\mbox{$\in$ \hspace{-.8em}/}} 

\def\g{{\Gamma}}
\def\h{{\Phi}}
\def\md{{\operatorname{mod}}}
\def\he{{\operatorname{ht}}}

\makeindex
\title{Extending a valuation centered in a local domain to the formal
completion.}
\author{F. J. Herrera Govantes\footnote{These authors are partially supported by MTM2007-66929, MTM2010-19298 and FEDER.}\\
Departamento de \'Algebra\\
Facultad de Matem\'aticas\\
Avda. Reina Mercedes, s/n\\
Universidad de Sevilla\\
41012 Sevilla, Spain \\ email: jherrera@algebra.us.es \and
M. A. Olalla Acosta$^*$\\
Departamento de \'Algebra\\
Facultad de Matem\'aticas\\
Avda. Reina Mercedes, s/n\\
Universidad de Sevilla\\
41012 Sevilla, Spain\\
email: miguelolalla@algebra.us.es \and
M. Spivakovsky\\
Institut de Math\'ematiques de Toulouse\\
UMR 5219 du CNRS,\\
Universit\'e Paul Sabatier\\
118, route de Narbonne\\
31062 Toulouse cedex 9, France.\\email:
mark.spivakovsky@math.univ-toulouse.fr\and
B. Teissier\footnote{This author is grateful for the hospitality of
  the RIMS in Kyoto, where a part of this project was completed.}{}\\
Equipe `` G\'eom\'etrie et Dynamique'',\\
Institut Math\'ematique de Jussieu,\\ UMR 7586 du CNRS\\
175 Rue du Chevaleret\\
F-75013 Paris, France.\\ email: teissier@math.jussieu.fr }

\begin{document}
\maketitle

\medskip

\section{Introduction}
\label{In}

All the rings in this paper will be commutative with 1.

Let $(R,m,k)$ be a local noetherian domain with field of fractions
$K$ and $R_\nu$ a valuation ring, dominating $R$ (not necessarily
birationally). Let $\nu|_K:K^*\twoheadrightarrow\Gamma$ be the
restriction of $\nu$ to $K$; by definition, $\nu|_K$ is centered at
$R$. Let $\hat R$ denote the $m$-adic completion of $R$. In the
applications of valuation theory to commutative algebra and the
study of singularities, one is often induced to replace $R$ by its
$m$-adic completion $\hat R$ and $\nu$ by a suitable extension
$\hat\nu_-$ to $\frac{\hat R}P$ for a suitably chosen prime ideal
$P$, such that $P\cap R=(0)$ (one specific application we have in mind has to do with the approaches to proving the Local Uniformization Theorem in arbitrary characteristic such as \cite{Spi2} and \cite{Te}). The first reason is that the ring
$\hat R$ is not in general an integral domain, so that we can only
hope to extend $\nu$ to a \textit{pseudo-valuation} on $\hat R$,
which means precisely a valuation $\hat\nu_-$ on a quotient
$\frac{\hat R}P$ as above. The prime ideal $P$ is called the
\textit{support} of the pseudo-valuation. It is well known and not
hard to prove that such extensions $\hat\nu_-$ exist for some
minimal prime ideals $P$ of $\hat R$. Although, as we shall see, the
datum of a valuation $\nu$ determines a unique minimal prime of
$\hat R$ when $R$ is excellent, in general there are many possible
primes $P$ as above and for a fixed $P$ many possible extensions
$\hat\nu_-$. This is the second reason to study extensions
$\hat\nu_-$.

The purpose of this paper is to give, assuming
that $R$ is excellent, a systematic description of all such
extensions $\hat\nu_-$ and to identify certain classes of extensions
which are of particular interest for applications. In fact, the only
assumption about $R$ we ever use in this paper is a weaker and more
natural condition than excellence, called the G condition, but we
chose to talk about excellent rings since this terminology seems to
be more familiar to most people. For the reader's convenience, the
definitions of excellent and G-rings are recalled in the
Appendix. Under this assumption, we study
\textbf{extensions to (an
  integral quotient of) the completion $\hat R$ of a valuation $\nu$
  and give descriptions of the
  valuations with which such extensions are composed. In particular we give criteria
  for the uniqueness of the extension if certain simple data on these
  composed valuations are fixed}.
\noindent

We conjecture (see statement 5.19 in \cite{Te} and  Conjecture \ref{teissier} below for a stronger and
more precise statement) that\par\noindent
 \textbf{given an excellent local ring $R$ and a valuation $\nu$ of
   $R$ which is positive on its maximal ideal $m$, there exists a prime ideal
   $H$ of the $m$-adic completion $\hat R$ such that $H\bigcap R=(0)$
   and an extension of $\nu$ to $\frac{\hat R}{H}$ which has the same
   value group as $\nu$.}

When studying extensions of $\nu$ to the completion of $R$, one is led to
the study of its extensions to the henselization $\tilde R$ of $R$ as a
natural first step. This, in turn, leads to the study of extensions of
$\nu$ to finitely generated local strictly \'etale extensions $R^e$ of
$R$. We therefore start out by letting $\sigma:R\rightarrow R^\dag$
denote one of the three operations of completion, (strict)
henselization, or a finitely generated local strictly \'etale
extension:
\begin{eqnarray}
R^\dag&=&\hat R\quad\mbox{ or}\label{eq:hatR}\\
R^\dag&=&\tilde R\quad\mbox{ or}\label{eq:tildeR}\\
R^\dag&=&R^e.\label{eq:Re}
\end{eqnarray}
The ring $R^\dag$ is local; let $m^\dag$ denote its maximal ideal.
The homomorphisms
$$
R\rightarrow\tilde R\quad\text{ and }\quad R\rightarrow R^e
$$
are regular for any ring $R$; by definition, if $R$ is an excellent ring
then the completion homomorphism is regular (in fact, regularity of
the completion homomorphism is precisely the defining property of
G-rings; see the Appendix for the definition of regular
homomorphism).

Let $r$ denote the (real) rank of $\nu$. Let $(0)=\Delta_r\subsetneqq
\Delta_{r-1}\subsetneqq\dots\subsetneqq \Delta_0=\Gamma$ be the isolated
subgroups of $\Gamma$ and $P_0=(0)\subsetneqq
P_1\subseteq\dots\subseteq P_r=m$ the prime valuation ideals of $R$,
which need not, in general, be distinct. In this paper, we will assume
that $R$ is excellent. Under this assumption, we will canonically
associate to $\nu$ a chain $H_1\subset H_3\subset\dots\subset
H_{2r+1}=mR^\dag$ of ideals of $R^\dag$, numbered by odd integers from
1 to $2r+1$, such that $H_{2\ell+1}\cap R=P_\ell$ for $0\le \ell\le
r$. We will show that all the ideals $H_{2\ell+1}$ are prime. We will
define $H_{2\ell}$ to be the unique minimal prime ideal of $P_\ell
R^\dag$, contained in $H_{2\ell+1}$ (that such a minimal prime is
unique follows from the regularity of the homomorphism $\sigma$).

We will thus obtain, in the cases (\ref{eq:hatR})--(\ref{eq:Re}), a
chain of $2r+1$ prime ideals
$$
H_0\subset H_1\subset\dots\subset H_{2r}=H_{2r+1}=mR^\dag,
$$
satisfying $H_{2\ell}\cap R=H_{2\ell+1}\cap R=P_\ell$ and such that
$H_{2\ell}$ is a minimal prime of $P_\ell R^\dag$ for $0\le \ell\le
r$. Moreover, if $R^\dag=\tilde R$ or $R^\dag=R^e$, then
$H_{2\ell}=H_{2\ell+1}$. We call $H_i$ the $i${\bf-th implicit prime
  ideal} of $R^\dag$, associated to $R$ and $\nu$. The ideals $H_i$
behave well under local blowing ups along $\nu$ (that is, birational
local homomorphisms $R\to R'$ such that $\nu$ is centered in $R'$),
and more generally under \textit{$\nu$-extensions} of $R$ defined
below in subsection \ref{trees}. This means that given any local
blowing up along $\nu$ or $\nu$-extension $R\rightarrow R'$, the
$i$-th implicit prime ideal $H'_i$ of ${R'}^\dag$ has the property
that $H'_i\cap R^\dag=H_i$. This intersection has a meaning in view of Lemma
\ref{factor} below.

For a prime ideal $P$ in a ring $R$, $\kappa(P)$ will denote the residue
field $\frac{R_P}{PR_P}$.

Let $(0)\subsetneqq \mathbf{m}_1\subsetneqq\dots\subsetneqq
\mathbf{m}_{r-1}\subsetneqq\mathbf{m}_r=\mathbf{m}_\nu$ be the prime
ideals of the valuation ring $R_\nu$. By definitions, our valuation
$\nu$ is a composition of $r$ rank one valuations
$\nu=\nu_1\circ\nu_2\dots\circ\nu_r$, where $\nu_\ell$ is a valuation
of the field $\kappa(\mathbf{m}_{\ell-1})$, centered at
$\frac{(R_\nu)_{\mathbf{m}_\ell}}{\mathbf{m}_{\ell -1}}$ (see
\cite{ZS}, Chapter VI, \S10, p. 43 for the definition of composition of
valuations; more information and a simple example of composition is
given below in subsection \ref{trees}, where we interpret each
$\mathbf{m}_\ell$ as the limit of a tree of ideals).

If $R^\dag=\tilde R$, we will prove that there is a unique extension
$\tilde\nu_-$ of $\nu$ to $\frac{\tilde R}{H_0}$. If $R^\dag=\hat
R$, the situation is more complicated. First, we need to discuss the
behaviour of our constructions under $\nu$-extensions.

\subsection{Local blowings up and trees.}\label{trees}

We consider \textit{extensions} $R\rightarrow R'$ of local rings, that
is, injective morphisms such that $R'$ is an $R$-algebra essentially
of finite type and $m'\cap R=m$. In this paper we consider only
extensions with respect to $\nu$; that is, both $R$ and $R'$ are
contained in a fixed valuation ring $R_\nu$. Such extensions form a
direct system $\{R'\}$. We will consider many direct systems of rings
and of ideals indexed by $\{R'\}$; direct limits will always be taken
with respect to the direct system $\{R'\}$. Unless otherwise
specified, we will assume that
\begin{equation}
\lim\limits_{\overset\longrightarrow{R'}}R'=R_\nu.\label{eq:ZariskiRiemann}
\end{equation}
Note that by the fundamental properties of valuation rings
(\cite{ZS}, \S VI), assuming the equality (\ref{eq:ZariskiRiemann})
is equivalent to assuming that
$\lim\limits_{\overset\longrightarrow{R'}}K'=K_\nu$, where $K'$ stands
for the field of fractions of $R'$ and $K_\nu$ for that of $R_\nu$, and
that $\lim\limits_{\overset\longrightarrow{R'}}R'$ is a valuation
ring.
\begin{definition} A \textbf{tree} of $R'$-algebras is a direct system
  $\{S'\}$ of rings, indexed by the directed set
$\{R'\}$, where $S'$ is an $R'$-algebra. Note that the maps are not
necessarily injective. A morphism $\{S'\}\to\{T'\}$ of
trees is the datum of a map of $R'$-algebras $S'\to T'$ for each $R'$
commuting with the tree morphisms for each map $R'\to R''$.
\end{definition}
\begin{lemma}\label{factor} Let $R\rightarrow R'$ be an extension of
  local rings. We have:\par
\noindent 1) The ideal
  $N:=m^\dag\otimes_R1+1\otimes_Rm'$ is maximal in the $R$-algebra
  $R^\dag\otimes_RR'$.\par
\noindent 2) The natural map of completions (resp. henselizations)
$R^\dag\to{R'}^\dag$ is injective.
\end{lemma}
\begin{proof} 1) follows from that fact that $R^\dag/m^\dag
=R/m$. The proof of 2) relies on a construction which we shall use
often: the map $R^\dag\to{R'}^\dag$ can be factored as
\begin{equation}
R^\dag\to\left(R^\dag\otimes_RR'\right)_N\to{R'}^\dag,\label{eq:iota}
\end{equation}
where the first map sends $x$ to $x\otimes 1$ and the second is determined
by $x\otimes x'\mapsto \hat b(x).c(x')$ where $\hat b$ is the natural
map $R^\dag\to {R'}^\dag$ and $c$ is the canonical map
$R'\to{R'}^\dag$. The first map is injective because $R^\dag$ is a
flat $R$-algebra and it is obtained by tensoring the injection $R\to
R'$ by the $R$-algebra $R^\dag$; furthermore, elements of $R^\dag$
whose image in $R^\dag\otimes_RR'$ lie outside of $N$ are precisely
units of $R^\dag$, hence they are not zero divisors in
$R^\dag\otimes_RR'$ and $R^\dag$ injects in every localization of
$R^\dag\otimes_RR'$.\par
\noindent Since $m'\cap R=m$, we see that the inverse image by the
natural map of $R'$-algebras
$$
\iota\colon R'\to(R^\dag\otimes_RR')_N,
$$
defined by $x'\mapsto1\otimes_Rx'$, of the maximal ideal
$M=(m^\dag\otimes_R1+1\otimes_Rm')(R^\dag\otimes_RR')_N$ of
$(R^\dag\otimes_RR')_N$ is the ideal $m'$ and that $\iota$ induces a
natural isomorphism
$\frac{R'}{{m'}^i}\overset\sim\rightarrow\frac{(R^\dag\otimes_RR')_N}{M^i}$
for each $i$. From this it follows by the universal properties of
completion and henselization that the second map in the sequence
(\ref{eq:iota}) is the completion (resp. the henselization inside the
completion) of $R^\dag\otimes_RR'$ with respect to the ideal $M$. It
is therefore also injective.
\end{proof}

\begin{definition} Let $\{S'\}$ be a tree of $R'$-algebras. For each $S'$,
  let $I'$ be an ideal of $S'$. We say that $\{I'\}$ is a tree of
  ideals if for any arrow $b_{S'S''}\colon S'\rightarrow S''$ in our
  direct system, we have $b^{-1}_{S'S''}I''=I'$. We have the obvious
  notion of inclusion of trees of ideals. In particular, we may speak
  about chains of trees of ideals.
\end{definition}

\begin{examples}
The maximal ideals of the local rings of our system $\{R'\}$ form a
tree of ideals.

For any non-negative element $\beta\in\Gamma$, the valuation ideals
$\P'_\beta\subset R'$ of value $\beta$ form a tree of ideals of
$\{R'\}$. Similarly, the $i$-th prime valuation ideals $P'_i\subset
R'$ form a tree. If $rk\ \nu=r$, the prime valuation ideals $P'_i$
give rise to a chain
\begin{equation}
P'_0=(0)\subsetneqq P'_1\subseteq\dots\subseteq P'_r=m'\label{eq:treechain'}
\end{equation}
of trees of prime ideals of $\{R'\}$.
\end{examples}

We discuss this last example in a little more detail and
generality in order to emphasize our point of view, crucial
throughout this paper: the data of a composite valuation is
equivalent to the data of its components. Namely, suppose we are
given a chain of trees of ideals as in (\ref{eq:treechain'}),
where we relax our assumptions of the $P'_i$ as follows. We no
longer assume that the chain (\ref{eq:treechain'}) is maximal, nor
that $P'_i\subsetneqq P'_{i+1}$, even for $R'$ sufficiently large;
in particular, for the purposes of this example we momentarily drop
the assumption that $rk\ \nu=r$. We will still assume, however, that
$P'_0=(0)$ and that $P'_r=m'$.

Taking the limit in (\ref{eq:treechain'}), we obtain a chain
\begin{equation}
(0)=\mathbf{m}_0\subsetneqq\mathbf{m}_1\subseteqq\dots\subseteqq
\mathbf{m}_r=\mathbf{m}_\nu\label{eq:treechainlim}
\end{equation}
of prime ideals of the valuation ring $R_\nu$.

Similarly, for each $1\leq\ell\leq r$ one has the equality
$$
\lim\limits_{\overset\longrightarrow{R'}}{\frac{R'}{P'_\ell}}=\frac{R_\nu}{\bf
  m_\ell}.
$$
Then \textbf{specifying the valuation $\nu$ is equivalent to
specifying valuations $\nu_0,\nu_1$, \dots, $\nu_r$, where $\nu_0$ is the
trivial valuation of $K$ and, for $1\le \ell\le r$, $\nu_\ell$ is a
valuation of the residue field
$k_{\nu_{\ell-1}}=\kappa(\mathbf{m}_{\ell-1})$, centered at the local
ring
$\lim\limits_{\longrightarrow}\frac{R'_{P'_\ell}}{P'_{\ell-1}R'_{P'_\ell}}=
\frac{(R_\nu)_{\mathbf{m}_\ell}}{\mathbf{m}_{\ell-1}}$ and taking its
values in the totally ordered group
$\frac{\Delta_{\ell-1}}{\Delta_\ell}$.}

The relationship between $\nu$ and the $\nu_\ell$ is that $\nu$ is the
composition
\begin{equation}
\nu=\nu_1\circ\nu_2\circ\dots\circ\nu_r.\label{eq:composition1}
\end{equation}
For example, the datum of the valuation $\nu$, or of its valuation ring $R_\nu$,
is equivalent to the datum of the valuation ring
$\frac{R_\nu}{\mathbf{m}_{r-1}}\subset
\frac{(R_\nu)_{\mathbf{m}_{r-1}}}{\mathbf{m}_{r-1}(R_\nu)_{\mathbf{m}_{r-1}}}=\kappa
(\mathbf{m}_{r-1})$ of the valuation $\nu_r$ of the field
$\kappa(\mathbf{m}_{r-1})$ and the valuation ring
$(R_\nu)_{\mathbf{m}_{r-1}}$. If we assume, in addition, that for $R$
sufficiently large the chain (\ref{eq:treechain'}) (equivalently,
(\ref{eq:treechainlim})) is a maximal chain of distinct prime ideals
then $rk\ \nu=r$ and $rk\ \nu_\ell=1$ for each $\ell$.\par\medskip
\begin{remark}\label{composite}
Another way to describe the same property of valuations is that, given
a prime ideal $H$ of the local integral domain $R$ one builds all
valuations centered in $R$ having $H$ as one of the $P_\ell$ by
choosing a valuation $\nu_1$ of $R$ centered at $H$, so that
$\mathbf{m}_{\nu_1}\cap R=H$ and choosing a valuation subring
$\overline R_{\overline\nu}$ of the field
$\frac{R_{\nu_1}}{\mathbf{m}_{\nu_1}}$ centered at $R/H$. Then
$\nu=\nu_1\circ \overline\nu$.\par
\noindent Note that choosing a valuation of $R/H$ determines a
valuation of its field of fractions $\kappa(H)$, which is in general
much smaller than $\frac{R_{\nu_1}}{\mathbf{m}_{\nu_1}}$. Given a
valuation of $R$ with center $H$, in order to determine a valuation of
$R$ with center $m$ inducing on $R/H$ a given valuation $\mu$ we must
choose an extension $\overline\nu$ of $\mu$ to
$\frac{R_{\nu_1}}{\mathbf{m}_{\nu_1}}$, and there are in general many
possibilities.\par
This will be used in the sequel. In particular, it will be applied to
the case where a valuation $\nu$ of $R$ extends uniquely to a
valuation $\hat\nu_-$ of $\frac{\hat R}{H}$ for some prime $H$ of $\hat
R$. Assuming that $\hat R$ is an integral domain, this determines a
unique valuation of $\hat R$ only if the height $\he\ H$ of $H$ in
$\hat R$ is at most one. In all other cases the dimension of $\hat
R_H$ is at least $2$ and we have infinitely many valuations with which
to compose $\hat\nu_-$. This is the source of the height conditions we
shall see in \S\ref{extensions}.
\end{remark}

\begin{example} Let $k_0$ be a field and $K=k_0((u,v))$ the
field of fractions of the complete local ring $R=k_0[[u,v]]$. Let
$\Gamma=\mathbf Z^2$ with lexicographical ordering. The isolated
subgroups of $\Gamma$ are $(0)\subsetneqq(0)\oplus\mathbf
Z\subsetneqq\mathbf Z^2$. Consider the valuation
$\nu:K^*\rightarrow\mathbf Z^2$, centered at $R$, given by
\begin{eqnarray}
\nu(v)&=&(0,1)\\
\nu(u)&=&(1,0)\\
\nu(c)&=&0\quad\text{ for any }c\in k_0^*.
\end{eqnarray}
This information determines $\nu$ completely; namely, for any power series
$$
f=\sum\limits_{\alpha,\beta}c_{\alpha\beta}u^\alpha v^\beta \in k_0[[u,v]],
$$
we have
$$
\nu(f)=\min\{(\alpha,\beta)\ |\ c_{\alpha\beta}\ne 0\}.
$$
We have $rk\ \nu=rat.rk\ \nu=2$. Let $\Delta=(0)\oplus\mathbf Z$. Let
$\Gamma_+$ denote the semigroup of all the non-negative elements of
$\Gamma$. Let $k_0[[\Gamma_+]]$ denote the $R$-algebra of power series
$\sum c_{\alpha,\beta}u^\alpha v^\beta$ where $c_{\alpha,\beta}\in
k_0$ and the exponents $(\alpha ,\beta)$ form a well ordered subset of
$\Gamma_+$. By classical results (see \cite{Kap1}, \cite{Kap2}), it is
a valuation ring with maximal ideal generated by all the monomials
$u^\alpha v^\beta$, where $(\alpha,\beta)>(0,0)$ (in other words,
either $\alpha>0,\beta\in\mathbf Z$ or $\alpha=0,\beta>0$).  Then
$$
R_\nu=k_0[[\Gamma_+]]\bigcap k_0((u,v))
$$
is a valuation ring of $K$, and contains $k[[u,v]]$; it is the valuation ring of the valuation
$\nu$. The prime ideal $\mathbf{m}_1$ is the ideal of $R_\nu$
generated by all the $uv^\beta$, $\beta\in\mathbf Z$. The valuation
$\nu_1$ is the discrete rank 1 valuation of $K$ with valuation ring
$$
(R_\nu)_{\mathbf{m}_1}=k_0[[u,v]]_{(u)}
$$
and $\nu_2$ is the discrete rank 1 valuation of $k_0((v))$ with
valuation ring $\frac{R_\nu}{\mathbf{m}_1}\cong k_0[[v]]$.
\end{example}

\begin{example} To give a more interesting example, let $k_0$ be
a field of characteristic zero and
$$
K=k_0(x,y,z)
$$
a purely transcendental extension of $k_0$ of degree 3. Let $w$ be an
independent variable and put $k=\bigcup\limits_{j=1}^\infty
k_0\left(w^{\frac 1j}\right)$. Let $\Gamma=\mathbf Z\oplus\Q$ with the
lexicographical ordering and $\Delta=(0)\oplus\Q$ the non-trivial isolated
subgroup of $\Gamma$. Let $u,v$ be new variables and let
$\mu_1:k((u,v))\rightarrow\Z^2_{lex}$ be the valuation of the
previous example. Let $\mu_2$ denote the $x$-adic valuation of $k$ and
put $\mu=\mu_1\circ\mu_2$.
Consider the map $\iota:k_0[x,y,z]\rightarrow
k[[u,v]]$ which sends $x$ to $w$, $y$ to $v$ and $z$ to
$u-\sum\limits_{j=1}^\infty w^{\frac 1j}v^j$. Let $\nu_1=\left.\mu_1\right|_K$ and
$\nu=\mu|_K$.

The valuation $\nu:K^*\rightarrow\Gamma$ is centered at the local ring
$R=k_0[x,y,z]_{(x,y,z)}$; we have
\begin{eqnarray}
\nu(x)&=&(0,1)\\
\nu(y)&=&(1,0),\\
\nu(z)&=&(1,1).
\end{eqnarray}
Write as a composition of two rank 1 valuations:
$\nu=\nu_1\circ\nu_2$. We have natural inclusions $R_{\nu_1}\subset
R_{\mu_1}$ and $k_{\nu_1}\subset k_{\mu_1}=k$. We claim that
$k_{\nu_1}$ is not finitely generated over $k_0$. Indeed, if this were not the
case then there would exist a prime number $p$ such that
$w^{\frac1p}\notin k_{\nu_1}$. Let
$k'=k_0\left(x^{\frac1{(p-1)!}}\right)$. Let $L= k'(y,z)$. Consier
the tower of field extensions $K\subset L\subset k[[u,v]]$ and let $\nu'$ denote
the restriction of $\mu$ to $L$. Let $\Gamma'$ be the value group of
$\nu'$ and $k_{\nu'}$ the residue field of its valuation ring. Now,
$L$ contains the element $z_p:=z-\sum\limits_{j=1}^{p-1}x^{\frac1j}y^j$ as
well as $\frac{z_p}{y^p}$. We have
\begin{equation}
\nu'(z_p)=\left(p,\frac1p\right),
\end{equation}
$\nu'\left(\frac{z_p}{y^p}\right)=0$ and the natural image of
$\frac{z_p}{y^p}$ in $k_{\mu_1}=k$ is $w^{\frac1p}$. Now,
$p\not|\ [L:K]$, $\left.[\Gamma':\Gamma]\ \right|\
[L:K]$ and $\left.[k_{\nu'}:k_\nu]\ \right|\
[L:K]$. This implies that $z_p\in L$ and $w^{\frac1p}\in k_{\nu_1}$,
which gives the desired contradicion.

It is not hard to show that for each $j$, there exists a local blowing
up $R\rightarrow R'$ of $R$ such that, in the notation of
(\ref{eq:treechain'}), we have
$\kappa(P'_1)=k_0\left(w^{\frac1{j!}}\right)$ and that
$\kappa(\mathbf{m}_1)=\lim\limits_{j\to\infty}\kappa(P'_1)=k$. The
first one is the blowing up of the ideal $(y,z)R$, localized at the
point $y=0, z/y=x$. Then one blows up the ideal $(z/y-x, y)$, and so
on.

Another way to see the valuation $\nu=\nu_1\circ\nu_2$ is to note that
$\nu_1$ is the restriction to $K$ of the $v$-adic valuation under the
inclusion of fields deduced from the inclusion of rings
$$
k_0[[x,y,z]]_{(y,z)}\hookrightarrow k\left[\left[v^{{\mathbf Z}_+}\right]\right]
$$
which sends $x$ to $w$, $y$ to $v$ and $z$ to $\sum\limits_{j=1}^\infty
w^{\frac1j}v^j$. Recall that the ring on the right is made of power
series with non negative rational exponents whose set of exponents is
well ordered. We have $k_{\nu_1}=k$.
\end{example}
\begin{remark} The point of the last example is to show that, given a
  composed valuation as in (\ref{eq:composition1}), $\nu_\ell$ is a
  valuation of the field $k_{\nu_{\ell-1}}$, which may
  \textbf{properly} contain $\kappa(P'_{\ell-1})$ for \textbf{every}
  $R'\in\mathcal{T}$. This fact will be a source of complication
  later on and we prefer to draw attention to it from the beginning.
\end{remark}
Coming back to the implicit prime ideals, we will see that the
implicit prime ideals $H'_i$ form a tree of ideals of $R^\dag$.

We will show that if $\nu$ extends to a valuation of $\hat\nu_-$ centered at
$\frac{\hat R}P$ with $P\cap R=(0)$ then the prime $P$ must
contain the minimal prime $H_0$ of $\hat R$. We will then show
that specifying an extension $\hat\nu_-$ of $\nu$ as above is
equivalent to specifying a chain of prime valuation ideals
\begin{equation}
\tilde H'_0\subset\tilde H'_1\subset\dots\subset\tilde H'_{2r}=m'\hat
R'\label{eq:chaintree}
\end{equation}
of $\hat R'$ such that $H'_\ell\subset\tilde H'_\ell$ for all
$\ell\in\{0,\dots,2r\}$, and valuations
$\hat\nu_1,\hat\nu_2,\dots,\hat\nu_{2r}$, where $\hat\nu_i$ is a
valuation of the field $k_{\hat\nu_{i-1}}$ (the residue field of the
valuation ring $R_{\hat\nu_{i-1}}$), arbitrary when $i$ is odd and
satisfying certain conditions, coming from the valuation $\nu_{\frac
  i2}$, when $i$ is even.

The prime ideals $H_i$ are defined as follows.\par
\noindent Recall that given a valued ring $(R,\nu)$, that is a subring $R\subseteq R_\nu$ of the valuation ring
$R_\nu$ of a valuation with value group $\Gamma$, one defines for each
$\beta\in \Gamma$ the valuation ideals of $R$ associated to $\beta$:
$$
\begin{array}{lr} {\cal P}_\beta (R)=&\{x\in R/\nu(x)\geq\beta\}\cr
{\cal P}^+_\beta (R)=&\{x\in R/\nu(x)>\beta\}\end{array}
$$
and the associated graded ring
$$
\hbox{\rm gr}_\nu R=\bigoplus_{\beta\in \Gamma}\frac{{\cal P}_\beta (R)}
{{\cal P}^+_\beta (R)}=\bigoplus_{\beta\in
\Gamma_+}\frac{{\cal P}_\beta (R)}{{\cal P}^+_\beta (R)}.
$$
The second equality comes from the fact that if
$\beta\in\Gamma_-\setminus\{0\}$, we have ${\cal P}^+_\beta (R)={\cal
  P}_\beta (R)=R$. If $R\to R'$ is an extension of local rings such
that $R\subset R'\subset R_\nu$ and $m_\nu\cap R'=m'$, we will write
${\cal P}'_\beta$ for ${\cal P}_\beta(R')$.

Fix a valuation ring $R_\nu$ dominating $R$, and a tree ${\cal
  T}=\{R'\}$ of n\oe therian local $R$-subalgebras of $R_\nu$, having
the following properties: for each ring $R'\in\cal{T}$, all the
birational $\nu$-extensions of $R'$ belong to $\cal{T}$. Moreover, we
assume that the field of fractions of $R_\nu$ equals
$\lim\limits_{\overset\longrightarrow{R'}}K'$, where $K'$ is the field
of fractions of $R'$. The tree $\cal{T}$ will stay constant throughout
this paper. In the special case when $R$ happens to have the same
field of fractions as $R_\nu$, we may take $\cal{T}$ to be the tree of
all the birational $\nu$-extensions of $R$.

\begin{notation} For a ring $R'\in\cal T$, we shall denote by
${\cal T}(R')$ the subtree of $\cal T$ consisting of all the
$\nu$-extensions $R''$ of $R'$.
\end{notation}

We now define
\begin{equation}
H_{2\ell+1}=\bigcap\limits_{\beta\in\Delta_{\ell}}
\left(\left(\lim\limits_{\overset\longrightarrow{R'}}{\cal P}'_\beta
    {R'}^\dag\right)\bigcap R^\dag\right),\ 0\le\ell\le
r-1\label{eq:defin}
\end{equation}
(in the beginning of \S\ref{basics} we provide some motivation for this
definition and give several elementary examples of $H'_i$ and $\tilde H'_i$).

The questions answered in this paper originally arose from our work on the
Local Uniformization Theorem, where passage to completion is required in
both the approaches of \cite{Spi2} and \cite{Te}. In \cite{Te}, one really
needs to pass to completion for valuations of arbitrary rank. One of the
main intended applications of the theory of implicit prime ideals is
the following conjecture. Let
\begin{equation}
\Gamma\hookrightarrow\hat\Gamma\label{eq:extGamma}
\end{equation}
be an extension of ordered groups of the same rank. Let
\begin{equation}
(0)=\Delta_r\subsetneqq\Delta_{r-1}\subsetneqq\dots\subsetneqq\Delta_0=\Gamma
\label{eq:isolated}
\end{equation}
be the isolated subgroups of $\Gamma$ and
$$
(0)=\hat\Delta_r\subsetneqq\hat\Delta_{r-1}\subsetneqq\dots\subsetneqq
\hat\Delta_0=\hat\Gamma
$$
the isolated subgroups of $\hat\Gamma$, so that the inclusion
(\ref{eq:extGamma}) induces inclusions
\begin{eqnarray}
\Delta_\ell&\hookrightarrow&\hat\Delta_\ell\quad\text{ and}\\
\frac{\Delta_\ell}{\Delta_{\ell+1}}&\hookrightarrow&
\frac{\hat\Delta_\ell}{\hat\Delta_{\ell+1}}.
\end{eqnarray}
Let $G\hookrightarrow\hat G$ be an extension of graded algebras
without zero divisors, such that $G$ is graded by $\Gamma_+$ and $\hat
G$ by $\hat\Gamma_+$. The graded algebra $G$ is endowed with a natural
valuation with value group $\Gamma$ and similarly for $\hat G$ and
$\hat\Gamma$. These natural valuations will both be denoted by $ord$.
\begin{definition} We say that the extension $G\hookrightarrow\hat G$ is
  \textbf{scalewise birational} if for any $x\in\hat G$ and
  $\ell\in\{1,\dots,r\}$ such that $ord\ x\in\hat\Delta_\ell$ there
  exists $y\in G$ such that $ord\ y\in\Delta_\ell$ and $xy\in G$.
\end{definition}
Of course, scalewise birational implies birational and also that
$\hat\Gamma=\Gamma$.\par
\noindent While the main result of this paper is the primality
of the implicit ideals associated to a valuation, and the subsequent
description of the extensions of the valuation to the completion, the
main conjecture stated here is the following:\par
\begin{conjecture}\label{teissier} Assume that $\dim\ R'=\dim\ R$ for all
  $R'\in\mathcal{T}$.  Then there exists a tree of prime ideals $H'$
  of $\hat R'$ with $H'\cap R'=(0)$ and a valuation $\hat\nu_-$,
  centered at $\lim\limits_\to\frac{\hat R'}{H'}$ and having the
  following property:\par\noindent For any $R'\in\cal{T}$ the graded
  algebra $\gr_{\hat\nu_-}\frac{\hat R'}{H'}$ is a scalewise birational
  extension of $\gr_\nu R'$.
\end{conjecture}
The example given in remark 5.20, 4) of \cite{Te} shows that the
morphism of associated graded rings is not an isomorphism in general.

The approach to the Local Uniformization Theorem taken in \cite{Spi2} is to
reduce the problem to the case of rank 1 valuations. The theory of implicit
prime ideals is much simpler for valuations of rank 1 and takes only a
few pages in Section \ref{archimedian}.

The paper is organized as follows. In \S\ref{basics} we define the
odd-numbered implicit ideals $H_{2\ell+1}$ and prove that $H_{2\ell+1}\cap
R=P_\ell$. We observe that by their very definition, the ideals
$H_{2\ell+1}$ behave well under $\nu$-extensions; they form a
tree. Proving that $H_{2\ell+1}$ is indeed prime is postponed until
later sections; it will be proved gradually in
\S\ref{technical}--\S\ref{prime}. In the beginning of \S\ref{basics}
we will explain in more detail the respective roles played by the
odd-numbered and the even-numbered implicit ideals, give several
examples (among other things, to motivate the need for taking the
limit with respect to $R'$ in (\ref{eq:defin})) and say one or two
words about the techniques used to prove our results.

In \S\ref{technical} we prove the primality of the implicit prime
ideals assuming a certain technical condition, called
\textbf{stability}, about the tree $\cal T$ and the operation ${\
}^\dag$. It follows from the noetherianity of $R^\dag$ that there
exists a specific $R'$ for which the limit in (\ref{eq:defin}) is
attained. One of the main points of \S\ref{technical} is to prove
properties of stable rings which guarantee that this limit is attained
whenever $R'$ is stable. We then use the excellence of $R$ to define
the even-numbered implicit prime ideals: for $i=2\ell$ the ideal
$H_{2\ell}$ is defined to be the unique minimal prime of $P_\ell
R^\dag$, contained in $H_{2\ell+1}$ (in the case $R^\dag=\hat R$ it is
the excellence of $R$ which implies the uniqueness of such a minimal
prime). We have
$$
H_{2\ell}\cap R=P_\ell
$$
for $\ell\in\{0,\dots,r\}$. The results of \S\ref{technical} apply
equally well to completions, henselizations and other local \'etale
extensions; to complete the proof of the primality of the implicit
ideals in various contexts such as henselization or completion, it
remains to show the existence of stable $\nu$-extensions in the
corresponding context.

In \S\ref{Rdag} we describe the set of extensions $\nu^\dag_-$ of $\nu$ to
$\lim\limits_{\overset\longrightarrow{R'}}\frac{{R'}^\dag}{P'{R'}^\dag}$,
where $P'$ is a tree of prime ideals of ${R'}^\dag$ such that $P'\cap
R'=(0)$. We show (Theorem \ref{classification}) that specifying such a
valuation $\nu^\dag_-$ is equivalent to specifying the following data:

(1) a chain (\ref{eq:chaintree}) of trees of prime ideals $\tilde
H'_i$ of ${R'}^\dag$ (where $\tilde H'_0=P'$), such that
$H'_i\subset\tilde H'_i$ for each $i$ and each $R'\in\mathcal{T}$,
satisfying one additional condition (we will refer to the chain
(\ref{eq:chaintree}) as the chain of trees of ideals,
\textbf{determined by} the extension $\nu^\dag_-$)

(2) a valuation $\nu^\dag_i$ of the residue field $k_{\nu^\dag_{i-1}}$
of $\nu^\dag_{i-1}$, whose restriction to the field
$\lim\limits_{\overset\longrightarrow{R'}}\kappa(\tilde H'_{i-1})$ is
centered at the local ring
$\lim\limits_{\overset\longrightarrow{R'}}\frac{{R'}^\dag_{\tilde
    H'_i}}{\tilde H'_{i-1}{R'}^\dag_{\tilde H'_i}}$.

If $i=2\ell$ is even, the valuation $\nu^\dag_i$ must be of rank 1 and
its restriction to $\kappa(\mathbf{m}_{\ell-1})$ must coincide with $\nu_\ell$.

Notice the recursive nature of this description of $\nu^\dag_-$: in
order to describe $\nu^\dag_i$ we must know $\nu^\dag_{i-1}$ in order
to talk about its residue field $k_{\nu^\dag_{i-1}}$.

In \S\ref{extensions} we address the question of uniqueness of
$\nu^\dag_-$. We describe several classes of extensions $\nu^\dag_-$ which
are particularly useful for the applications: \textbf{minimal} and
\textbf{evenly minimal} extensions, and also those $\nu^\dag_-$ for
which, denoting by $\he\ I$ the height of an ideal, we have
\begin{equation}
\he\ \tilde H'_{2\ell+1}- \he\ \tilde H'_{2_\ell} \le 1\quad\text{ for
}0\le\ell\le r;\label{eq:odd=even4}
\end{equation}
in fact, the special case of (\ref{eq:odd=even4}) which is of most
interest for the applications is
\begin{equation}
\tilde H'_{2\ell}=\tilde H'_{2\ell+1}\quad\text{ for }1\le\ell\le
r.\label{eq:odd=even}
\end{equation}
We prove some necessary and some sufficient conditions under which an
extension $\nu^\dag_-$ whose corresponding ideals $\tilde H'_i$ satisfy
(\ref{eq:odd=even}) is uniquely determined by the ideals $\tilde
H'_i$. We also give sufficient conditions for the graded algebra
$gr_\nu R'$ to be scalewise birational to $gr_{\hat\nu_-}\hat R'$ for
each $R'\in\cal{T}$. These sufficient conditions are used in
\S\ref{locuni1} to prove some partial results towards Conjecture \ref{teissier}.

In \S\ref{henselization} we show the existence of $\nu$-extensions in
$\cal T$, stable for henselization, thus reducing the proof of the
primality of $H_{2\ell+1}$ to the results of \S\ref{technical}. We
study the extension of $\nu$ to $\tilde R$ modulo its first  prime
ideal and prove that such an extension is unique.

In \S\ref{prime} we use the results of \S\ref{henselization} to prove
the existence of $\nu$-extensions in $\cal T$, stable for
completion. Combined with the results of \S\ref{technical} this proves
that the ideals $H_{2\ell+1}$ are prime.

In \S\ref{locuni1} we describe a possible
approach and prove some partial results towards constructing a
chain of trees (\ref{eq:chaintree}) of prime ideals of $\hat R'$
satisfying (\ref{eq:odd=even}) and a corresponding valuation $\hat\nu_-$
which satisfies the conclusion of Conjecture \ref{teissier}. We also
prove a necessary and a sufficient condition for the uniqueness of $\hat\nu_-$, assuming Conjecture \ref{teissier}.

We would like to acknowledge the paper \cite{HeSa} by Bill Heinzer and Judith
Sally which inspired one of the authors to continue thinking about this subject, as well as the work of S.D. Cutkosky, S. El Hitti and L. Ghezzi: \cite{CG} (which contains results closely related to those of \S\ref{archimedian}) and \cite{CE}.

\section{Extending a valuation of rank one centered in a local domain
to its formal completion.}
\label{archimedian}

Let $(R,M,k)$ be a local noetherian domain, $K$ its field of
fractions, and $\nu:K\rightarrow\g_+\cup\{\infty\}$ a rank one valuation,
centered at $R$ (that is, non-negative on $R$ and positive
on $M$).\par Let $\hat R$ denote the formal completion of $R$. It is
convenient to extend $\nu$ to a valuation centered at $\frac{\hat R}H$, where
$H$ is a prime ideal of $\hat R$ such that $H\cap R=(0)$. In this
section, we will assume that $\nu$ is of rank one, so that the value
group $\g$ is archimedian. We will explicitly describe a prime ideal
$H$ of $\hat R$, canonically associated to $\nu$, such that $H\cap
R=(0)$ and such that $\nu$ has a unique extension $\hat\nu_-$ to $\frac{\hat R}H$.

Let $\h=\nu(R\setminus (0))$, let
$\P_\beta$ denote the $\nu$-ideal of $R$ of value $\beta$ and $\P_{\beta}^+$ the
greatest $\nu$-ideal, properly contained in $\P_\beta$. We now
define the main object of study of this section. Let
\begin{equation}
H:=\bigcap\limits_{\beta\in\h}(\P_\beta\hat R).\label{tag51}
\end{equation}

\begin{remark}\label{Remark51} Since $R$ is noetherian, we have $\nu
  (M)>0$ and since the ordered group $\Gamma$ is archimedian, for
  every $\beta\in\h$ there exists $n\in\mathbf N$ such that
  $M^n\subset \P_\beta$. In other words, the $M$-adic topology on $R$
  is finer than (or equal to) the $\nu$-adic topology. Therefore an
  element $x\in\hat R$ lies in $\P_\beta\hat R\iff$ there exists a
  Cauchy sequence $\{x_n\}\subset R$ in the $M$-adic topology,
  converging to $x$, such that $\nu(x_n)\ge\beta$ for all $n\iff$ for
  {\it every} Cauchy sequence $\{x_n\}\subset R$, converging to $x$,
  $\nu(x_n)\ge\beta$ for all $n\gg0$. By the same token, $x\in H\iff$
  there exists a Cauchy sequence $\{x_n\}\subset R$, converging to
  $x$, such that $\lim\limits_{n\to\infty}\nu(x_n)=\infty\iff$ for
  {\it every} Cauchy sequence $\{x_n\}\subset R$, converging to $x$,
  $\lim\limits_{n\to\infty}\nu(x_n)=\infty$.
\end{remark}

\begin{example} Let $R=k[u,v]_{(u,v)}$. Then $\hat
R=k[[u,v]]$. Consider an element $w=u-\sum\limits_{i=1}^\infty
c_iv^i\in\hat R$, where $c_i\in k^*$ for all $i\in\mathbf N$, such
that $w$ is transcendental over $k(u,v)$. Consider the injective map
$\iota:k[u,v]_{(u,v)}\rightarrow k[[t]]$ which sends $v$ to $t$ and
$u$ to $\sum\limits_{i=1}^\infty c_it^i$. Let $\nu$ be the valuation
induced from the $t$-adic valuation of $k[[t]]$ via $\iota$. The value
group of $\nu$ is $\mathbf Z$ and $\h=\mathbf N_0$. For each
$\beta\in\mathbf N$,
$\P_\beta=\left(v^\beta,u-\sum\limits_{i=1}^{\beta-1}c_iv^i\right)$. Thus
$H=(w)$.

We come back to the general theory. Since the formal completion
homomorphism $R\rightarrow\hat R$ is faithfully flat,
\begin{equation}
\P_\beta\hat R\cap R=\P_\beta\quad\text{for all }\beta\in\h.\label{tag52}
\end{equation}
Taking the intersection over all $\beta\in\h$, we obtain
\begin{equation}
H\cap R=\left(\bigcap\limits_{\beta\in\h}\left(\P_\beta\hat R\right)\right)\cap
R=\bigcap\limits_{\beta\in\h}\P_\beta=(0),\label{tag53}
\end{equation}

In other words, we have a natural inclusion $R\hookrightarrow\frac{\hat R}H$.
\end{example}
\begin{theorem}\label{th53}
\begin{enumerate}
\item $H$ is a prime ideal of $\hat R$.
\item $\nu$ extends uniquely to a valuation $\hat\nu_-$, centered at
$\frac{\hat R}H$.
\end{enumerate}
\end{theorem}

\begin{proof} Let $\bar x\in\frac{\hat R}H\setminus\{0\}$. Pick a
representative $x$ of $\bar x$ in $\hat R$, so that $\bar x= x\ \md\
H$. Since $x\notin H$, we have $ x\notin \P_\alpha\hat R$ for some
$\alpha\in\h$.
\begin{lemma}\label{lemma36} {\rm (See \cite{ZS}, Appendix 5, lemma 3)} Let $\nu$ be a valuation of rank one centered in a
local noetherian domain $(R,M,k)$. Let
$$
\h=\nu(R\setminus (0))\subset\g.
$$
Then $\h$ contains no infinite bounded sequences.
\end{lemma}

\begin{proof} An infinite ascending sequence $\alpha_1<\alpha_2<\dots$ in
$\h$, bounded above by an element $\beta\in\h$, would give rise to an
infinite descending chain of ideals in $\frac R{\P_\beta}$. Thus it is
sufficient to prove that $\frac R{\P_\beta}$ has finite length.

Let $\delta:=\nu(M)\equiv\min(\h\setminus\{0\})$. Since $\h$ is
archimedian, there exists $n\in\mathbf N$ such that $\beta\le n\delta$.
Then $M^n\subset \P_\beta$, so that there is a surjective map $\frac
R{M^n}\twoheadrightarrow\frac R{\P_\beta}$. Thus $\frac R{\P_\beta}$ has
finite length, as desired.
\end{proof}

By Lemma \ref{lemma36}, the set $\{\beta\in\h\ |\ \beta<\alpha\}$
is finite. Hence there exists a unique $\beta\in\h$ such that
\begin{equation}
x\in \P_\beta\hat R\setminus \P_{\beta}^+\hat R.\label{tag54}
\end{equation}
Note that $\beta$ depends only on $\bar x$, but not on the choice of the
representative $ x$. Define the function $\hat\nu_-:\frac{\hat
R}H\setminus\{0\}\rightarrow\h$ by
\begin{equation}
\hat\nu_-(\bar x)=\beta.\label{tag55}
\end{equation}
By (\ref{tag52}), if $x\in R\setminus\{0\}$ then
\begin{equation}
\hat\nu_-(x)=\nu(x).
\end{equation}
It is obvious that
\begin{equation}
\hat\nu_-(x+y)\ge\min\{\hat\nu_-(x),\hat\nu_-(y)\}\label{tag57}
\end{equation}
\begin{equation}
\hat\nu_-(xy)\ge\hat\nu_-(x)+\hat\nu_-(y)\label{tag58}
\end{equation}
for all $x,y\in\frac{\hat R}H$. The point of the next lemma is to
show that $\frac{\hat R}H$ is a domain and that $\hat\nu_-$ is, in fact, a
valuation (i.e. that the inequality (\ref{tag58}) is, in fact, an equality).

\begin{lemma}\label{lemma54} For any non-zero $\bar x,\bar
  y\in\frac{\hat R}H$, we have $\bar x\bar y\ne0$ and $\hat\nu_-(\bar
  x\bar y)=\hat\nu_-(\bar x)+\hat\nu_-(\bar y)$.
\end{lemma}

\begin{proof} Let $\alpha=\hat\nu_-(\bar x)$, $\beta=\hat\nu_-(\bar
y)$. Let $ x$ and $ y$ be representatives in $\hat R$ of $\bar x$ and
$\bar y$, respectively. We have $M\P_\alpha\subset \P_{\alpha}^+$, so that
\begin{equation}
\frac{\P_\alpha}{\P_{\alpha}^+}\cong\frac{\P_\alpha}{\P_{\alpha}^++M\P_\alpha}\cong
\frac{\P_\alpha}{\P_{\alpha}^+}\otimes_Rk\cong\frac{\P_\alpha}{\P_{\alpha}^+}
\otimes_R\frac{\hat R}{M\hat R}\cong\frac{\P_\alpha\hat
R}{(\P_{\alpha}^++M\P_\alpha)\hat R}\cong\frac{\P_\alpha\hat
R}{\P_{\alpha}^+\hat R},\label{tag59}
\end{equation}
and similarly for $\beta$. By (\ref{tag59}) there exist $z\in \P_\alpha$, $w\in
\P_\beta$, such that $z\equiv x\ \md\ \P_{\alpha}^+\hat R$ and
$w\equiv y\ \md\ \P_{\beta}^+\hat R$. Then
\begin{equation}
 xy\equiv zw\ \md\ \P_{\alpha+\beta}^+\hat R.\label{tag510}
\end{equation}
Since $\nu$ is a valuation, $\nu(zw)=\alpha+\beta$, so that $zw\in
\P_{\alpha+\beta}\setminus \P_{\alpha+\beta}^+$. By (\ref{tag52}) and
(\ref{tag510}), this proves that $xy\in \P_{\alpha+\beta}\hat R\setminus
\P_{\alpha+\beta}^+\hat R$. Thus $xy\notin H$ (hence $\bar x\bar
y\ne0$ in $\frac{\hat R}H$) and $\hat\nu_-(\bar x\bar y)=\alpha+\beta$,
as desired.
\end{proof}

By Lemma \ref{lemma54}, $H$ is a prime ideal of $\hat R$. By
(\ref{tag57}) and Lemma \ref{lemma54}, $\hat\nu_-$ is a valuation,
centered at $\frac{\hat R}H$. To complete the proof of Theorem
\ref{th53}, it remains to prove the uniqueness of $\hat\nu_-$. Let $x$,
$\bar x$, the element $\alpha\in\h$ and
\begin{equation}
z\in \P_\alpha\setminus \P_{\alpha}^+\label{tag511}
\end{equation}
be as in the proof of Lemma \ref{lemma54}. Then there exist
\begin{equation}
\begin{array}{rl}
u_1,\ldots , u_n&\in \P_{\alpha}^+\text{ and}\\
 v_1,\ldots ,v_n&\in\hat R
\end{array}\label{tag512}
\end{equation}
such that $ x=z+\sum\limits_{i=1}^nu_i v_i$. Letting $\bar v_i:=v_i\
\md\ H$, we obtain $\bar x=\bar z+\sum\limits_{i=1}^n\bar u_i\bar
v_i$ in $\frac{\hat R}H$. Therefore, by
(\ref{tag511})--(\ref{tag512}), for any extension of $\nu$ to a
valuation $\hat\nu '_-$, centered at $\frac{\hat R}H$, we have
\begin{equation}
\hat\nu '_-(\bar x)=\alpha=\hat\nu_-(\bar x),\label{tag513}
\end{equation}
as desired. This completes the proof of Theorem \ref{th53}.
\end{proof}

\begin{definition}\label{deft55} The ideal $H$ is called the {\bf implicit prime
ideal} of $\hat R$, associated to $\nu$. When dealing with more than one
ring at a time, we will sometimes write $H(R,\nu)$ for $H$.
\end{definition}

More generally, let $\nu$ be a valuation centered at $R$, not
necessarily of rank one. In any case, we may write $\nu$ as a
composition $\nu=\mu_2\circ\mu_1$, where $\mu_2$ is centered at a
non-maximal prime ideal $P$ of $R$ and $\mu_1\left|_{\frac RP}\right.$
is of rank one. The valuation $\mu_1\left|_{\frac RP}\right.$ is
centered at $\frac RP$. We define the {\bf implicit prime ideal} of
$R$ with respect to $\nu$, denoted $H(R,\nu)$, to be the inverse image
in $\hat R$ of the implicit prime ideal of $\frac{\hat R}P$ with
respect to $\mu_1\left|_{\frac RP}\right.$. For the rest of this
section, we will continue to assume that $\nu$ is of rank one.

\begin{remark}\label{Remark56} By (\ref{tag59}), we have the following natural
isomorphisms of graded algebras:
$$
\begin{array}{rl}
\gr_\nu R&\cong\gr_{\hat\nu_-}\frac{\hat R}H\\
G_\nu&\cong G_{\hat\nu_-}.
\end{array}
$$
\end{remark}

\par\medskip
We will now study the behaviour of $H$ under local blowings up of $R$
with respect to $\nu$ and, more generally, under local homomorphisms. Let
$\pi:(R,M)\rightarrow(R',M')$ be a local homomorphism of local noetherian
domains. Assume that $\nu$ extends to a rank one valuation
$\nu':R'\setminus\{0\}\rightarrow\g'$, where $\g'\supset\g$. The
homomorphism $\pi$ induces a local homomorphism $\hat\pi:\hat R\rightarrow\hat
R'$ of formal completions. Let $\h'=\nu'(R'\setminus\{0\})$. For
$\beta\in\h'$, let $\P'_\beta$ denote the $\nu'$-ideal of $R_{\nu'}$
of value $\beta$, as above. Let $H'=H(R',\nu')$.

\begin{lemma}\label{lemma58} Let $\beta\in\h$. Then
\begin{equation}
\left(\P'_\beta\hat R'\right)\cap\hat R=\P_\beta\hat R.\label{tag516}
\end{equation}
\end{lemma}

\begin{proof} Since by assumption $\nu'$ extends $\nu$ we have
$\P'_\beta\cap R=\P_\beta$ and the inclusion
\begin{equation}
\left(\P'_\beta\hat R'\right)\cap\hat R\supseteq \P_\beta\hat R.\label{tag517}
\end{equation}
We will now prove the opposite inclusion. Take an element
$x\in\left(\P'_\beta\hat R'\right)\cap\hat R$. Let $\{x_n\}\subset R$
be a Cauchy sequence in the $M$-adic topology, converging to $x$. Then
$\{\pi(x_n)\}$ converge to $\hat\pi(x)$ in the $M'$-adic topology of
$\hat R'$. Applying remark \ref{Remark51} to $R'$, we obtain
\begin{equation}
\nu(x_n)\equiv\nu'(\pi(x_n))\ge\beta\quad\text{for }n\gg0.\label{tag518}
\end{equation}
By (\ref{tag518}) and Remark \ref{Remark56}, applied to $R$, we have
$x\in \P_\beta\hat R$. This proves the opposite inclusion in
(\ref{tag517}), as desired.
\end{proof}

\begin{corollary}\label{Corollary59} We have
$$
H'\cap\hat R=H.
$$
\end{corollary}

\begin{proof} Since $\nu'$ is of rank one, $\h$ is cofinal in $\h'$. Now
the Corollary follows by taking the intersection over all $\beta\in\h$
in (\ref{tag516}).
\end{proof}

Let $J$ be a non-zero ideal of $R$ and let $R\rightarrow R'$ be the
local blowing up along $J$ with respect to $\nu$. Take an element
$f\in J$, such that $\nu(f)=\nu(J)$. By the {\bf strict transform} of
$J$ in $\hat R'$ we will mean the ideal
$$
J^{\text{str}}:=\bigcup\limits_{i=1}^\infty\left(\left(J\hat
R'\right):f^i\right)\equiv\left(J\hat R'_f\right)\cap\hat R'.
$$
If $g$ is another element of $J$ such that $\nu(g)=\nu(J)$ then
$\nu\left(\frac fg\right)=0$, so that $\frac fg$ is a unit in
$R'$. Thus the definition of strict transform is independent of the
choice of $f$.

\begin{corollary}\label{Corollary510} $H^{\text{str}}\subset H'$.
\end{corollary}

\begin{proof} Since $H\hat R'\subset H'$, we have
$H^{\text{str}}=\left(H\hat R'_f\right)\cap\hat R'\subset\left(H'\hat
  R'_f\right)\cap\hat R'=H'$, where the last equality holds because
$H'$ is a prime ideal of $\hat R'$, not containing $f$.
\end{proof}

Using Zariski's Main Theorem, it can be proved that $H^{\text{str}}$
is prime. Since this fact is not used in the sequel, we omit the proof.
\begin{corollary}\label{Corollary511} Let the notation and assumptions be as in
corollary \ref{Corollary510}. Then
\begin{equation}
\he\  H'\ge\he\  H.\label{tag519}
\end{equation}
In particular,
\begin{equation}
\dim\frac{\hat R'}{H'}\le\dim\frac{\hat R}H.\label{tag520}
\end{equation}
\end{corollary}

\begin{proof} Let $\bar R:=\left(\hat
  R\otimes_RR'\right)_{M'\hat R'\cap(\hat R\otimes_RR')}$. Let $\phi$
denote the natural local homomorphism
$$
\bar R\rightarrow\hat R'.
$$
Let $\bar H:=H'\cap\bar R$. Now, take $f\in J$ such that
$\nu(f)=\nu(J)$. Then $f\notin H'$ and, in particular, $f\notin\bar
H$. Since $R'_f\cong R_f$, we have $\hat R_f=\bar R_f$. In view of
Corollary \ref{Corollary59}, we obtain $H\hat R_f\cong\bar H\bar R_f$,
so
\begin{equation}
\he \ H=\he\ \bar H.\label{tag521}
\end{equation}
Now, $\bar R$ is a local noetherian ring, whose formal completion is
$\hat R'$. Hence $\phi$ is faithfully flat and therefore satisfies the going
down theorem. Thus we have $\he\ H'\ge\he\ \bar H$. Combined with
(\ref{tag521}), this proves (\ref{tag519}). As for the last statement
of the Corollary, it follows from the well known fact that dimension
does not increase under blowing up (\cite{Spi1}, Lemma 2.2): we have
$\dim\ R'\le\dim\ R$, hence
$$
\dim\ \hat R'=\dim R'\le\dim\ R=\dim\ \hat R,
$$
and (\ref{tag520}) follows from (\ref{tag519}) and from the fact that
complete local rings are catenarian.
\end{proof}

It may well happen that the containment of corollary \ref{Corollary510} and the inequality
in (\ref{tag519}) are strict. The possibility of strict containement in corollary \ref{Corollary510}
is related to the existence of subanalytic functions, which are not
analytic. We illustrate this statement by an example in which
$H^{\text{str}}\subsetneqq H'$ and $\he\ H<\he\ H'$.

\begin{example} Let $k$ be a field and let
$$
\begin{array}{rl}
R&=k[x,y,z]_{(x,y,z)},\\
R'&=k[x',y',z']_{(x',y',z')},
\end{array}
$$
where $x'=x$, $y'=\frac yx$ and $z'=z$. We have $K=k(x,y,z)$, $\hat
R=k[[x,y,z]]$, $\hat R'=k[[x',y',z']]$. Let $t_1,t_2$ be auxiliary variables
and let $\sum\limits_{i=1}^\infty c_it_1^i$ (with $c_i\in k$) be an element of
$k[[t_1]]$, transcendental over $k(t_1)$. Let $\theta$ denote the valuation,
centered at $k[[t_1,t_2]]$, defined by $\theta(t_1)=1$, $\theta(t_2)=\sqrt2$
(the value group of $\theta$ is the additive subgroup of $\mathbf R$, generated
by 1 and $\sqrt2$). Let $\iota:R'\hookrightarrow k[[t_1,t_2]]$ denote the
injective map defined by $\iota(x')=t_2$, $\iota(y')=t_1$,
$\iota(z')=\sum\limits_{i=1}^\infty c_it_1^i$. Let $\nu$ denote the
restriction of $\theta$ to $K$, where we view $K$ as a subfield of
$k((t_1,t_2))$ via $\iota$. Let $\h=\nu(R\setminus\{0\})$;
$\h'=\nu(R'\setminus\{0\})$. For $\beta\in\h'$, $P'_\beta$ is generated by all
the monomials of the form ${x'}^\alpha{y'}^\gamma$ such that
$\sqrt2\alpha+\gamma\ge\beta$, together with
$z'-\sum\limits_{j=1}^ic_j{y'}^j$, where $i$ is the greatest non-negative
integer such that $i<\beta$.

Let $w':=z'-\sum\limits_{i=1}^\infty c_i{y'}^i$. Then $H'=(w')$, but
$H=H'\cap\hat R=(0)$, so that $H^{\text{str}}=(0)\subsetneqq H'$ and
$\he\ H=0<1=\he\ H'$.
\end{example}
Recall the following basic result of the theory of G-rings:
\begin{proposition}\label{Corollary57} Assume that $R$ is a reduced G-ring. Then $\hat
R_H$ is a regular local ring.
\end{proposition}

\begin{proof}  Let $K=R_{\P_\infty}=\kappa(\P_\infty)$ (here we are using that $R$ is reduced and that $\P_\infty$ is a minimal
prime of $R$). By definition of G-ring, the map
$R\rightarrow\hat R$ is a regular homomorphism. Then by (\ref{tag53}) $\hat
R_H$ is geometrically regular over $K$, hence regular.
\end{proof}
\begin{remark} Having extended in a unique manner the valuation $\nu$ to
  a valuation $\hat\nu_-$ of $\frac{\hat R}{H}$, we see that if $R$ is
  a G-ring, by Proposition \ref{Corollary57} there is a unique minimal
  prime $\hat\P_\infty$ of $\hat R$ contained in $H$, corresponding to
  the ideal $(0)$ in $\hat R_H$. Since $H\cap R=(0)$, we have the
  equality $\hat\P_\infty\cap R=(0)$. Choosing a valuation $\mu$ of
  the fraction field of $\frac{\hat R_H}{\hat\P_\infty\hat R_H}$
  centered at $\frac{\hat R_H}{\hat\P_\infty\hat R_H}$ whose value
  group $\Psi$ is a free abelian group produces a composed valuation
  $\hat\nu_-\circ\mu$ on $\frac{\hat R}{\hat\P_\infty}$ with value group
  $\Psi\bigoplus\Gamma$ ordered lexicographically, as
  follows:\par\noindent Given $x\in \frac{\hat R}{\hat\P_\infty}$, let
  $\psi=\mu(x)$ and blow up in $R$ the ideal $\P_\psi$ along our
  original valuation, obtaining a local ring $R'$. According to what
  we have seen so far in this section, in its completion $\hat R'$ we
  can write $x=ye$ with $\mu(e)=\psi$ and $y\in \hat R'\setminus
  H'$. The valuation $\nu$ on $R'$ extends uniquely to a valuation of
  $\frac{\hat R'}{H'}$, which we may still denote by $\hat\nu_-$
  because it induces $\hat\nu_-$ on $\frac{\hat R}{H}$. Let us
  consider the image $\overline y$ of $y$ in $\frac{\hat
    R'}{H'}$. Setting $(\hat\nu_-\circ\mu)(x)=\psi\bigoplus
  \hat\nu_-(\overline y)\in \Psi\bigoplus\Gamma$ determines a
  valuation of $\frac{\hat R}{\hat\P_\infty}$ as
  required.\par

If we drop the assumption that $\Psi$ is a free abelian group, the
above construction still works, but the value group $\hat\Gamma$ of
$\hat\nu_-\circ\mu$ need not be isomorphic to the direct sum
$\Psi\bigoplus\Gamma$. Rather, we have an exact sequence
$0\rightarrow\Gamma\rightarrow\bar\Gamma\rightarrow\Psi\rightarrow0$,
which need not, in general, be split;  see {\rm\cite{V1}, Proposition 4.3}.

\noindent In the sequel we shall reduce to the case where $\hat R$ is an
integral domain, so that $\hat\P_\infty=(0)$ and we will have
constructed a valuation of $\hat R$.
\end{remark}

\section{Definition and first properties of implicit ideals.}
\label{basics}

Let the notation be as above. Before plunging into technical details,
we would like to give a brief and informal overview of our
constructions and the motivation for them. Above we recalled the well
known fact that if $rk\ \nu=r$ then for every $\nu$-extension
$R\rightarrow R'$ the valuation $\nu$ canonically determines a flag
(\ref{eq:treechain'}) of $r$ subschemes of $\spec\ R'$. This paper
shows the existence of subschemes of $\spec\ \hat R$, determined by
$\nu$, which are equally canonical and which become explicit only
after completion. To see what they are, first of all note that the
ideal $P'_l\hat R'$, for $R'\in\cal{T}$ and $0\le\ell\le r-1$, need
not in general be prime (although it is prime whenever $R'$ is
henselian). Another way of saying the same thing is that the ring
$\frac{R'}{P'_\ell}$ need not be analytically irreducible in
general. However, we will see in \S\ref{prime}
(resp. \S\ref{henselization}) that the valuation $\nu$ picks out in a
canonical way one of the minimal primes of $P'_l\hat R'$
(resp. $P'_l\tilde R'$). We call this minimal prime $H'_{2\ell}$ for
reasons which will become apparent later. By the flatness of
completion (resp. henselization), we have $H'_{2\ell}\cap
R'=P'_\ell$. We will show that the ideals $H'_{2\ell}$ form a tree.

Let
\begin{equation}
(0)=\Delta_r\subsetneqq\Delta_{r-1}\subsetneqq\dots\subsetneqq\Delta_0=
\Gamma\label{eq:isolated1}
\end{equation}
be the isolated subgroups of $\Gamma$. There are other ideals of $\hat
R$, apart from the $H_{2\ell}$, canonically associated to $\nu$, whose
intersection with $R$ equals $P_\ell$, for example, the ideal
$\bigcap\limits_{\beta\in\Delta_\ell}\P_\beta\hat R$. The same is true
of the even larger ideal
\begin{equation}
H_{2\ell+1}=\bigcap\limits_{\beta\in\Delta_\ell}
\left(\left(\lim\limits_{\overset\longrightarrow{R'}}{\cal
      P}'_\beta\hat R'\right)\bigcap\hat R\right),\label{eq:defin2}
\end{equation}
(that $H_{2\ell+1}\cap R=P_\ell$ is easy to see and will be shown
later in this section, in Proposition \ref{contracts}). While the
examples below show that the ideal
$\bigcap\limits_{\beta\in\Delta_\ell}\P_\beta\hat R$ need not, in
general, be prime, the ideal $H_{2\ell+1}$ always is (this is the main
theorem of this paper; it will be proved in \S\ref{prime}). The ideal
$H_{2\ell+1}$ contains $H_{2\ell}$ but is not, in general equal to
it. To summarize, we will show that the valuation $\nu$ picks out in a
canonical way a generic point $H_{2\ell}$ of the formal fiber over
$P_\ell$ and also another point $H_{2\ell+1}$ in the formal fiber,
which is a specialization of $H_{2\ell}$.

The main technique used to prove these results is to to analyze the
set of zero divisors of $\frac{{R'}^\dag}{P'_\ell{R'}^\dag}$ (where
$R^\dag$ stands for either $\hat R$, $\tilde R$, or a finite type
\'etale extension $R^e$ of $R$), as follows. We show that the
reducibility of $\frac{{R'}^\dag}{P'_\ell{R'}^\dag}$ is related to the
existence of non-trivial algebraic extensions of $\kappa(P_\ell)$
inside $\kappa(P_\ell)\otimes_RR^\dag$. More precisely, in the next
section we define $R$ to be \textbf{stable} if $\frac{R^\dag}{P_{\ell
    +1}R^\dag}$ is a domain and there does not exist a non-trivial
algebraic extension of $\kappa(P_{\ell+1})$ which embeds both into
$\kappa(P_{\ell +1})\otimes_RR^\dag$ and into $\kappa(P'_{\ell+1})$
for some $R'\in\mathcal{T}$. We show that if $R$ is stable then
$\frac{{R'}^\dag}{P'_{\ell+1}{R'}^\dag}$ is a domain for all
$R'\in\cal T$. For $\overline\beta\in\frac\Gamma{\Delta_{\ell+1}}$,
let
\begin{equation}
\P_{\overline\beta}=\left\{x\in R\ \left|\
\nu(x)\mod\Delta_{\ell+1}\ge\overline\beta\right.\right\}\label{eq:pbetamodl}
\end{equation}
If $\Phi$ denotes the semigroup $\nu(R\setminus\{0\})\subset\Gamma$,
which is well ordered since $R$ is noetherian (see [ZS], Appendix 4,
Proposition 2), and
$$
\beta(\ell)=\min\{\gamma\in\Phi\ |\ \beta-\gamma\in\Delta_{\ell+1}\}
$$
then $\P_{\overline\beta}=\P_{\beta(l)}$.\par\noindent
 We have the inclusions
$$
P_\ell\subset \P_{\overline\beta}\subset P_{\ell+1},
$$
and $\P_{\overline\beta}$ is the inverse image in $R$ by the canonical map
$R\to \frac{R}{P_\ell}$ of a valuation ideal $\overline
\P_{\overline\beta}\subset \frac{R}{P_\ell}$ for the rank one
valuation $\frac{R}{P_\ell}\setminus\{0\}\to
\frac{\Delta_\ell}{\Delta_{\ell+1}}$ induced by $\nu_{\ell+1}$.\par
We will deduce from the above that if $R$ is stable then for each
$\overline\beta\in\frac{\Delta_\ell}{\Delta_{\ell+1}}$ and each
$\nu$-extension $R\rightarrow R'$ we have
$\P'_{\overline\beta}{R'}^\dag\cap R^\dag=\P_{\overline\beta} R^\dag$,
which gives us a very good control of the limit in the definition of
$H_{2\ell+1}$ and of the $\nu$-extensions $R'$ for which the limit is attained.

We then show, separately in the cases when $R^\dag=\tilde R$
(\S\ref{henselization}) and $R^\dag=\hat R$ (\S\ref{prime}), that
there always exists a stable $\nu$-extension $R'\in\cal{T}$.

We are now ready to go into details, after giving several examples of
implicit ideals and the phenomena discussed above.

Let $0\le\ell\le r$. We define our main object of study, the
$(2\ell+1)$-st implicit prime ideal $H_{2\ell+1}\subset R^\dag$, by
\begin{equation}
H_{2\ell+1}=\bigcap\limits_{\beta\in\Delta_\ell}
\left(\left(\lim\limits_{\overset\longrightarrow{R'}}{\cal P}'_\beta
    {R'}^\dag\right)\bigcap R^\dag\right),\label{eq:defin1}
\end{equation}
where $R'$ ranges over $\mathcal{T}$. As usual, we think of
(\ref{eq:defin1}) as a tree equation: if we replace $R$ by any other
$R''\in{\cal T}$ in (\ref{eq:defin1}), it defines the corresponding
ideal $H''_{2\ell +1}\subset\hat R''^\dag$. Note that for $\ell=r$
(\ref{eq:defin1}) reduces to
$$
H_{2r+1}=mR^\dag.
$$
We start by giving several examples of the ideals $H'_i$ (and also of
$\tilde H'_i$, which will appear a little later in the paper).

\begin{example}{\label{Example31}} Let $R=k[x,y,z]_{(x,y,z)}$. Let $\nu$ be the
valuation with value group $\Gamma=\mathbf Z^2_{lex}$, defined as
follows. Take a transcendental power series $\sum\limits_{j=1}^\infty
c_ju^j$  in a variable $u$ over $k$. Consider the homomorphism
$R\hookrightarrow k[[u,v]]$ which sends $x$ to $v$, $y$ to $u$ and $z$
to $\sum\limits_{j=1}^\infty c_ju^j$. Consider the valuation $\nu$,
centered at $k[[u,v]]$, defined by $\nu(v)=(0,1)$ and $\nu(u)=(1,0)$;
its restriction to $R$ will also be denoted by $\nu$, by abuse of
notation. Let $R_\nu$ denote the valuation ring of $\nu$ in $k(x,y,z)$
and let $\mathcal{T}$ be the tree consisting of all the local rings
$R'$ essentially of finite type over $R$, birationally dominated by
$R_\nu$. Let ${}^\dag=\hat{\ }$ denote the operation of formal
completion. Given $\beta=(a,b)\in\mathbf Z^2_{lex}$, we have
${\P}_\beta=x^b\left(y^a,z-c_1y-\dots-c_{a-1}y^{a-1}\right)$. The
first isolated subgroup $\Delta_1=(0)\oplus\mathbf Z$. Then
$\bigcap\limits_{\beta\in(0)\oplus\mathbf Z}\left({\cal P}_\beta\hat
  R\right)=(y,z)$ and
$\bigcap\limits_{\beta\in\Gamma=\Delta_0}\left({\cal P}_\beta\hat
  R\right)=\left(z-\sum\limits_{j=1}^\infty c_jy^j\right)$. It is not
hard to show that for any $R'\in\mathcal{T}$ we have
$H'_1=\left(z-\sum\limits_{j=1}^\infty c_jy^j\right)\hat R'$ and
that $H_3=(y,z)\hat R$. It will follow from the general theory
developed in \S\ref{extensions} that $\nu$ admits a unique extension
$\hat\nu$ to $\lim\limits_{\overset\longrightarrow{R'}}\hat R'$.
This extension has value group $\hat\Gamma=\mathbf Z^3_{lex}$ and is
defined by $\hat\nu(x)=(0,0,1)$, $\hat\nu(y)=(0,1,0)$ and
$\hat\nu\left(z-\sum\limits_{j=1}^\infty c_jy^j\right)=(1,0,0)$. For
each $R'\in\mathcal{T}$ the ideal $H'_1$ is the prime valuation
ideal corresponding to the isolated subgroup $(0)\oplus\mathbf
Z^2_{lex}$ of $\hat\Gamma$ (that is, the ideal whose elements have
values outside of $(0)\oplus\mathbf Z^2_{lex}$) while $H'_3$ is the
prime valuation ideal corresponding to the isolated subgroup
$(0)\oplus(0)\oplus\mathbf Z$.
\end{example}

\begin{example}{\label{Example32}} Let $R=k[x,y,z]_{(x,y,z)}$, $\Gamma=\mathbf
Z^2_{lex}$, the power series $\sum\limits_{j=1}^\infty c_ju^j$ and the
operation ${}^\dag=\hat{\ }$ be as in the previous example. This time,
let $\nu$ be defined as follows. Consider the homomorphism
$R\hookrightarrow k[[u,v]]$ which sends $x$ to $u$, $y$ to
$\sum\limits_{j=1}^\infty c_ju^j$ and $z$ to $v$. Consider the
valuation $\nu$, centered at $k[[u,v]]$, defined by $\nu(v)=(1,0)$ and
$\nu(u)=(0,1)$; its restriction to $R$ will be also denoted by
$\nu$. Let $R_\nu$ denote the valuation ring of $\nu$ in $k(x,y,z)$
and let $\mathcal{T}$ be the tree consisting of all the local rings
$R'$ essentially of finite type over $R$, birationally dominated by
$R_\nu$. Given $\beta=(a,b)\in\mathbf Z^2_{lex}$, we have
$\P_\beta=z^a\left(x^b,y-c_1x-\dots-c_{b-1}x^{b-1}\right)$. The first
isolated subgroup $\Delta_1=(0)\oplus\mathbf Z$. Then
$\bigcap\limits_{\beta\in(0)\oplus\mathbf Z}\left({\P}_\beta\hat
  R\right)=\left(y-\sum\limits_{j=1}^\infty c_jx^j,z\right)$ and
$\bigcap\limits_{\beta\in\Gamma=\Delta_0}\left({\cal P}_\beta\hat
  R\right)=(0)$. It is not hard to show that for any $R'\in\cal{T}$ we
have $H'_1=(0)$ and that $H_3=\left(y-\sum\limits_{j=1}^\infty
  c_jx^j,z\right)\hat R'$. In this case, the extension $\hat\nu$ to
$\lim\limits_{\overset\longrightarrow{R'}}\hat R'$ is not unique.
Indeed, one possible extension $\hat\nu^{(1)}$ has value group
$\hat\Gamma=\mathbf Z^3_{lex}$ and is defined by
$\hat\nu^{(1)}(x)=(0,0,1)$,
$\hat\nu^{(1)}\left(y-\sum\limits_{j=1}^\infty
c_jx^j\right)=(0,1,0)$ and $\hat\nu^{(1)}(z)=(1,0,0)$. In this case,
for any $R'\in\cal{T}$ the ideal $H'_3$ is the prime valuation ideal
corresponding to the isolated subgroup $(0)\oplus(0)\oplus\mathbf Z$
of $\hat\Gamma$.

Another extension $\hat\nu^{(2)}$ of $\nu$ is defined by
$\hat\nu^{(2)}(x)=(0,0,1)$,
$\hat\nu^{(2)}\left(y-\sum\limits_{j=1}^\infty
c_jx^j\right)=(1,0,0)$ and $\hat\nu^{(2)}(z)=(0,1,0)$. In this case,
the tree of ideals corresponding to the isolated subgroup
$(0)\oplus(0)\oplus\mathbf Z$ is $H'_3$ (exactly the same as for
$\hat\nu^{(1)}$) while that corresponding to $(0)\oplus\mathbf
Z^2_{lex}$ is $\tilde H'_1=\left(y-\sum\limits_{j=1}^\infty
c_jx^j\right)$. The tree $\tilde H'_1$ of prime
$\hat\nu^{(2)}$-ideals determines the extension $\hat\nu^{(2)}$
completely.
\end{example}

The following two examples illustrate the need for taking the limit
over the tree $\mathcal{T}$.

\begin{example}{\label{Example33}} Let us consider the local domain
$S=\frac{k[x,y]_{(x,y)}}{(y^2-x^2-x^3)}$. There are two distinct valuations
centered in $(x,y)$. Let $a_i\in k,\ i\geq 2$ be such that
$$
\left(y+x+\sum_{i\geq 2}a_ix^i\right)\left(y-x-\sum_{i\geq
2}a_ix^i\right)=y^2-x^2-x^3.
$$
We shall denote by $\nu_+$ the rank one discrete valuation defined by
$$
\nu_+(x)=\nu_+(y)=1,
$$
$$
\nu_+(y+x)=2,
$$
$$
\nu_+\left(y+x+\sum_{i= 2}^{b-1}a_ix^i\right)=b.
$$
Now let $R=\frac{k[x,y,z]_{(x,y,z)}}{(y^2-x^2-x^3)}$. Let $\Gamma =\mathbf Z^2$
with the lexicographical ordering. Let $\nu$ be the composite valuation
of the $(z)$-adic one with $\nu_+$, centered in $\frac R{(z)}$. The
point of this example is to show that
$$
H^*_{2\ell+1}=\bigcap_{\beta\in\Delta_{\ell}}\P_{\beta}{\hat R}
$$
does not work as the definition of the $(2\ell+1)$-st implicit prime
ideal because the resulting ideal $H^*_{2\ell+1}$ is not
prime. Indeed, as $\P_{(a,0)}=(z^a)$, we have
$$
H_1^*=\bigcap_{(a,b)\in\Z^2}\P_{(a,b)}{\hat R}=(0).
$$
Let $f=y+x+\sum\limits_{i\geq 2}a_ix^i,g=y-x-\sum\limits_{i\geq
  2}a_ix^i\in\hat R$. Clearly $f,g\notin H^*_1=(0)$, but $f\cdot
g=0$, so the ideal $H^*_1$ is not prime.

One might be tempted (as we were) to correct this problem by
localizing at $H^*_{2\ell+3}$. Indeed, if we take the new definition
of $H^*_{2\ell+1}$ to be, recursively in the descending order of
$\ell$,
\begin{equation}
H^*_{2\ell +1}=\left(\bigcap_{\beta\in\Delta_{\ell}}\P_{\beta}{\hat R}_{H^*_{2\ell
+3}}\right)\cap{\hat R},\label{eq:localization}
\end{equation}
then in the present example the resulting ideals $H^*_3=(z,f)$ and
$H^*_1=(f)$ are prime. However, the next example shows that the
definition (\ref{eq:localization}) also does not, in general, give
rise to prime ideals.
\end{example}

\begin{example}{\label{Example34}} Let
$R=\frac{k[x,y,z]_{(x,y,z)}}{(z^2-y^2(1+x))}$. Let $\Gamma=\mathbf
Z^2$ with the lexicographical ordering. Let $t$ be an independent
variable and let $\nu$ be the valuation, centered in $R$, induced by
the $t$-adic valuation of $k\left[\left[t^\Gamma\right]\right]$ under
the injective homomorphism $\iota:R\hookrightarrow
k\left[\left[t^\Gamma\right]\right]$, defined by $\iota(x)=t^{(0,1)}$,
$\iota(y)=t^{(1,0)}$ and $\iota(z)=t^{(1,0)}\sqrt{1+t^{(0,1)}}$. The
prime $\nu$-ideals of $R$ are $(0)\subsetneqq P_1\subsetneqq m$, with
$P_1=(y,z)$. We have $\bigcap\limits_{\beta\in\Delta_1}\P_\beta\hat
R=(y,z)\hat R=P_1\hat R$ and
$\bigcap\limits_{\beta\in\Gamma}\P_\beta\hat
R_{(y,z)}=\bigcap\limits_{\beta\in\Gamma}\P_\beta\hat R=(0)$. Note
that the ideal $(0)$ is not prime in $\hat R$. Now, let
$R'=R\left[\frac zy\right]_{m'}$, where $m'=\left(x,y,\frac
  zy-1\right)$ is the center of $\nu$ in $R\left[\frac zy\right]$. We
have $z-y\sqrt{1+x}\in\hat R\setminus\P_{(2,0)}\hat R$. On the other
hand, $z-y\sqrt{1+x}=y\left(\frac zy-\sqrt{1+x}\right)=0$ in $\hat
R'$; in particular,
$z-y\sqrt{1+x}\in\bigcap\limits_{\beta\in\Gamma}\P'_\beta\hat
R'$. Thus this example also shows that the ideals $\P_\beta\hat R$,
$\bigcap\limits_{\beta\in\Delta_\ell}\P_\beta\hat R$ and
$\bigcap\limits_{\beta\in\Delta_\ell}\P_\beta\hat R_{H_{2\ell+3}}$ do
not behave well under blowing up.
\end{example}

Note that both Examples \ref{Example33} and \ref{Example34} occur not only for the completion $\hat
R$ but also for the henselization $\tilde R$.

We come back to the general theory of implicit ideals.
\begin{proposition}\label{contracts} We have $H_{2\ell+1}\cap R=P_\ell$.
\end{proposition}
\begin{proof} Recall that $P_\ell=\left\{x\in R\ \left|\ \
    \nu(x)\notin \Delta_\ell\right.\right\}$. If $x\in P_\ell$ then,
since $\Delta_\ell$ is an isolated subgroup, we have $x\in {\cal
  P}_\beta$ for all $\beta\in \Delta_\ell$. The same inclusion holds
for the same reason in all extensions $R'\subset R_\nu$ of $R$, and
this implies the inclusion $P_\ell\subseteq H_{2\ell+1}\cap R$. Now
let $x$ be in $H_{2\ell+1}\cap R$ and assume $x\notin P_\ell$. Then
there is a $\beta\in \Delta_\ell$ such that $x\notin{\cal
  P}_\beta$. By faithful flatness of $R^\dag$ over $R$ we have ${\cal
  P}_\beta R^\dag\cap R={\cal P}_\beta$. This implies that
$x\notin{\cal P}_\beta R^\dag$, and the same argument holds in all the
extensions $R'\in\mathcal{T}$, so $x$ cannot be in $H_{2\ell+1}\cap
R$. This contradiction shows the desired equality.
\end{proof}

\begin{proposition}\label{behavewell} The ideals $H'_{2\ell+1}$ behave well
  under $\nu$-extensions $R\rightarrow R'$ in $\mathcal{T}$. In other
  words, let $R\rightarrow R'$ be a $\nu$-extension in $\mathcal{T}$
  and let $H'_{2\ell+1}$ denote the $(2\ell+1)$-st implicit prime
  ideal of $\hat R'$. Then $H_{2\ell+1}=H'_{2\ell+1}\cap R^\dag$.
\end{proposition}
\begin{proof} Immediate from the definitions.
\end{proof}

To study the ideals $H_{2\ell+1}$, we need to understand more explicitly the
nature of the limit appearing in (\ref{eq:defin1}). To study the relationship
between the ideals $P_\beta R^{\dag}$ and $P'_\beta {R'}^\dag\bigcap
R^\dag$, it is useful to factor the natural map $R^\dag\rightarrow{R'}^\dag$
as
$R^\dag\rightarrow(R^\dag\otimes_RR')_{M'}\overset\phi\rightarrow{R'}^\dag$
as we did in the proof of Lemma \ref{factor}. In general, the ring
$R^\dag\otimes_RR'$ is not local (see the above examples), but it has
one distinguished maximal ideal $M'$, namely, the ideal generated by
$mR^\dag\otimes1$ and $1\otimes m'$, where $m'$ denotes the maximal
ideal of $R'$. The map $\phi$ factors through the local ring
$\left(R^\dag\otimes_RR'\right)_{M'}$ and the resulting map
$\left(R^\dag\otimes_RR'\right)_{M'}\rightarrow{R'}^\dag$ is either
the formal completion or the henselization; in either case, it is
faithfully flat. Thus
$P'_\beta{R'}^\dag\cap\left(R^\dag\otimes_RR'\right)_{M'}=
P'_\beta\left(R^\dag\otimes_RR'\right)_{M'}$. This shows that we may
replace  ${R'}^\dag$ by $\left(R^\dag\otimes_RR'\right)_{M'}$ in
(\ref{eq:defin1}) without affecting the result, that is,
\begin{equation}
H_{2\ell+1}=\bigcap\limits_{\beta\in\Delta_\ell}\left(
\left(\lim\limits_{\overset\longrightarrow{R'}}
{\cal P}'_\beta\left(R^\dag\otimes_RR'\right)_{M'}\right)\bigcap
R^\dag\right).\label{eq:defin3}
\end{equation}
From now on, we will use (\ref{eq:defin3}) as our working definition of the
implicit prime ideals. One advantage of the expression (\ref{eq:defin3}) is
that it makes sense in a situation more general than the completion and the
henselization. Namely, to study the case of the henselization $\tilde R$, we
will need to consider local \'etale extensions $R^e$ of $R$, which are contained
in $\tilde R$ (particularly, those which are essentially of finite type). The
definition (\ref{eq:defin3}) of the implicit prime ideals makes sense also in
that case.

\section{Stable rings and primality of their implicit ideals.}
\label{technical}

Let the notation be as in the preceding sections. As usual, $R^\dag$
will denote one of $\hat R$, $\tilde R$ or $R^e$ (a local \'etale
$\nu$-extension essentially of finite type). Take an
$R'\in\mathcal{T}$ and
$\overline\beta\in\frac{\Delta_\ell}{\Delta_{\ell+1}}$. We have the
obvious inclusion of ideals
\begin{equation}
\P_{\overline\beta} R^\dag\subset\P_{\overline\beta}{R'}^\dag\cap
R^\dag\label{eq:stab1}
\end{equation}
(where $\P_{\overline\beta}$ is defined in (\ref{eq:pbetamodl})). A
useful subtree of $\mathcal{T}$ is formed by the $\ell$-stable rings,
which we now define. An important property of stable rings, proved
below, is that the inclusion (\ref{eq:stab1}) is an equality whenever
$R'$ is stable.
\begin{definition}\label{stable} A ring $R'\in\mathcal{T}(R)$ is said to be
  $\ell$-\textbf{stable} if the following two conditions hold:

(1) the ring
\begin{equation}
\kappa\left(P'_\ell\right)\otimes_R\left(R'\otimes_RR^\dag\right)_{M'}
\label{eq:extension1}
\end{equation}
is an integral domain and

(2) there do not exist an $R''\in\mathcal{T}(R')$ and a non-trivial
algebraic extension $L$ of $\kappa(P'_\ell)$ which embeds both into
$\kappa\left(P'_\ell\right)\otimes_R\left(R'\otimes_RR^\dag\right)_{M'}$
and $\kappa(P''_\ell)$.

We say that $R$ is \textbf{stable} if it is $\ell$-stable for each
$\ell\in\{0,\dots,r\}$.
\end{definition}
\begin{remark}\label{interchanging} (1) Rings of the form
  (\ref{eq:extension1}) will be a basic object of study in this
  paper. Another way of looking at the same ring, which we will often
  use, comes from interchanging the order of tensor product and
  localization. Namely, let $T'$ denote the image of the
  multiplicative system $\left(R'\otimes_RR^\dag\right)\setminus M'$
  under the natural map
  $R'\otimes_RR^\dag\rightarrow\kappa\left(P'_\ell\right)\otimes_RR^\dag$. Then
  the ring (\ref{eq:extension1}) equals the localization
  $(T')^{-1}\left(\kappa\left(P'_\ell\right)\otimes_RR^\dag\right)$.

(2) In the special case $R'=R$ in Definition \ref{stable}, we have
$$
\kappa\left(P'_\ell\right)\otimes_R\left(R'\otimes_RR^\dag\right)_{M'}=
\kappa\left(P_\ell\right)\otimes_RR^\dag.
$$
If, moreover, $\frac R{P_\ell}$ is analytically irreducible then the
hypothesis that $\kappa\left(P_\ell\right)\otimes_RR^\dag$ is a domain
holds automatically; in fact, this hypothesis is equivalent to
analytic irreducibility of $\frac R{P_\ell}$ if $R^\dag=\hat R$ or
$R^\dag=\tilde R$.

(3) Consider the special case when $R'$ is Henselian and ${\ }^\dag=\hat{\ }$. Excellent Henselian rings are algebraically closed inside their formal completions, so both (1) and (2) of Definition \ref{stable} hold automatically for this $R'$. Thus excellent Henselian local rings are always stable.
\end{remark}
In this section we study $\ell$-stable rings. We prove that if $R$ is
$\ell$-stable then so is any $R'\in\mathcal{T}(R)$ (justifying the
name ``stable''). The main result of this section, Theorem
\ref{primality1}, says that if $R$ is stable then the implicit ideal
$H'_{2\ell+1}$ is prime for each $\ell\in\{0,\dots,r\}$ and each
$R'\in\mathcal{T}(R)$.
\begin{remark}\label{hypotheses} In the next two sections we will show
  that there exist stable rings $R'\in\cal T$ for both $R^\dag=\hat R$
  and $R^\dag=R^e$. However, the proof of this is different depending
  on whether we are dealing with completion or with an \'etale
  extension, and will be carried out separately in two separate
  sections: one devoted to henselization, the other to completion.
\end{remark}
\begin{proposition}\label{largeR1} Fix an integer $\ell$, $0\le\ell\le
  r$. Assume that $R'$ is $\ell$-stable and let
  $R''\in\mathcal{T}(R')$. Then $R''$ is $\ell$-stable.
\end{proposition}
\begin{proof} We have to show that (1) and (2) of
Definition \ref{stable} for $R'$ imply (1) and (2) of
Definition \ref{stable} for $R''$. The ring
\begin{equation}
\kappa\left(P''_\ell\right)\otimes_R\left(R''\otimes_RR^\dag\right)_{M''}
\label{eq:extension}
\end{equation}
is a localization of
$\kappa\left(P''_\ell\right)\otimes_R\left(\kappa\left(P'_\ell\right)
\otimes_R\left(R'\otimes_RR^\dag\right)_{M'}\right)$. Hence (1) and
(2) of Definition \ref{stable}, applied to $R'$, imply that
$\kappa\left(P''_\ell\right)\otimes_R\left(R''\otimes_RR^\dag\right)_{M''}$
is an integral domain, so (1) of Definition \ref{stable} holds for
$R''$. Replacing $R'$ by $R''$ clearly does not affect the hypotheses
about the non-existence of the extension $L$, so (2) of Definition
\ref{stable} also holds for $R''$.
\end{proof}

Next, we prove a technical result on which much of the rest of the
paper is based. For $\overline\beta\in\frac\Gamma{\Delta_{\ell+1}}$,
let
\begin{equation}
\P_{\overline\beta+}=\left\{x\in R\ \left|\
\nu(x)\mod\Delta_{\ell+1}>\overline\beta\right.\right\}.\label{eq:pbetamodl+}
\end{equation}
As usual, $\P'_{\overline\beta+}$ will stand for the analogous notion,
but with $R$ replaced by $R'$, etc.
\begin{proposition}\label{largeR2} Assume that $R$ itself is $(\ell+1)$-stable and let
$R'\in\mathcal{T}(R)$.
\begin{enumerate}
\item For any $\overline\beta\in\frac{\Delta_\ell}{\Delta_{\ell+1}}$
\begin{equation}
\P'_{\overline\beta}{R'}^\dag\cap R^\dag=\P_{\overline\beta} R^\dag.\label{eq:stab}
\end{equation}
\item For any $\overline\beta\in\frac\Gamma{\Delta_{\ell+1}}$ the natural map
    \begin{equation}\label{eq:gammaversion}
    \frac{\P_{\overline\beta}R^\dag}{\P_{\overline\beta+}R^\dag}\rightarrow
\frac{\P'_{\overline\beta}{R'}^\dag}{\P'_{\overline\beta+}{R'}^\dag}
     \end{equation}
     is injective.
\end{enumerate}
\end{proposition}
\begin{proof} As explained at the end of the
previous section, since ${R'}^\dag$ is faithfully flat over the ring
$\left(R^\dag\otimes_RR'\right)_{M'}$, we may replace ${R'}^\dag$ by
$\left(R^\dag\otimes_RR'\right)_{M'}$ in both 1 and 2 of the Proposition.

\noi\textbf{Proof of 1 of the Proposition:} It is sufficient to prove that
\begin{equation}
\P'_{\overline\beta}\left(R^\dag\otimes_RR'\right)_{M'}\bigcap
R^\dag=\P_{\overline\beta} R^\dag.\label{eq:stab2}
\end{equation}
One inclusion in (\ref{eq:stab2}) is trivial; we must show that
\begin{equation}
\P'_{\overline\beta}\left(R^\dag\otimes_RR'\right)_{M'}\bigcap
R^\dag\subset\P_{\overline\beta} R^\dag.\label{eq:stab3}
\end{equation}
\begin{lemma}\label{injectivity} Let $T'$ denote the image of the
multiplicative set $\left(R'\otimes_RR^\dag\right)\setminus M'$ under
the natural map of $R$-algebras
$R'\otimes_RR^\dag\rightarrow\frac{R'_{P'_{\ell+1}}}{\P'_{\overline\beta}
R'_{P'_{\ell+1}}}\otimes_RR^\dag$. Then the map of $R$-algebras
\begin{equation}
\bar\pi:\frac{R_{P_{\ell+1}}}{\P_{\overline\beta}R_{P_{\ell+1}}}\otimes_RR^\dag
\rightarrow(T')^{-1}\left(\frac{R'_{P'_{\ell+1}}}{\P'_{\overline\beta}
R'_{P'_{\ell+1}}}\otimes_RR^\dag\right)
\label{eq:inclusion3}
\end{equation}
induced by $\pi:R\rightarrow R'$ is injective.
\end{lemma}
\begin{proof}\textit{(of Lemma \ref{injectivity})} We start with the field extension
$$
\kappa(P_{\ell+1})\hookrightarrow\kappa(P'_{\ell+1})
$$
induced by $\pi$. Since $R^\dag$ is flat over $R$, the induced map
$\pi_1:\kappa(P_{\ell+1})\otimes_RR^\dag\rightarrow\kappa(P'_{\ell+1})\otimes_R
R^\dag$ is also injective. By (1) of Definition \ref{stable},
$\kappa(P'_{\ell+1})\otimes_RR^\dag$ is a domain. In particular,
\begin{equation}
\kappa\left(P'_{\ell+1}\right)\otimes_RR^\dag=
\left(\frac{R'_{P'_{\ell+1}}}{\P'_{\overline\beta}R'_{P'_{\ell+1}}}\otimes_R
R^\dag\right)_{red}.
\label{eq:reduced1}
\end{equation}
The local ring
$\frac{R'_{P'_{\ell+1}}}{\P'_{\overline\beta}R'_{P'_{\ell+1}}}$ is
artinian because it can be seen as the quotient of
$\frac{R'_{P'_{\ell+1}}}{P'_\ell R'_{P'_{\ell+1}}}$ by a valuation
ideal corresponding to a rank one valuation. Since the ring is
noetherian the valuation of the maximal ideal is positive, and since
the group is archimedian, a power of the maximal ideal is contained in
the valuation ideal.\par
Therefore, its only associated prime is its nilradical, the ideal
$\frac{P'_{\ell+1}
  R'_{P'_{\ell+1}}}{\P'_{\overline\beta}R'_{P'_{\ell+1}}}$; in
particular, the $(0)$ ideal in this ring has no embedded
components. Since $R^\dag$ is flat over $R$,
$\frac{R'_{P'_{\ell+1}}}{\P'_{\overline\beta}R'_{P'_{\ell+1}}}\otimes_RR^\dag$
is flat over
$\frac{R'_{P'_{\ell+1}}}{\P'_{\overline\beta}R'_{P'_{\ell+1}}}$ by
base change. Hence the $(0)$ ideal of
$\frac{R'_{P'_{\ell+1}}}{\P'_{\overline\beta}R'_{P'_{\ell+1}}}\otimes_RR^\dag$
has no embedded components. In particular, since the multiplicative
system $T'$ is disjoint from the nilradical of
$\frac{R'_{P'_{\ell+1}}}{\P'_{\overline\beta}R'_{P'_{\ell+1}}}\otimes_RR^\dag$,
the set $T'$ contains no zero divisors, so localization by $T'$ is
injective.\par
By the definition of $\P_{\overline\beta}$, the map
$\frac{R_{P_{\ell+1}}}{\P_{\overline\beta}R_{P_{\ell+1}}}\rightarrow
\frac{R'_{P'_{\ell+1}}}{\P'_{\overline\beta}R'_{P'_{\ell+1}}}$ is
injective, hence so is
$$
\frac{R_{P_{\ell+1}}}{\P_{\overline\beta}R_{P_{\ell+1}}}\otimes_RR^\dag\rightarrow
\frac{R'_{P'_{\ell+1}}}{\P'_{\overline\beta}R'_{P'_{\ell+1}}}\otimes_RR^\dag
$$
by the flatness of $R^\dag$ over $R$. Combining this with the injectivity of the
localization by $T'$, we obtain that $\bar\pi$ is injective, as
desired. Lemma \ref{injectivity} is proved.
\end{proof}

Again by the definition of $\P_{\overline\beta}$, the localization map
$\frac
R{\P_{\overline\beta}}\rightarrow\frac{R_{P_{\ell+1}}}{\P_{\overline\beta}
R_{P_{\ell+1}}}$ is injective, hence so is the map
\begin{equation}
\frac {R}{\P_{\overline\beta}}\otimes_RR^\dag\rightarrow
\frac{R_{P_{\ell+1}}}{\P_{\overline\beta}R_{P_{\ell+1}}}\otimes_RR^\dag
\label{eq:injective}
\end{equation}
by the flatness of $R^\dag$ over $R$. Combining this with Lemma
\ref{injectivity}, we see that the composition
\begin{equation}
\frac R{\P_{\overline\beta}}\otimes_RR^\dag\rightarrow
(T')^{-1}\left(\frac{R'_{P'_\ell}}{\P'_{\overline\beta}
    R'_{P'_\ell}}\otimes_RR^\dag\right)
\label{eq:injective1}
\end{equation}
of (\ref{eq:injective}) with $\bar\pi$ is also injective. Now,
(\ref{eq:injective1}) factors through
$\left(\frac{R'}{\P'_{\overline\beta}}\otimes_RR^\dag\right)_{M'}$
(here we are guilty of a slight abuse of notation: we denote the
natural image of $M'$ in
$\frac{R'}{\P'_{\overline\beta}}\otimes_RR^\dag$ also by $M'$). Hence
the map
\begin{equation}
\frac
R{\P_{\overline\beta}}\otimes_RR^\dag\rightarrow
\left(\frac{R'}{\P'_{\overline\beta}}\otimes_RR^\dag\right)_{M'}
\label{eq:injective2}
\end{equation}
is injective. Since $\frac
R{\P_{\overline\beta}}\otimes_RR^\dag\cong\frac{R^\dag}{\P_{\overline\beta}
  R^\dag}$ and
$\left(\frac{R'}{\P'_{\overline\beta}}\otimes_RR^\dag\right)_{M'}\cong
\frac{\left(R'\otimes_RR^\dag\right)_{M'}}{\P'_{\overline\beta}
\left(R'\otimes_RR^\dag\right)_{M'}}$, the injectivity of
(\ref{eq:injective2}) is the same as (\ref{eq:stab3}). This completes
the proof of 1 of the Proposition.\medskip

\noi\textbf{Proof of 2 of the Proposition:} We start with the
injective homomorphism
\begin{equation}
\frac{\P_{\overline\beta}}{\P_{\overline\beta+}}\otimes_R\kappa(P_{\ell+1})
\rightarrow\frac{\P'_{\overline\beta}}{\P'_{\overline\beta+}}\otimes_R
\kappa(P'_{\ell+1})\label{eq:vspaces}
\end{equation}
of $\kappa(P_{\ell+1})$-vector spaces. Since $R^\dag$ is flat over
$R$, tensoring (\ref{eq:vspaces}) produces an injective homomorphism
\begin{equation}
\frac{\P_{\overline\beta}R^\dag}{\P_{\overline\beta+}R^\dag}\otimes_R
\kappa(P_{\ell+1})\rightarrow\frac{\P'_{\overline\beta}}{\P'_{\overline\beta+}}
\otimes_R\kappa(P'_{\ell+1})\otimes_RR^\dag\label{eq:Rdagmodules}
\end{equation}
of $R^\dag$-modules. Now, the $\kappa(P_{\ell+1})$-vector space
$\frac{\P'_{\overline\beta}}{\P'_{\overline\beta+}}\otimes_R\kappa(P'_{\ell+1})$
is, in particular, a torsion-free $\kappa(P_{\ell+1})$-module. Since
$\kappa(P_{\ell+1})\otimes_RR^\dag$ is a domain by definition of
$(\ell+1)$-stable and by the flatness of
$R^\dag\otimes_R\kappa(P_{\ell+1})$ over $\kappa(P_{\ell+1})$, the
$R^\dag\otimes_R\kappa(P_{\ell+1})$-module
$\frac{\P'_{\overline\beta}}{\P'_{\overline\beta+}}\otimes_R\kappa(P'_{\ell+1})
\otimes_RR^\dag$ is also torsion-free; in particular, its localization
map by any multiplicative system is injective. Let $S'$ denote the image of the
multiplicative set $\left(R'\otimes_RR^\dag\right)\setminus M'$ under
the natural map of $R$-algebras
$R'\otimes_RR^\dag\rightarrow\kappa(P'_{\ell+1})\otimes_RR^\dag$. By
the above, the composition
\begin{equation}
\frac{\P_{\overline\beta}R^\dag}{\P_{\overline\beta+}R^\dag}\otimes_R
\kappa(P_{\ell+1})\rightarrow(S')^{-1}
\left(\frac{\P'_{\overline\beta}}{\P'_{\overline\beta+}}\otimes_R
\kappa(P'_{\ell+1})\otimes_RR^\dag\right)\label{eq:Rdagmodules1}
\end{equation}
of (\ref{eq:Rdagmodules}) with the localization by $S'$ is injective.

By the definition of $\P_{\overline\beta}$, the localization map
$\frac{\P_{\overline\beta}}{\P_{\overline\beta+}}\rightarrow
\frac{\P_{\overline\beta}}{\P_{\overline\beta+}}\otimes_R\kappa(P_{\ell+1})$
is injective, hence so is the map
\begin{equation}
\frac{\P_{\overline\beta}R^\dag}{\P_{\overline\beta+}R^\dag}=
\frac{\P_{\overline\beta}}{\P_{\overline\beta+}}\otimes_RR^\dag\rightarrow
\frac{\P_{\overline\beta}R^\dag}{\P_{\overline\beta+}R^\dag}\otimes_R
\kappa(P_{\ell+1})\label{eq:injectivegamma}
\end{equation}
by the flatness of $R^\dag$ over $R$. Combining this with the
injectivity of (\ref{eq:Rdagmodules1}), we see that the composition
\begin{equation}
\frac{\P_{\overline\beta}R^\dag}{\P_{\overline\beta+}R^\dag}\rightarrow
(S')^{-1}\left(\frac{\P'_{\overline\beta}}{\P'_{\overline\beta+}}\otimes_R
\kappa(P'_{\ell+1})\otimes_RR^\dag\right)
\label{eq:injective1gamma}
\end{equation}
of (\ref{eq:injectivegamma}) with (\ref{eq:Rdagmodules1}) is also
injective. Now, (\ref{eq:injective1gamma}) factors through
$\frac{\P'_{\overline\beta}}{\P'_{\overline\beta+}}\otimes_{R'}
\left(R'\otimes_RR^\dag\right)_{M'}$. Hence the map
\begin{equation}
\frac{\P_{\overline\beta}R^\dag}{\P_{\overline\beta+}R^\dag}\rightarrow
\frac{\P'_{\overline\beta}}{\P'_{\overline\beta+}}\otimes_{R'}
\left(R'\otimes_RR^\dag\right)_{M'}\label{eq:injective2gamma}
\end{equation}
is injective. Since
$\frac{\P'_{\overline\beta}}{\P'_{\overline\beta+}}\otimes_{R'}{R'}^\dag\cong
\frac{\P'_{\overline\beta}{R'}^\dag}{\P'_{\overline\beta+}{R'}^\dag}$ and
by faithful flatness of ${R'}^\dag$ over
$\left(R'\otimes_RR^\dag\right)_{M'}$, the injectivity of
(\ref{eq:injective2gamma}) implies the injectivity of the map
(\ref{eq:gammaversion}) required in 2 of the Proposition.

This completes the proof of the Proposition.
\end{proof}

\begin{corollary}\label{stableimplicit} Take an integer
  $\ell\in\{0,\dots,r-1\}$ and assume that $R$ is
  $(\ell+1)$-stable. Then
\begin{equation}
H_{2\ell+1}=\bigcap\limits_{\beta\in\Delta_\ell}{\cal P}_\beta
R^\dag.\label{eq:defin5}
\end{equation}
\end{corollary}
\begin{proof} By Lemma 4 of Appendix 4 of \cite{ZS}, the ideals
$\P_{\overline\beta}$ are cofinal among the ideals $\P_\beta$ for
$\beta\in \Delta_\ell$.
\end{proof}

\begin{corollary}\label{stablecontracts} Assume that $R$ is stable. Take an
  element $\beta\in\Gamma$. Then $\P'_\beta{R'}^\dag\cap
  R^\dag=\P_\beta$.
\end{corollary}
\begin{proof} It is sufficient to prove that for each
$\ell\in\{0,\dots,r-1\}$ and
$\bar\beta\in\frac\Gamma{\Delta_{\ell+1}}$, we have
\begin{equation}
\P'_{\bar\beta}{R'}^\dag\cap R^\dag=\P_{\bar\beta};\label{eq:scalewisebeta}
\end{equation}
the Corollary is just the special case of (\ref{eq:scalewisebeta})
when $\ell=r-1$. We prove (\ref{eq:scalewisebeta}) by
contradiction. Assume the contrary and take the smallest $\ell$ for
which (\ref{eq:scalewisebeta}) fails to be true. Let
$\Phi'=\nu(R'\setminus\{0\})$. We will denote by
$\frac\Phi{\Delta_{\ell+1}}$ the image of $\Phi$ under the composition
of natural maps
$\Phi\hookrightarrow\Gamma\rightarrow\frac\Gamma{\Delta_{\ell+1}}$ and
similarly for $\frac{\Phi'}{\Delta_{\ell+1}}$. Clearly, if
(\ref{eq:scalewisebeta}) fails for a certain $\bar\beta$, it also
fails for some $\bar\beta\in\frac{\Phi'}{\Delta_{\ell+1}}$; take the
smallest $\bar\beta\in\frac{\Phi'}{\Delta_{\ell+1}}$ with this
property. If we had
$\bar\beta=\min\left\{\left.\tilde\beta\in\frac{\Phi'}{\Delta_{\ell+1}}\
  \right|\ \tilde\beta-\bar\beta\in\Delta_\ell\right\}$, then
(\ref{eq:scalewisebeta}) would also fail with $\bar\beta$ replaced by
$\bar\beta\mod\Delta_\ell\in\frac\Gamma{\Delta_\ell}$, contradicting
the minimality of $\ell$. Thus
\begin{equation}\label{eq:notminimum}
\bar\beta>\min\left\{\left.\tilde\beta\in\frac{\Phi'}{\Delta_{\ell+1}}\
  \right|\ \tilde\beta-\bar\beta\in\Delta_\ell\right\}.
\end{equation}
Let $\bar\beta-$ denote the immediate predecessor of $\bar\beta$ in
$\frac{\Phi'}{\Delta_{\ell+1}}$. By (\ref{eq:notminimum}), we have
$\bar\beta-\bar\beta-\in\Delta_\ell$. By the choice of $\bar\beta$, we
have $\P_{\bar\beta-}=\P'_{\bar\beta-}{R'}^\dag\cap R^\dag$ but
$\P_{\bar\beta}\subsetneqq\P'_{\bar\beta}{R'}^\dag\cap R^\dag$. This
contradicts Proposition \ref{largeR2}, applied to $\bar\beta-$. The
Corollary is proved.
\end{proof}

Now we are ready to state and prove the main Theorem of this section.
\begin{theorem}\label{primality1} (1) Fix an integer
  $\ell\in\{1,\dots,r+1\}$. Assume that there exists
  $R'\in\mathcal{T}(R)$ which is $(\ell+1)$-stable. Then the ideal
  $H_{2\ell+1}$ is prime.

(2) Let $i=2\ell+2$. There exists an extension $\nu^\dag_{i0}$ of
$\nu_{\ell+1}$ to
$\lim\limits_{\overset\longrightarrow{R'}}\kappa(H'_{i-1})$, with
value group
\begin{equation}
\Delta_{i-1,0}=\frac{\Delta_\ell}{\Delta_{\ell+1}},\label{eq:groupequal2}
\end{equation}
whose valuation ideals are described as follows. For an element
$\overline\beta\in\frac{\Delta_\ell}{\Delta_{\ell+1}}$, the
$\nu^\dag_{i0}$-ideal of $\frac{R^\dag}{H_{i-1}}$ of value
$\overline\beta$, denoted by $\P^\dag_{\overline\beta\ell}$, is given
by the formula
\begin{equation}
\P^\dag_{\overline\beta,\ell+1}=\left(\lim\limits_{\overset\longrightarrow{R'}}
\frac{\P'_{\overline\beta}{R'}^\dag}{H'_{i-1}}\right)\cap\frac{R^\dag}{H_{i-1}}.
\label{eq:valideal1}
\end{equation}
\end{theorem}
\begin{remark} Once the even-numbered implicit prime ideals $H'_{2\ell}$
  are defined below, we will show that $\nu^\dag_{i0}$ is the unique
  extension of $\nu_{\ell+1}$ to
  $\lim\limits_{\overset\longrightarrow{R'}}\kappa(H'_{i-1})$,
  centered in the local ring
  $\lim\limits_{\overset\longrightarrow{R'}}
\frac{R'^\dag_{H'_{2\ell+2}}}{H'_{2\ell+1}R'^\dag_{H'_{2\ell+2}}}$.
\end{remark}
\begin{proof}\textit{(of Theorem \ref{primality1})} Let $R'$ be a stable
ring in $\mathcal{T}(R)$. Once Theorem \ref{primality1} is proved for
$R'$, the same results for $R$ will follow easily by intersecting all
the ideals of ${R'}^\dag$ in sight with $R^\dag$. Therefore from now
on we will replace $R$ by $R'$, that is, we will assume that $R$
itself is stable.

Let $\Phi_\ell$ denote the image of the semigroup
$\nu(R\setminus\{0\})$ in $\frac{\Gamma}{\Delta_{\ell+1}}$. As we saw
above, $\Phi_\ell$ is well ordered. For an element
$\overline\beta\in\Phi_\ell$, let $\overline\beta+$ denote the
immediate successor of $\overline\beta$ in $\Phi_\ell$.

Take any element $x\in R^\dag\setminus H_{i-1}$. By  Corollary
\ref{stableimplicit}, there exists (a unique)
$\overline\beta\in\Phi_\ell\cap\frac{\Delta_\ell}{\Delta_{\ell+1}}$
such that
\begin{equation}
x\in{\cal P}_{\overline\beta} R^\dag\setminus{\cal
  P}_{\overline\beta+}R^\dag\label{eq:xinbeta}
\end{equation}
(where, of course, we allow $\overline\beta=0$). Let $\bar x$ denote the image of $x$ in $\frac{R^\dag}{H_{i-1}}$.
We define
$$
\nu^\dag_{i0}(\bar x)=\overline\beta.
$$
Next, take another element $y\in R^\dag\setminus H_{2\ell+1}$ and let
$\gamma\in\Phi_\ell\cap\frac{\Delta_\ell}{\Delta_{\ell+1}}$ be such that
\begin{equation}
y\in{\cal P}_{\overline\gamma}R^\dag\setminus{\cal P}_{\overline\gamma+}R^\dag.\label{eq:yingamma}
\end{equation}
Let $(a_1,...,a_n)$ be a set of generators of ${\cal P}_{\overline\beta}$ and
$(b_1,...,b_s)$ a set of generators of ${\cal P}_{\overline\gamma}$, with
$\nu_{\ell+1}(a_1)=\overline\beta$ and $\nu_{\ell+1}(b_1)=\overline\gamma$. Let $R'$ be a local
blowing up along $\nu$ such that $R'$ contains all the fractions
$\frac{a_i}{a_1}$ and $\frac{b_j}{b_1}$. By Proposition \ref{largeR1} and Definition \ref{stable} (1), the ideal
$P'_{\ell+1}{R'}^\dag$ is prime. By construction, we have $a_1\ |\ x$ and $b_1\ |\ y$ in ${R'}^\dag$.
Write $x=za_1$ and $y=wb_1$ in ${R'}^\dag$. The equality
(\ref{eq:stab}), combined with (\ref{eq:xinbeta}) and (\ref{eq:yingamma}), implies that $z,w\notin P'_{\ell+1}{R'}^\dag$, hence
\begin{equation}
zw\notin P'_{\ell+1}{R'}^\dag\label{eq:zwnotin}
\end{equation}
by the primality of $P'_{\ell+1}{R'}^\dag$. We obtain
\begin{equation}
xy=a_1b_1zw.\label{eq:xyzw}
\end{equation}
Since $\nu$ is a valuation on $R'$, we have $\left({\cal P}'_{\overline\beta+\overline\gamma+}:(a_1b_1)R'\right)\subset
P'_{\ell +1}$. By faithful flatness of ${R'}^\dag$ over $R'$ we obtain
\begin{equation}
\left({\cal P}'_{\overline\beta+\overline\gamma+}{R'}^\dag:(a_1b_1){R'}^\dag\right)\subset P'_{\ell+1}{R'}^\dag.
\end{equation}
Combining this with (\ref{eq:zwnotin}) and (\ref{eq:xyzw}), we obtain
\begin{equation}
xy\notin{\cal P}_{\overline\beta+\overline\gamma+}R^\dag,\label{eq:xynotin}
\end{equation}
in particular, $xy\notin H_{2\ell+1}$. We started with two arbitrary elements $x,y\in R^\dag\setminus H_{2\ell+1}$ and
showed that $xy\notin H_{2\ell+1}$. This proves (1) of the Theorem.

Furthermore, (\ref{eq:xynotin}) shows that $\nu^\dag_{i0}(\bar x\bar y)=\overline\beta+\overline\gamma$, so $\nu^\dag_{i0}$ induces a
valuation of $\kappa(H_{i-1})$ and hence also of $\lim\limits_{\overset\longrightarrow{R'}}\kappa(H'_{i-1})$.
Equality (\ref{eq:groupequal2}) holds by definition and (\ref{eq:valideal1}) by the assumed stability of $R$.
\end{proof}

Next, we define the even-numbered implicit prime ideals $H'_{2\ell}$. The only information we need to use to define the
prime ideals $H'_{2\ell}\subset H'_{2\ell+1}$ and to prove that $H'_{2\ell-1}\subset H'_{2\ell}$ are the facts that
$H_{2\ell+1}$ is a prime lying over $P_\ell$ and that the ring homomorphism $R'\rightarrow{R'}^\dag$ is regular.
\begin{proposition}\label{H2l} There exists a unique minimal prime ideal $H_{2\ell}$ of
$P_\ell R^\dag$, contained in $H_{2\ell+1}$.
\end{proposition}
\begin{proof} Since $H_{2\ell+1}\cap R=P_\ell$, $H_{2\ell+1}$ belongs to
the fiber of the map $Spec\ R^\dag\rightarrow Spec\ R$ over $P_\ell$. Since
$R$ was assumed to be excellent, $S:={R^\dag}\otimes_R\kappa(P_\ell)$
is a regular ring (note that the excellence assumption is needed only in the case $R^\dag=\hat R$; the ring homomorphism
$R\rightarrow R^\dag$ is automatically regular if $R^\dag=\tilde R$ or $R^\dag=R^e$). Hence its localization $\bar
S:=S_{H_{2\ell+1}S}\cong\frac{R^\dag_{H_{2\ell+1}}}{P_\ell R^\dag_{H_{2\ell+1}}}$ is
a regular {\em local} ring. In particular, $\bar S$ is an integral
domain, so $(0)$ is its unique minimal prime ideal. The set of minimal
prime ideals of $\bar S$ is in one-to-one correspondence with the set of
minimal primes of $P_\ell$, contained in $H_{2\ell+1}$, which shows that such
a minimal prime $H_{2\ell}$ is unique, as desired.
\end{proof}

We have $P_\ell\subset H_{2\ell}\cap R\subset H_{2\ell+1}\cap R=P_\ell$, so $H_{2\ell}\cap R\subset P_\ell$.
\begin{proposition} We have $H_{2\ell-1}\subset H_{2\ell}$.
\end{proposition}
\begin{proof} Take an element $\beta\in\frac{\Delta_{\ell-1}}{\Delta_\ell}$ and a stable ring $R'\in\cal T$.
Then $\P'_\beta\subset P'_\ell$, so
\begin{equation}
H'_{2\ell-1}\subset\P'_\beta{R'}^\dag\subset P'_\ell{R'}^\dag\subset H'_{2\ell}.\label{eq:inclusion}
\end{equation}
Intersecting (\ref{eq:inclusion}) back with $R^\dag$ we get the result.
\end{proof}

In \S\ref{henselization} we will see that if $R^\dag=\tilde R$ or $R^\dag=R^e$ then $H_{2\ell}=H_{2\ell+1}$ for all
$\ell$.

Let the notation be the same as in Theorem \ref{primality1}.
\begin{proposition}\label{nu0unique} The valuation $\nu_{i0}^\dag$ is the unique extension of $\nu_\ell$ to a valuation of
$\lim\limits_{\overset\longrightarrow{R'}}\kappa(H'_{i-1})$, centered in the local ring
$\lim\limits_{\overset\longrightarrow{R'}}\frac{{R'}^\dag_{H'_{2\ell}}}{H'_{2\ell-1}{R'}^\dag_{H'_{2\ell}}}$.
\end{proposition}
\begin{proof} As usual, without loss of generality we may assume that $R$ is stable. Take an element
$x\in R^\dag\setminus H_{2\ell-1}$. Let $\beta=\nu^\dag_{i0}(\bar x)$ and let $R'$ be the blowing up of the ideal
$\P_\beta=(a_1,\dots,a_n)$, as in the proof of Theorem \ref{primality1}. Write
\begin{equation}
x=za_1\label{eq:xza}
\end{equation}
in $R'$. We have $z\in{R'}^\dag\setminus P'_\ell{R'}^\dag$, hence
\begin{equation}
\bar z\in\frac{{R'}^\dag_{H'_{2\ell}}}{H'_{2\ell-1}{R'}^\dag_{H'_{2\ell}}}\setminus\frac{P'_\ell{R'}^\dag_{H'_{2\ell}}}{H'_{2\ell-1}{R'}^\dag_{H'_{2\ell}}}=
\frac{{R'}^\dag_{H'_{2\ell}}}{H'_{2\ell-1}{R'}^\dag_{H'_{2\ell}}}\setminus\frac{H'_{2\ell}{R'}^\dag_{H'_{2\ell}}}{H'_{2\ell-1}{R'}^\dag_{H'_{2\ell}}}.
\label{eq:znotinPl}
\end{equation}
If $\nu^*$ is any other extension of $\nu_\ell$ to $\lim\limits_{\overset\longrightarrow{R'}}\kappa(H'_{i-1})$, centered in
$\lim\limits_{\overset\longrightarrow{R'}}\frac{{R'}^\dag_{H'_{2\ell}}}{H'_{2\ell-1}{R'}^\dag_{H'_{2\ell}}}$, then $\nu^*(\bar a_1)=\beta$, $\nu^*(z)=0$
by (\ref{eq:znotinPl}), so $\nu^*(\bar x)=\beta=\nu_{i0}^\dag(\bar x)$. This completes the proof of the uniqueness of
$\nu_{i0}^\dag$.
\end{proof}
\begin{remark}\label{sameresfield} If $R'$ is stable, we have a natural isomorphism of graded algebras
$$
\gr_{\nu^\dag_{i0}}\frac{{R'}^\dag_{H'_{2\ell}}}{H'_{2\ell-1}{R'}^\dag_{H'_{2\ell}}}\cong
\gr_{\nu_\ell}\frac{R'_{P'_\ell}}{P'_{\ell-1}R'_{P'_\ell}}\otimes_{R'}\kappa(H'_{2\ell}).
$$
In particular, the residue field of $\nu^\dag_{i0}$ is
$k_{\nu^\dag_{i0}}=\lim\limits_{\overset\longrightarrow{R'}}\kappa(H'_{2\ell})$.
\end{remark}

\section{A classification of extensions of $\nu$ to $\hat R$.}
\label{Rdag}

The purpose of this section is to give a systematic description of
all the possible extensions $\nu^\dag_-$ of $\nu$ to a quotient of
$R^\dag$ by a minimal prime as compositions of $2r$ valuations,
\begin{equation}
\nu^\dag_-=\nu^\dag_1\circ\dots\circ\nu^\dag_{2r},\label{eq:composition}
\end{equation}
satisfying certain conditions. One is naturally led to consider the
more general problem of extending $\nu$ not only to rings of the form
$\frac{R^\dag}P$ but also to the ring
$\lim\limits_\to\frac{{R'}^\dag}{P'}$, where $P'$ is a tree of prime
ideals of ${R'}^\dag$, such that $P'\cap R'=(0)$. We deal in a uniform
way with all the three cases $R^\dag=\hat R$, $R^\dag=\tilde R$ and
$R^\dag=R^e$, in order to be able to apply the results proved here to
all three later in the paper. However, the reader should think of the
case $R^\dag=\hat R$ as the main case of interest and the cases
$R^\dag=\tilde R$ and $R^\dag=R^e$ as auxiliary and slightly
degenerate, since, as we shall see, in these cases the equality
$H_{2\ell}=H_{2\ell+1}$ is satisfied for all $\ell$ and the extension
$\nu^\dag_-$ will later be shown to be unique.

We will associate to each extension $\nu^\dag_-$ of $\nu$ to $R^\dag$ a chain
\begin{equation}
\tilde H'_0\subset\tilde H'_1\subset\dots\subset\tilde
H'_{2r}=m'{R'}^\dag\label{eq:chaintree''}
\end{equation}
of prime $\nu^\dag_-$-ideals, corresponding to the decomposition
(\ref{eq:composition}) and prove some basic properties of this chain of ideals.

Now for the details. We wish to classify all the pairs $\left(\left\{\tilde
H'_0\right\},\nu^\dag _+\right)$, where $\left\{\tilde H'_0\right\}$
is a tree of prime ideals of ${R'}^\dag$, such that $\tilde H'_0\cap
R'=(0)$, and $\nu^\dag_+$ is an extension of $\nu$ to the ring
$\lim\limits_\to\frac{{R'}^\dag}{\tilde H'_0}$.

Pick and fix one such pair $\left(\left\{\tilde
    H'_0\right\},\nu^\dag_+\right)$. We associate to it the following
collection of data, which, as we will see, will in turn determine the
pair $\left(\left\{\tilde H'_0\right\},\nu^\dag_+\right)$.

First, we associate to $\left(\left\{\tilde
    H'_0\right\},\nu^\dag_-\right)$ a chain (\ref{eq:chaintree''}) of
$2r$ trees of prime $\nu^\dag_-$-ideals. Let $\Gamma^\dag$ denote the
value group of $\nu^\dag_-$. Defining (\ref{eq:chaintree''}) is
equivalent to defining a chain
\begin{equation}
\Gamma^\dag=\Delta^\dag_0\supset\Delta^\dag_1\supset\dots\supset\Delta^\dag_{2r}=
\Delta^\dag_{2r+1}=(0)\label{eq:groups}
\end{equation}
of $2r$ isolated subroups of $\Gamma^\dag$ (the chain
(\ref{eq:groups}) will not, in general, be maximal, and
$\Delta^\dag_{2\ell +1}$ need not be distinct from $\Delta^\dag_{2\ell}$).

We define the $\Delta^\dag_i$ as follows. For $0\le\ell\le r$, let
$\Delta^\dag_{2\ell}$ and $\Delta^\dag_{2\ell +1}$ denote, respectively,
the greatest and the smallest isolated subgroups of $\Gamma^\dag$ such that
\begin{equation}
\Delta^\dag_{2\ell}\cap\Gamma=\Delta^\dag_{2\ell
  +1}\cap\Gamma=\Delta_\ell.\label{eq:Delta}
\end{equation}
\begin{lemma}\label{rank1} We have
\begin{equation}
rk\ \frac{\Delta^\dag_{2\ell-1}}{\Delta^\dag_{2\ell}}=1\label{eq:rank1}
\end{equation}
for $1\le \ell\le r$.
\end{lemma}
\begin{proof} Since by construction
$\Delta^\dag_{2\ell}\neq\Delta^\dag_{2\ell-1}$, equality
(\ref{eq:rank1}) is equivalent to saying that there is no isolated
subgroup $\Delta^\dag$ of $\Gamma^\dag$ which is properly contained in
$\Delta^\dag_{2\ell-1}$ and properly contains
$\Delta^\dag_{2\ell}$. Suppose such an isolated subgroup $\Delta^\dag$
existed. Then
\begin{equation}
\Delta_\ell=\Delta^\dag_{2\ell}\cap\Gamma\subsetneqq\Delta^\dag\cap\Gamma
\subsetneqq\Delta^\dag_{2\ell-1}\cap\Gamma=\Delta_{\ell-1},
\label{eq:noninclusion}
\end{equation}
where the first inclusion is strict by the maximality of
$\Delta^\dag_{2\ell}$ and the second by the minimality of
$\Delta^\dag_{2\ell -1}$. Thus $\Delta^\dag\cap\Gamma$ is an isolated subgroup of
$\Gamma$, properly containing $\Delta_\ell$ and properly contained in
$\Delta_{l-1}$, which is impossible since $rk\
\frac{\Delta_{\ell-1}}{\Delta_\ell}=1$. This is a contradiction, hence
$rk\ \frac{\Delta^\dag_{2\ell -1}}{\Delta^\dag_{2\ell}}=1$, as desired.
\end{proof}
\begin{definition} Let $0\le i\le 2r$. The $i$-th prime ideal
  \textbf{determined} by $\nu^\dag_-$ is the prime $\nu^\dag_-$-ideal
  $\tilde H'_i$ of ${R'}^\dag$, corresponding to the isolated subgroup
  $\Delta^\dag_i$ (that is, the ideal $\tilde H'_i$ consisting of all
  the elements of ${R'}^\dag$ whose values lie outside of
  $\Delta^\dag_i$). The chain of trees (\ref{eq:chaintree''}) of prime
  ideals of ${R'}^\dag$ formed by the $\tilde H'_i$ is referred to as
  the chain of trees \textbf{determined} by $\nu^\dag_-$.
\end{definition}
The equality (\ref{eq:Delta}) says that
\begin{equation}
\tilde H'_{2\ell}\cap R'=\tilde H'_{2\ell+1}\cap R'=P'_\ell\label{eq:tildeHcapR}
\end{equation}
By definitions, for $1\le i\le2r$, $\nu^\dag_i$ is a valuation of the field
$k_{\nu^\dag_{i-1}}$. In the sequel, we will find it useful to talk
about the restriction of $\nu^\dag_i$ to a smaller field, namely, the
field of fractions of the ring
$\lim\limits_{\overset\longrightarrow{R'}}\frac{{R'}^\dag_{\tilde
    H'_i}}{\tilde H'_{i-1}{R'}^\dag_{\tilde H'_i}}$; we will denote
this restriction by $\nu^\dag_{i0}$. The field of fractions of
$\frac{{R'}^\dag_{\tilde H'_i}}{\tilde H'_{i-1}{R'}^\dag_{\tilde
    H'_i}}$ is $\kappa(\tilde H'_{i-1})$, hence that of
$\lim\limits_{\overset\longrightarrow{R'}}\frac{{R'}^\dag_{\tilde
    H'_i}}{\tilde H'_{i-1}{R'}^\dag_{\tilde H'_i}}$ is
$\lim\limits_{\overset\longrightarrow{R'}}\kappa(\tilde H'_{i-1})$,
which is a subfield of $k_{\nu^\dag_{i-1}}$. The value group of
$\nu^\dag_{i0}$ will be denoted by $\Delta_{i-1,0}$; we have
$\Delta_{i-1,0}\subset\frac{\Delta^\dag_{i-1}}{\Delta^\dag_i}$. If
$i=2\ell$ is even then
$\frac{R'_{P'_l}}{P'_{l-1}R'_{P'_l}}<\frac{{R'}^\dag_{\tilde
    H'_i}}{\tilde H'_{i-1}{R'}^\dag_{\tilde H'_i}}$, so
$\lim\limits_{\overset\longrightarrow{R'}}\frac{R'_{P'_l}}{P'_{l-1}R'_{P'_l}}<
\lim\limits_{\overset\longrightarrow{R'}}\frac{{R'}^\dag_{\tilde
    H'_i}}{\tilde H'_{i-1}{R'}^\dag_{\tilde H'_i}}$. In this case $rk\
\nu^\dag_i=1$ and $\nu^\dag_i$ an extension of the rank 1 valuation
$\nu_\ell$ from $\kappa(P_{\ell-1})$ to $k_{\nu^\dag_{i-1}}$; we have
\begin{equation}
\frac{\Delta_{\ell-1}}{\Delta_\ell}\subset\Delta_{i-1,0}\subset
\frac{\Delta^\dag_{i-1}}{\Delta^\dag_i}.\label{eq:deltalindelati-1}
\end{equation}
\begin{proposition}\label{necessary} Let $i=2\ell$. As usual, for an element
$\overline\beta\in\left(\frac{\Delta_\ell}{\Delta_{\ell+1}}\right)_+$,
let $\P_{\overline\beta}$ (resp. $\P'_{\overline\beta}$) denote
the preimage in $R$ (resp. in $R'$) of the $\nu_{\ell+1}$-ideal of
$\frac R{P_\ell}$ (resp. $\frac{R'}{P'_\ell}$) of value greater
than or equal to $\overline\beta$. Then
\begin{equation}
\bigcap\limits_{\overline\beta\in\left(\frac{\Delta_\ell}{\Delta_{\ell+1}}\right)_+}
\lim\limits_{\overset\longrightarrow{R'}}
\left(\P'_{\overline\beta}{R'}^\dag+\tilde
  H'_{i+1}\right){R'}^\dag_{\tilde H'_{i+2}}\cap{R}^\dag\subset\tilde
H_{i+1}.
\label{eq:restriction}
\end{equation}
\end{proposition}
The inclusion (\ref{eq:restriction}) should be understood as a
condition on the tree of ideals. In other words, it is equally valid
if we replace $R'$ by any other ring $R''\in\cal T$.

\begin{proof}\textit{(of Proposition \ref{necessary})} Since
$rk\frac{\Delta^\dag_{i+1}}{\Delta^\dag_{i+2}}=1$ by Lemma
\ref{rank1}, $\frac{\Delta_\ell}{\Delta_{\ell+1}}$ is cofinal in
$\frac{\Delta^\dag_{i+1}}{\Delta^\dag_{i+2}}$. Then for any
$x\in\bigcap\limits_{\overline\beta\in\left(\frac{\Delta_\ell}
{\Delta_{\ell+1}}\right)_+}\lim\limits_{\overset
\longrightarrow{R'}}\left(\P'_{\overline\beta}{R'}^\dag+\tilde
H'_{i+1}\right){R'}^\dag_{\tilde H'_{i+2}}\cap{R}^\dag$ we have
$\nu^\dag_-(x)\notin\Delta^\dag_i$, hence $x\in\tilde H_{i+1}$, as
desired.
\end{proof}

From now to the end of \S\ref{extensions}, we will assume that $\cal
T$ contains a stable ring $R'$, so that we can apply the results of
the previous section, in particular, the primality of the ideals
$H'_i$.
\begin{proposition}\label{Hintilde} We have
\begin{equation}
H'_i\subset\tilde H'_i\qquad\text{ for all
}i\in\{0,\dots,2r\}.\label{eq:Hintilde}
\end{equation}
\end{proposition}
\begin{proof} For $\beta\in\Gamma^\dag$ and $R'\in\cal T$, let
${\P'_\beta}^\dag$ denote the $\nu^\dag_-$-ideal of ${R'}^\dag$ of value
$\beta$. Fix an integer $\ell\in\{0,\dots,r\}$. For each $R'\in\cal
T$, each $\beta\in\Delta_\ell$ and $x\in\P'_\beta$ we have
$\nu^\dag_-(x)=\nu(x)\ge\beta$, hence
\begin{equation}
\P'_\beta{R'}^\dag\subset{\P'_\beta}^\dag.\label{eq:PinPdag}
\end{equation}
Taking the inductive limit over all $R'\in\cal T$ and the intersection
over all $\beta\in\Delta_\ell$ in (\ref{eq:PinPdag}), and using the
cofinality of $\Delta_\ell$ in $\Delta^\dag_{2\ell+1}$ and the fact
that $\bigcap\limits_{\beta\in\Delta^\dag_{2\ell}}
\left(\lim\limits_{\overset\longrightarrow{R'}}{\P'_\beta}^\dag\right)=
\lim\limits_{\overset\longrightarrow{R'}}\tilde H'_{2\ell+1}$, we
obtain the inclusion (\ref{eq:Hintilde}) for $i=2\ell+1$. To prove
(\ref{eq:Hintilde}) for $i=2\ell$, note that $\tilde H'_{2\ell}\cap
R'=\tilde H'_{2\ell+1}\cap R'=P_\ell$. By the same argument as in
Proposition \ref{H2l}, excellence of $R'$ implies that there is a
unique minimal prime $H^*_{2\ell}$ of $P'_\ell{R'}^\dag$, contained in
$\tilde H'_{2\ell+1}$ and a unique minimal prime $H^{**}_{2\ell}$ of
$P'_\ell{R'}^\dag$, contained in $\tilde H'_{2\ell}$. Now, Proposition
\ref{H2l} and the facts that $H'_{2\ell+1}\subset\tilde H'_{2\ell+1}$
and $\tilde H'_{2\ell}\subset\tilde H'_{2\ell+1}$ imply that
$H'_{2\ell}=H^*_{2\ell}=H^{**}_{2\ell}$, hence
$H'_{2\ell}=H^{**}_{2\ell}\subset\tilde H'_{2\ell}$, as desired.
\end{proof}
\begin{definition} A chain of trees (\ref{eq:chaintree''}) of prime ideals
  of ${R'}^\dag$ is said to be \textbf{admissible} if
  $H'_i\subset\tilde H'_i$ and (\ref{eq:tildeHcapR}) and
  (\ref{eq:restriction}) hold.
\end{definition}
Equalities (\ref{eq:tildeHcapR}), Proposition \ref{necessary} and
Proposition \ref{Hintilde} say that a chain of trees
(\ref{eq:chaintree''}) of prime ideals of ${R'}^\dag$, determined by
$\nu^\dag_-$, is admissible.

Summarizing all of the above results, and keeping in mind the fact
that specifying a composition of $2r$ valuation is equivalent to
specifying all of its $2r$ components, we arrive at one of the main
theorems of this paper:
\begin{theorem}\label{classification} Specifying the valuation $\nu^\dag_-$
  is equivalent to specifying the following data. The data will be
  described recursively in $i$, that is, the description of
  $\nu^\dag_i$ assumes that $\nu^\dag_{i-1}$ is already defined:

(1) An admissible chain of trees (\ref{eq:chaintree''}) of prime
ideals of ${R'}^\dag$.

(2) For each $i$, $1\le i\le 2r$, a valuation $\nu^\dag_i$ of
$k_{\nu^\dag_{i-1}}$ (where $\nu^\dag_0$ is taken to be the
trivial valuation by convention), whose restriction to
$\lim\limits_{\overset\longrightarrow{R'}}\kappa(\tilde H'_{i-1})$ is
centered at the local ring
$\lim\limits_{\overset\longrightarrow{R'}}\frac{{R'}^\dag_{\tilde
    H'_i}}{\tilde H'_{i-1}{R'}^\dag_{\tilde H'_i}}$.

The data $\left\{\nu^\dag_i\right\}_{1\le i\le 2r}$ is subject to the
following additional condition: if $i=2\ell$ is even then $rk\
\nu^\dag_i=1$ and $\nu^\dag_i$ is an extension of $\nu_\ell$ to
$k_{\nu^\dag_{i-1}}$ (which is naturally an extension of $k_{\nu_{\ell-1}}$).
\end{theorem}
In particular, note that such extensions $\nu^\dag_-$ always exist, and
usually there are plenty of them. The question of uniqueness of
$\nu^\dag_-$ and the related question of uniqueness of $\nu^\dag_i$,
especially in the case when $i$ is even, will be addressed in the next section.

\section{Uniqueness properties of $\nu^\dag_-$.}
\label{extensions}

In this section we address the question of uniqueness of the extension
$\nu^\dag_-$. One result in this direction, which will be very useful
here, was already proved in \S\ref{technical}: Proposition
\ref{nu0unique}. We give some necessary and some sufficient conditions
both for the uniqueness of $\nu^\dag_-$ once the chain
(\ref{eq:chaintree''}) of prime ideals determined by $\nu^\dag_-$ has
been fixed, and also for the unconditional uniqueness of
$\nu^\dag_-$. In \S\ref{henselization} we will use one of these
uniqueness criteria to prove uniqueness of $\nu^\dag_-$ in the cases
$R^\dag=\tilde R$ and $R^\dag=R^e$. At the end of this section we
generalize and give a new point of view of an old result of W. Heinzer
and J. Sally (Proposition \ref{HeSal}), which provides a sufficient
condition for the uniqueness of $\nu^\dag_-$; see also \cite{Te}, Remarks 5.22.

For a ring $R'\in\cal T$ let $K'$ denote the field of fractions of
$R'$. For some results in this section we will need to impose an
additional condition on the tree $\cal T$: we will assume that there
exists $R_0\in\cal T$ such that for all $R'\in{\cal T}(R_0)$ the field
$K'$ is algebraic over $K_0$. This assumption is needed in order to be
able to control the height of all the ideals in sight. Without loss of
generality, we may take $R_0=R$.
\begin{proposition}\label{htstable} Assume that for all $R'\in\cal T$ the
  field $K'$ is algebraic over $K$. Consider a ring homomorphism
  $R'\rightarrow R''$ in $\cal T$. Take an $\ell\in\{0,\dots,r\}$. We
  have
\begin{equation}
\he\ H''_{2\ell}\le \he\ H'_{2\ell}.\label{eq:odddecreases}
\end{equation}
If equality holds in (\ref{eq:odddecreases}) then
\begin{equation}
\he\ H''_{2\ell+1}\ge\he\ H'_{2\ell+1}.\label{eq:odddecreases1}
\end{equation}
\end{proposition}
\begin{proof} We start by recalling a well known Lemma (for a
proof see \cite{ZS}, Appendix 1, Propositions 2 and 3, p. 326):
\begin{lemma}\label{idealheight} Let $R\hookrightarrow R'$ be an
  extension of integral domains, essentially of finite type. Let $K$
  and $K'$ be the respective fields of fractions of $R$ and
  $R'$. Consider prime ideals $P\subset R$ and $P'\subset R'$ such
  that $P=P'\cap R$. Then
\begin{equation}
\he\ P'+tr.deg.(\kappa(P')/\kappa(P))\le\ \he\
P+tr.deg.(K'/K).\label{eq:heightdrops}
\end{equation}
Moreover, equality holds in (\ref{eq:heightdrops}) whenever $R$ is
universally catenarian.
\end{lemma}
Apply the Lemma to the rings $R'$ and $R''$ and the prime ideals
$P'_\ell\subset R'$ and $P''_\ell\subset R''$. In the case at hand we
have $tr.deg.(K''/K')=0$ by assumption. Hence
\begin{equation}
\he\ P''_\ell\le\ \he\ P'_\ell.\label{eq:htP}
\end{equation}
Since $H'_{2\ell}$ is a minimal prime of $P'_\ell{R'}^\dag$ and
${R'}^\dag$ is faithfully flat over $R'$, we have $ht\ P'_\ell=ht\
H'_{2\ell}$. Similarly, $\he\ P''_\ell=\he\ H''_{2\ell}$, and
(\ref{eq:odddecreases}) follows. Furthermore, equality in
(\ref{eq:odddecreases}) is equivalent to equality in (\ref{eq:htP}).

To prove (\ref{eq:odddecreases1}), let $\bar
R=(R''\otimes_{R'}{R'}^\dag)_{M''}$, where $M''=(m''\otimes1+1\otimes
m'{R'}^\dag)$ and let $\bar m$ denote the maximal ideal of $\bar
R$. We have the natural maps ${R'}^\dag\overset\iota\rightarrow\bar
R\overset\sigma\rightarrow{R''}^\dag$. The homomorphism $\sigma$ is
nothing but the formal completion of the local ring $\bar R$; in
particular, it is faithfully flat. Let
\begin{equation}
\bar H=H''_{2\ell+1}\cap\bar R,\label{eq:barH}
\end{equation}
$\bar H_0=H''_0\cap\bar R$. Since $H''_0$ is a minimal prime of
${R''}^\dag$ and $\sigma$ is faithfully flat, $\bar H_0$ is a minimal
prime of $\bar R$.

Assume that equality holds in (\ref{eq:odddecreases}) (and hence also
in (\ref{eq:htP})). Since equality holds in (\ref{eq:htP}), by Lemma
\ref{idealheight} (applied to the ring extension $R'\rightarrow R''$)
the field $\kappa(P'')$ is algebraic
over $\kappa(P')$.

Apply Lemma \ref{idealheight} to the ring extension
$\frac{{R'}^\dag}{H'_0}\hookrightarrow\frac{\bar R}{\bar H_0}$ and the
prime ideals $\frac{H'_{2\ell+1}}{H'_0}$ and $\frac{\bar H}{\bar
  H_0}$. Since $K''$ is algebraic over $K'$, $\kappa(\bar H_0)$ is
algebraic over $\kappa(H'_0)$. Since $\kappa(P'')$ is algebraic over
$\kappa(P')$, $\kappa(\bar H)$ is algebraic over
$\kappa(H'_{2\ell+1})$. Finally, $\hat R'$ is universally catenarian
because it is a complete local ring. Now in the case ${\ }^\dag=\hat{\
}$ Lemma \ref{idealheight} says that $\he\
\frac{H'_{2\ell+1}}{H'_0}=\he\ \frac{\bar H}{\bar H_0}$. Since both
$\hat R'$ and $\bar R$ are catenarian, this implies that
\begin{equation}
\he\ H'_{2\ell+1}=\he\ \bar H.\label{eq:htequal}
\end{equation}
In the case where ${\ }^\dag$ stands for henselization or a finite
\'etale extension, (\ref{eq:htequal}) is an immediate consequence of
(\ref{eq:barH}). Thus (\ref{eq:htequal}) is true in all the
cases. Since $\sigma$ is faithfully flat and in view of
(\ref{eq:barH}), $\he\bar H\le \he\ H''_{2\ell+1}$. Combined with
(\ref{eq:htequal}), this completes the proof.
\end{proof}
\begin{corollary}\label{htstable1} For each $i$, $0\le i\le 2r$, the
  quantity $\he\ H'_i$ stabilizes for $R'$ sufficiently far out in $\cal T$.
\end{corollary}
The next Proposition is an immediate consequence of Theorem
\ref{classification}.
\begin{proposition}\label{uniqueness2} Suppose given an admissible chain of
  trees (\ref{eq:chaintree''}) of prime ideals of ${R'}^\dag$. For
  each $\ell\in\{0,\dots,r-1\}$, consider the set of all $R'\in\cal T$
  such that
\begin{equation}
\he\ \tilde H'_{2\ell+1}- \he\ \tilde H'_{2_\ell} \le 1\qquad\text{
  for all even }i\label{eq:odd=even3}
\end{equation}
and, in case of equality, the 1-dimensional local ring
$\frac{{R'}^\dag_{\tilde H'_{2\ell+1}}}{\tilde
  H'_{2\ell}{R'}^\dag_{\tilde H'_{2\ell+1}}}$ is unibranch (that is,
analytically irreducible). Assume that for each $\ell$ the set of such
$R'$ is cofinal in $\cal T$. Let $\nu^\dag_{2\ell+1,0}$ denote the
unique valuation centered at
$\lim\limits_{\overset\longrightarrow{R'}}\frac{{R'}^\dag_{\tilde
    H'_{2\ell+1}}}{\tilde H'_{2\ell}{R'}^\dag_{\tilde H'_{2\ell+1}}}$.

Assume that for each even $i=2\ell$, $\nu_\ell$ admits a unique
extension $\nu^\dag_{i0}$ to a valuation of
$\lim\limits_{\overset\longrightarrow{R'}}\kappa(\tilde H'_{i-1})$,
centered in
$\lim\limits_{\overset\longrightarrow{R'}}\frac{{R'}^\dag_{\tilde
    H'_i}}{\tilde H'_{i-1}{R'}^\dag_{\tilde H'_i}}$. Then specifying
the valuation $\nu^\dag_-$ is equivalent to specifying for each odd $i$,
$2<i<2r$, an extension $\nu_i^\dag$ of the valuation
$\nu^\dag_{i0}$ of
$\lim\limits_{\overset\longrightarrow{R'}}\kappa(\tilde H'_{i-1})$ to
its field extension $k_{\nu^\dag_{i-1}}$ (in particular, such
extensions $\nu^\dag_-$ always exist). If for each odd $i$, $2<i<2r$,
the field extension
$\lim\limits_{\overset\longrightarrow{R'}}\kappa(\tilde
H'_{i-1})\rightarrow k_{\nu^\dag_{i-1}}$ is algebraic and the
extension $\nu_i^\dag$ of $\nu^\dag_{i0}$ to $k_{\nu^\dag_{i-1}}$ is
unique then there is a unique extension $\nu^\dag_-$ of $\nu$ such that
the $\tilde H'_i$ are the prime ideals, determined by $\nu^\dag_-$.

Conversely, assume that $K'$ is algebraic over $K$ and that there
exists a unique extension $\nu^\dag_-$ of $\nu$ such that the $\tilde
H'_i$ are the prime $\nu^\dag_-$-ideals, determined by $\nu^\dag_-$. Then
for each $\ell\in\{0,\dots,r-1\}$ and for all $R'$ sufficiently far
out in $\cal T$ the inequality (\ref{eq:odd=even3}) holds. For each even $i=2\ell$,
$\nu_\ell$ admits a unique extension $\nu^\dag_{i0}$ to a valuation of
$\lim\limits_{\overset\longrightarrow{R'}}\kappa\left(\tilde
  H'_{i-1}\right)$, centered in
$\lim\limits_{\overset\longrightarrow{R'}}\frac{{R'}^\dag_{\tilde
    H'_i}}{\tilde H'_{i-1}{R'}^\dag_{\tilde H'_i}}$; we have $rk\
\nu^\dag_{i0}=1$. For each odd $i$, the ring
$\lim\limits_{\overset\longrightarrow{R'}}\frac{{R'}^\dag_{\tilde
    H'_i}}{\tilde H'_{i-1}{R'}^\dag_{\tilde H'_i}}$ is a valuation
ring of a (not necessarily discrete) rank 1 valuation. For each odd $i$,
$1\le i<2r$, the field extension
$\lim\limits_{\overset\longrightarrow{R'}}\kappa(\tilde
H'_{i-1})\rightarrow k_{\nu^\dag_{i-1}}$ is algebraic and the
extension $\nu_i^\dag$ of $\nu^\dag_{i0}$ to $k_{\nu^\dag_{i-1}}$ is
unique.
\end{proposition}
\begin{remark} We do not know of a simple criterion to decide when, given
  an algebraic field extension $K\hookrightarrow L$ and a valuation
  $\nu$ of $K$, is the extension of $\nu$ to $L$ unique. See
  \cite{HOS}, \cite{V2} for more information about this question and an
  algorithm for arriving at the answer using MacLane's key
  polynomials.
\end{remark}
Next we describe three classes of extensions of $\nu$ to
$\lim\limits_{\overset\longrightarrow{R'}}{R'}^\dag$, which are of
particular interest for applications, and which we call
\textbf{minimal}, \textbf{evenly minimal} and \textbf{tight} extensions.
\begin{definition} Let $\nu^\dag_-$ be an extension of $\nu$ to
  $\lim\limits_{\overset\longrightarrow{R'}}{R'}^\dag$ and let the
  notation be as above. We say that $\nu^\dag_-$ is \textbf{evenly
    minimal} if whenever $i=2\ell$ is even, the following two conditions hold:

(1)
\begin{equation}
\Delta_{i-1,0}=\frac{\Delta_{\ell-1}}{\Delta_\ell}.\label{eq:groupequal}
\end{equation}

(2) For an element $\overline\beta\in\frac{\Delta_{\ell-1}}{\Delta_\ell}$,
the $\nu^\dag_{i0}$-ideal of $\frac{R^\dag_{\tilde H_i}}{\tilde
H_{i-1}R^\dag_{\tilde H_i}}$ of value $\overline\beta$, denoted by
$\P^\dag_{\overline\beta ,\ell}$, is given by the formula
\begin{equation}
\P^\dag_{\overline\beta,\ell}=\left(\lim\limits_{\overset\longrightarrow{R'}}
\frac{\P'_{\overline\beta}{R'}^\dag_{\tilde H'_i}}{\tilde
  H'_{i-1}{R'}^\dag_{\tilde H'_i}}\right)\cap\frac{R^\dag_{\tilde
  H_i}}{\tilde H_{i-1}R^\dag_{\tilde H_i}}.\label{eq:valideal}
\end{equation}
We say that $\nu^\dag_-$ is \textbf{minimal} if $\tilde H'_i=H'_i$
for each $R'$ and each $i\in\{0,\dots,2r+1\}$. We say that
$\nu^\dag_-$ is \textbf{tight} if it is evenly minimal and
\begin{equation}
\tilde H'_i=\tilde H'_{i+1}\qquad\text{ for all even }i.\label{eq:odd=even2}
\end{equation}
\end{definition}
\begin{remark} (1) The valuation $\nu^\dag_{i0}$ is uniquely determined by
  conditions (\ref{eq:groupequal}) and (\ref{eq:valideal}). Recall
  also that if $i=2\ell$ is even and we have:
\begin{eqnarray}
\tilde H'_i&=&H'_i\qquad\text{ and}\label{eq:tilde=H1}\\
\tilde H'_{i-1}&=&H'_{i-1}\label{eq:tilde=H2}
\end{eqnarray}
then $\nu^\dag_{i0}$ is uniquely determined by $\nu_\ell$ by
Proposition \ref{nu0unique}. In particular, if $\nu^\dag_-$ is minimal
(that is, if (\ref{eq:tilde=H1})--(\ref{eq:tilde=H2}) hold for all
$i$) then $\nu^\dag_-$ is evenly minimal.

(2) The definition of evenly minimal extensions can be rephrased as follows in terms of the associated graded algebras of $\nu_\ell$ and $\nu^\dag_{i0}$. First, consider a homomorphism in $\mathcal T$ and the following diagram, composed of natural homomorphisms:
\begin{equation}\label{eq:CDevminimal}
\xymatrix{{\frac{\P'_{\overline\beta}}{\P'_{\overline\beta+}}\otimes_{R'}{R'}^\dag_{\tilde H'_i}}\ar[r]^-{\lambda'}&{\frac{\P'_{\overline\beta}{R'}^\dag_{\tilde H'_i}}{\P'_{\overline\beta+}{R'}^\dag_{\tilde H'_i}}} \ar[r]^-{\phi'}& {\frac{{\P'_{\overline\beta}}^\dag}{{\P'}_{\overline\beta+}^\dag}}\\
{\ }&{\ }&{\frac{\P_{\overline\beta}^\dag}{\P_{\overline\beta+}^\dag}}\ar[u]_-{\psi'}}
\end{equation}
It follows from Nakayama's Lemma that equality (\ref{eq:valideal}) is equivalent to saying that
\begin{equation}\label{eq:equalgraded}
\frac{\P_{\overline\beta}^\dag}{\P_{\overline\beta+}^\dag}=\lim\limits_{\overset\longrightarrow{R'}}{\psi'}^{-1}\left((\phi'\circ\lambda')\left(\frac{\P'_{\overline\beta}}{\P'_{\overline\beta+}}\otimes_{R'}\kappa\left(\tilde H'_i\right)\right)\right).
\end{equation}
Taking the direct sum in (\ref{eq:equalgraded}) over all $\overline\beta\in\frac{\Delta_{\ell-1}}{\Delta_\ell}$ and passing to the limit on the both sides, we see that the extension $\nu^\dag_-$ is evenly minimal if and only if we have the following equality of graded algebras:
$$
\gr_{\nu_\ell}\left(\lim\limits_{\overset\longrightarrow{R'}}\frac{R'}{P'_\ell}\right)\otimes_{R'}\kappa\left(\tilde H'_i\right)=\gr_{\nu^\dag_{i0}}\left(\lim\limits_{\overset\longrightarrow{R'}}\frac{{R'}^\dag_{\tilde H'_i}}{\tilde
  H'_{i-1}{R'}^\dag_{\tilde H'_i}}\right).
$$
\end{remark}
\begin{examples} The extension $\hat\nu$ of Example \ref{Example31} (page \pageref{Example31}) is
minimal, but not tight. The valuation $\nu$ admits a unique tight
extension $\hat\nu_2\circ\hat\nu_3$ to
$\lim\limits_{\overset\longrightarrow{R'}}\frac{\hat R'}{H'_1}$; the
valuation $\hat\nu$ is the composition of the discrete rank 1
valuation $\hat\nu_1$, centered in
$\lim\limits_{\overset\longrightarrow{R'}}\hat R'_{H'_1}$ with
$\hat\nu_2\circ\hat\nu_3$.

The extension $\hat\nu^{(1)}$ of Example \ref{Example32} (page \pageref{Example32}) is minimal. The
extension $\hat\nu^{(2)}$ is evenly minimal but not  minimal.
Neither $\hat\nu^{(1)}$ nor $\hat\nu^{(2)}$ is tight. The valuation
$\nu$ admits a unique tight extension $\hat\nu_2\circ\hat\nu_3$ to
$\lim\limits_{\overset\longrightarrow{R'}}\frac{R'}{\tilde H'_1}$,
where $\tilde H'_1=\left(y-\sum\limits_{j=1}^\infty c_jx^j\right)$;
the valuation $\hat\nu^{(2)}$ is the composition of the discrete
rank 1 valuation $\hat\nu_1$, centered in
$\lim\limits_{\overset\longrightarrow{R'}}\hat R'_{\tilde H'_1}$
with $\hat\nu_2\circ\hat\nu_3$.
\end{examples}
\begin{remark} As of this moment, we do not know of any examples of
  extensions $\hat\nu_-$ which are not evenly minimal. Thus, formally,
  the question of whether every extension $\hat\nu_-$ is evenly minimal
  is open, though we strongly suspect that counterexamples do
  exist.
\end{remark}
\begin{proposition}\label{resmin} Let $i=2\ell$ be even and let
  $\nu^\dag_{i0}$ be the extension of $\nu_\ell$ to
  $\lim\limits_{\overset\longrightarrow{R'}}\kappa(\tilde H'_{i-1})$,
  centered at the local ring
  $\lim\limits_{\overset\longrightarrow{R'}}\frac{{R'}^\dag_{\tilde
      H'_i}}{\tilde H'_{i-1}{R'}^\dag_{\tilde H'_i}}$, defined by
  (\ref{eq:valideal}).
Then
\begin{equation}
k_{\nu^\dag_{i0}}=\lim\limits_{\overset\longrightarrow{R'}}\kappa(\tilde
H'_i).\label{eq:resfield}
\end{equation}
\end{proposition}
\begin{proof} Take two elements
$x,y\in\lim\limits_{\overset\longrightarrow{R'}}\frac{{R'}^\dag_{\tilde
    H'_i}}{\tilde H'_{i-1}{R'}^\dag_{\tilde H'_i}}$, such that
$\nu_{i0}^\dag(x)=\nu_{i0}^\dag(y)$. We must show that the image of
$\frac xy$ in $k_{\nu^\dag_{i0}}$ belongs to
$\lim\limits_{\overset\longrightarrow{R'}}\kappa(\tilde
H'_i)$. Without loss of generality, we may assume that
$x,y\in\frac{R^\dag_{\tilde H_i}}{\tilde H_{i-1}{R}^\dag_{\tilde
    H_i}}$. Let $\beta=\nu_{i0}^\dag(x)=\nu_{i0}^\dag(y)$. Choose
$R'\in\cal T$ sufficiently far out in the direct system so that
$x,y\in\frac{\P'_\beta{R'}^\dag_{\tilde H'_i}}{\tilde
  H'_{i-1}{R'}^\dag_{\tilde H'_i}}$. Let $R'\rightarrow R''$ be the
blowing up of the ideal $\P'_\beta R'$. Then in
$\frac{\P''_\beta{R''}^\dag_{\tilde H''_i}}{\tilde
  H''_{i-1}{R''}^\dag_{\tilde H''_i}}$ we can write
\begin{eqnarray}
x&=&az\qquad\text{ and }\label{eq:xaz}\\
y&=&aw,
\end{eqnarray}
where $\nu_{i0}^\dag(a)=\beta$ and
$\nu_{i0}^\dag(z)=\nu_{i0}^\dag(w)=0$. Let $\bar z$ be the image of
$z$ in $\kappa(\tilde H''_i)$ and similarly for $\bar w$. Then the
image of $\frac xy$ in $k_{\nu^\dag_{i0}}$ equals $\frac{\bar z}{\bar
  w}\in\kappa(\tilde H''_i)$, and the result is proved.
  \end{proof}
\begin{remark}\label{minimalexist} Theorem \ref{classification} and the existence of the
  extension $\nu^\dag_{2\ell,0}$ of $\nu_\ell$ in the case when
  $\tilde H'_{2\ell}=H'_{2\ell}$ and $\tilde
  H'_{2\ell-1}=H'_{2\ell-1}$ guaranteed by Theorem \ref{primality1}
  (2) allow us to give a fairly explicit description of the totality
  of minimal extensions as compositions of $2r$ valuations and, in
  particular, to show that they always exist. Indeed, minimal
  extensions $\nu^\dag_-$ can be constructed at will, recursively in
  $i$, as follows. Assume that the valuations
  $\nu^\dag_1,\dots,\nu^\dag_{i-1}$ are already constructed. If $i$ is
  odd, let $\nu^\dag_i$ be an arbitrary valuation of the residue field
  $k_{\nu^\dag_{i-1}}$ of the valuation ring $R_{\nu^\dag_{i-1}}$. If
  $i=2\ell$ is even, let $\nu^\dag_{i0}$ be the extension of $\nu_\ell$ to
  $\lim\limits_{\overset\longrightarrow{R'}}\kappa(H'_{i-1})$,
  centered at the local ring
  $\lim\limits_{\overset\longrightarrow{R'}}
\frac{{R'}^\dag_{H'_i}}{H'_{i-1}{R'}^\dag_{H'_i}}$, whose existence
and uniqueness are guaranteed by Theorem \ref{primality1} (2) and
Proposition \ref{nu0unique}, respectively. Let $\nu^\dag_i$ be an
arbitrary extension of $\nu^\dag_{i0}$ to the field
$k_{\nu^\dag_{i-1}}$. It is clear that all the minimal extensions
$\nu^\dag_-$ of $\nu$ are obtained in this way.

In the next section we will use this remark to show that if
$R^\dag=\tilde R$ or $R^\dag=R^e$ then $\nu$ admits a unique extension
to $\frac{R^\dag}{H_0}$, which is necessarily minimal.
\end{remark}
We end this section by giving some sufficient conditions for the
uniqueness of $\nu^\dag_-$.
\begin{proposition}\label{uniqueness1} Suppose given an admissible chain of
  trees (\ref{eq:chaintree''}) of prime ideals of ${R'}^\dag$. For
  each $\ell\in\{0,\dots,r-1\}$, consider the set of all $R'\in\cal T$ such that
\begin{equation}
ht\ \tilde H'_{2\ell+1}- ht\ \tilde H'_{2\ell} \le 1\qquad\text{ for
  all even }i\label{eq:odd=even1}
\end{equation}
and, in case of equality, the 1-dimensional local ring $\frac{{R'}^\dag_{\tilde
    H'_{2\ell+1}}}{\tilde H'_{2\ell}{R'}^\dag_{\tilde H'_{2\ell+1}}}$
is unibranch (that is, analytically irreducible). Assume that for each
$\ell$ the set of such $R'$ is cofinal in $\cal T$.

Let $\nu^\dag_-$ be an extension of $\nu$ such that the $\tilde H'_i$
are prime $\nu^\dag_-$-ideals. Assume that $\nu^\dag_-$ is evenly
minimal. Then there is at most one such extension $\nu^\dag_-$ and
exactly one such $\nu^\dag_-$ if
\begin{equation}
\tilde H'_i=H'_i\quad\text{ for all }i.\label{eq:tildeH=H}
\end{equation}
(in the latter case $\nu^\dag_-$ is minimal by definition).
\end{proposition}
\begin{proof} By Theorem
\ref{primality1} (2) and Proposition \ref{nu0unique}, if
(\ref{eq:tildeH=H}) holds then $\nu^\dag_-$ is minimal and for each even
$i$ the extension $\nu^\dag_{i0}$ exists and is unique. Therefore we
may assume that in all the cases $\nu^\dag_-$ is evenly minimal and that
$\nu^\dag_{i0}$ exists whenever (\ref{eq:tildeH=H}) holds.

The valuation $\nu^\dag_-$, if it exists, is a composition of $2r$ valuations:
$\nu^\dag_-=\nu^\dag_1\circ\nu^\dag_2\circ\dots\circ\nu^\dag_{2r}$,
subject to the conditions of Theorem \ref{classification}. We prove
the uniqueness of $\nu^\dag_-$ by induction on $r$. Assume the result is
true for $r-1$. This means that there is at most one evenly minimal
extension $\nu^\dag_3\circ\nu^\dag_4\circ\dots\circ\nu^\dag_{2r}$ of
$\nu_2\circ\nu_3\circ\dots\circ\nu_r$ to
$\lim\limits_{\overset\longrightarrow{R'}}\kappa(\tilde H'_2)$, and
exactly one in the case when (\ref{eq:tildeH=H}) holds. To complete
the proof of uniqueness of $\nu^\dag_-$, it is sufficient to show that
both $\nu_1^\dag$ and $\nu_2^\dag$ are unique and that the residue
field of $\nu_2^\dag$ equals
$\lim\limits_{\overset\longrightarrow{R'}}\kappa(\tilde H'_2)$.

We start with the uniqueness of $\nu_1^\dag$. If (\ref{eq:odd=even2})
holds then $\nu_1^\dag$ is the trivial valuation. Suppose, on the
other hand, that equality holds in (\ref{eq:odd=even1}). Then the
restriction of $\nu_1^\dag$ to each $R'\in\cal T$ such that the local
ring $\lim\limits_{\overset\longrightarrow{R'}}\frac{\hat R'_{\tilde
    H'_1}}{\tilde H'_0\hat R'_{\tilde H'_1}}$ is one-dimensional and
unibranch is the unique divisorial valuation centered in that ring (in
particular, its residue field is $\kappa(\tilde H'_1)$). By the assumed
cofinality of such $R'$, the valuation $\nu^\dag_1$ is unique and its
residue field equals
$\lim\limits_{\overset\longrightarrow{R'}}\kappa(\tilde H'_1)$. Thus,
regardless of whether or not the inequality in (\ref{eq:odd=even1}) is
strict, $\nu^\dag_1$ is unique and we have the equality of residue
fields
\begin{equation}
k_{\nu_1^\dag}=\lim\limits_{\overset\longrightarrow{R'}}\kappa(\tilde
H'_1)\label{eq:eqres}
\end{equation}
This equality implies that $\nu_2^\dag=\nu^\dag_{20}$. Now, the valuation
$\nu^\dag_2=\nu^\dag_{20}$ is uniquely determined by the conditions
(\ref{eq:groupequal}) and (\ref{eq:valideal}), and its residue field is
\begin{equation}
k_{\nu^\dag_2}=k_{\nu^\dag_{20}}=\lim\limits_{\overset\longrightarrow{R'}}
\kappa(\tilde H'_2).\label{eq:resfield1}
\end{equation}
by Proposition \ref{resmin}. Furthermore, by Theorem \ref{primality1}
exactly one such $\nu_2^\dag$ exists whenever (\ref{eq:tildeH=H})
holds. This proves that there is at most one possibility for
$\nu^\dag_-$: the composition of $\nu_1^\dag\circ\nu^\dag_2$ with
$\nu^\dag_3\circ\nu^\dag_4\circ\dots\circ\nu^\dag_{2r}$, and exactly
one if (\ref{eq:tildeH=H}) holds.
\end{proof}
\begin{proposition}\label{tight=scalewise} The extension $\nu^\dag_-$ is tight if and only if for each $R'$ in our direct
system the natural graded algebra extension $\gr_\nu
R'\rightarrow\gr_{\nu^\dag_-}{R'}^\dag$ is scalewise birational.
\end{proposition}
\begin{remark}\label{rephrasing} Proposition \ref{tight=scalewise} allows us
  to rephrase Conjecture \ref{teissier} as follows: the valuation $\nu$
  admits at least one tight extension $\nu^\dag_-$.
\end{remark}
\begin{proof}\textit{(of Proposition \ref{tight=scalewise})} ``If'' Assume that for each $R'$ in our direct
system the natural graded algebra extension $\gr_\nu
R'\rightarrow\gr_{\nu^\dag_-}{R'}^\dag$ is scalewise birational. Then
\begin{equation}\label{eq:gamma=dag}
\Gamma^\dag=\Gamma.
\end{equation}
Together with (\ref{eq:Delta}) this implies that for each
$l\in\{1,\dots,r+1\}$ we have
$\Delta^\dag_{2\ell-2}=\Delta^\dag_{2\ell-1}=\Delta_{\ell-1}$ under
the identification (\ref{eq:gamma=dag}). Then
$\frac{\Delta^\dag_{2\ell-1}}{\Delta^\dag_{2\ell}}=(0)$, so for all
odd $i$ the valuation $\nu^\dag_i$ is trivial. This proves the
equality (\ref{eq:odd=even2}) in the definition of tight. It remains
to show that $\nu^\dag_-$ is evenly minimal. We will prove the even
minimality in the form of equality (\ref{eq:equalgraded}) for each
$\bar\beta\in\frac{\Delta_{\ell-1}}{\Delta_\ell}$. The right hand
side of (\ref{eq:equalgraded}) is trivially contained in the left
hand side; we must prove the opposite inclusion. To do that, take a
non-zero element
$x\in\frac{\P_{\overline\beta}^\dag}{\P_{\overline\beta+}^\dag}$. By
scalewise birationaliy, there exist non-zero elements $\bar y,\bar
z\in\gr_\nu R$, with $ord\ \bar y,ord\ \bar z\in\Delta_\ell$, such
that $x\bar y=\bar z$. Let $y$ be a representative of $\bar y$ in
$R$, and similarly for $z$. Let $R\rightarrow R'$ be the local
blowing up with respect to $\nu$ along the ideal $(y,z)$. Then, in
the notation of (\ref{eq:equalgraded}), we have
\begin{equation}\label{eq:equalgraded1}
x={\psi'}^{-1}\left(\left(\phi'\circ\lambda'\right)\left(\frac{\bar z}{\bar y}\otimes_{R'}1\right)\right)\in{\psi'}^{-1}\left((\phi'\circ\lambda')\left(\frac{\P'_{\overline\beta}}{\P'_{\overline\beta+}}\otimes_{R'}\kappa\left(\tilde H'_i\right)\right)\right).
\end{equation}
This proves (\ref{eq:equalgraded}). ``If'' is proved.

``Only if''. Assume that $\nu^\dag_-$ is tight (that is, it is evenly minimal and
(\ref{eq:odd=even2}) holds) and take $R'\in\cal T$. Then the valuation
$\nu_{2\ell+1}$ is trivial for all $\ell$, so
$\nu^\dag_-=\nu^\dag_2\circ\nu^\dag_4\circ\dots\circ\nu^\dag_{2r}$. We
must show that the graded algebra extension $\gr_\nu
R'\rightarrow\gr_{\nu^\dag_-}{R'}^\dag$ is scalewise birational. Again,
we use induction on $r$. Take an element $x\in{R'}^\dag$. If
$\nu^\dag_-(x)\in\Delta_1$ then
$\init_{\nu^\dag_-}x\in\gr_{\nu^\dag_4\circ\dots\circ\nu^\dag_{2r}}
\frac{{R'}^\dag}{\tilde H'_2}$, hence by the induction assumption
there exists $y\in R'$ with $\nu^\dag_-(x)\in\Delta_1$ and
$\init_{\nu^\dag_-}(xy)\in\gr_\nu\frac{R'}{P_1}$. In this case, there is
nothing more to prove. Thus we may assume that
$\nu^\dag_-(x)\notin\Delta_1$. It remains to show that there exists
$y\in R'$ such that $\init_{\nu^\dag_-}(xy)\in\gr_\nu R'$. Since the natural map sending each element of the ring to its image in the graded algebra behaves well with respect to multiplication and division, local
blowings up induce birational transformations of graded algebras, and it
is enough to find a local blowing up $R''\in{\cal T}(R')$ and $y\in
R''$ such that $\init_{\nu^\dag_-}(xy)\in\gr_\nu R''$.

Now, Proposition \ref{resmin} shows that there exists a local blowing
up $R'\rightarrow R''$ such that $x=az$ (\ref{eq:xaz}), with $z\in
R''$ and $\nu^\dag_2(a)=\nu^\dag_{2,0}(a)=0$. The last equality means
that $\nu^\dag_-(a)\in\Delta_1$, and the result follows from the
induction assumption, applied to $a$.
\end{proof}

The argument above also shows the following. Let
${\Phi'}^\dag=\nu^\dag_-\left({R'}^\dag\setminus\{0\}\right)$, take an
element $\beta\in{\Phi'}^\dag$ and let ${\cal P'}^\dag_\beta$ denote
the $\nu^\dag_-$-ideal of ${R'}^\dag$ of value $\beta$.
\begin{corollary}\label{blup1} Take an element $x\in{\cal P'}^\dag_\beta$. There
exists a local blowing up $R'\rightarrow R''$ such that
$\beta\in\nu(R'')\setminus\{0\}$ and $x\in{\cal P}''_\beta{R''}^\dag$.
\end{corollary}
The next Proposition gives a sufficient condition for the uniqueness
of $\nu^\dag_-$ (this result is due to Heinzer and Sally \cite{HeSa}).
\begin{proposition}\label{HeSal} Assume that $K'$ is algebraic over $K$ for
  all $R'\in\cal T$ and that the following conditions hold:

(1) $ht\ H'_1\le1$

(2) $ht\ H'_1+rat.rk\ \nu=\dim\ R'$, where $R'$ is taken to be
sufficiently far out in the direct system.

Let $\nu^\dag_-$ be an extension of $\nu$ to a ring of the form
$\lim\limits_{\overset\longrightarrow{R'}}\frac{{R'}^\dag}{\tilde
  H'_0}$. Then either
\begin{eqnarray}
\tilde H'_0&=&H'_0\qquad\text{ or }\label{eq:tildeH=H0}\\
\tilde H'_0&=&H'_1.\label{eq:tildeH=H1}
\end{eqnarray}
The valuation $\nu$ admits a unique extension to
$\lim\limits_{\overset\longrightarrow{R'}}\frac{{R'}^\dag}{H'_0}$ and
a unique extension to
$\lim\limits_{\overset\longrightarrow{R'}}\frac{{R'}^\dag}{H'_1}$. The
first extension is minimal and the second is tight.
\end{proposition}
\begin{proof} For $1\le\ell\le r$, let $r_\ell$ denote the rational
rank of $\nu_\ell$. Let $\nu^\dag_-$ be an extension of $\nu$ to a ring
of the form
$\lim\limits_{\overset\longrightarrow{R'}}\frac{{R'}^\dag}{\tilde
  H'_0}$, where $\tilde H'_0$ is a tree of prime ideals of ${R'}^\dag$
such that $\tilde H'_0\cap R'=(0)$.

By Corollary \ref{htstable1} $ht\ H'_i$ stabilizes for $1\le i\le 2r$
and $R'$ sufficiently far out in the direct system. From now on, we
will assume that $R'$ is chosen sufficiently far so that the stable
value of $ht\ H'_i$ is attained. Now, let $i=2\ell$. The valuation
$\nu^\dag_{i0}$ is centered in the local noetherian ring
$\frac{{R'}^\dag_{\tilde H'_i}}{\tilde H'_{i-1}{R'}^\dag_{\tilde
    H'_i}}$, hence by Abhyankar's inequality
\begin{equation}
rat.rk\ \nu^\dag_{i0}\le\dim\frac{{R'}^\dag_{\tilde H'_i}}{\tilde
  H'_{i-1}{R'}^\dag_{\tilde H'_i}}\le ht\ \tilde H'_i-ht\ \tilde
H'_{i-1}.\label{eq:abhyankar}
\end{equation}
Since this inequality is true for all even $i$, summing over all $i$ we obtain:
\begin{equation}
\begin{array}{rcl}
\dim R'&=&\dim\ {R'}^\dag=\sum\limits_{i=1}^{2r}(ht\ \tilde H'_i-ht\
\tilde H'_{i-1})\ge ht\ \tilde H'_1+\sum\limits_{\ell=1}^r(ht\ \tilde
H'_{2\ell}-ht\ \tilde H'_{2\ell-1})\ge\\
&\ge&ht\ H'_1+\sum\limits_{\ell=1}^rrat.rk\ \nu^\dag_{2\ell,0}\ge ht\
H'_1+\sum\limits_{\ell=1}^rr_\ell=ht\ H'_1+rat.rk\ \nu=\dim
R'.\label{eq:inequalities}
\end{array}
\end{equation}
Hence all the inequalities in (\ref{eq:abhyankar}) and
(\ref{eq:inequalities}) are equalities. In particular, we
have
$$
ht\ \tilde H'_1=ht\ H'_1;
$$
combined with Proposition \ref{Hintilde} this shows that
\begin{equation}
\tilde H'_1=H'_1.
\end{equation}
Together with the hypothesis (1) of the Proposition, this already
proves that at least one of (\ref{eq:tildeH=H0})--(\ref{eq:tildeH=H1})
holds. Furthermore, equalities in (\ref{eq:abhyankar}) and
(\ref{eq:inequalities}) prove that
$$
ht\ \tilde H'_i=ht\ \tilde H'_{i-1}
$$
for all odd $i>1$, so that
\begin{equation}
\tilde H'_i=\tilde H'_{i-1}\qquad\text{whenever $i>1$ is
  odd}\label{eq:oddiseven}
\end{equation}
and that
\begin{equation}
r_i=ht\ \tilde H'_i-ht\ \tilde H'_{i-1}\label{heightfixed}
\end{equation}
whenever $i$ is even.

Now, consider the special case when $\tilde H'_i=H'_i$ for $i\ge1$ and
$\tilde H'_0$ is as in
(\ref{eq:tildeH=H0})--(\ref{eq:tildeH=H1}). According to Proposition
\ref{nu0unique} for each even $i=2\ell$ there exists a unique
extension $\nu^\dag_{i0}$ of $\nu_l$ to a valuation of
$\lim\limits_{\overset\longrightarrow{R'}}\kappa(H'_{i-1})$, centered
in the local ring
$\lim\limits_{\overset\longrightarrow{R'}}
\frac{{R'}^\dag_{H'_{2\ell}}}{H'_{2\ell-1}}$. Moreover, we have
\begin{equation}
k_{\nu^\dag_{2\ell,0}}=\lim\limits_{\overset\longrightarrow{R'}}\kappa(H'_{2\ell})
\label{eq:knudag2l}
\end{equation}
by Remark \ref{sameresfield}. By Theorem \ref{classification}, there
exists an extension $\nu^\dag_-$ of $\nu$ to
$\lim\limits_{\overset\longrightarrow{R'}}\frac{{R'}^\dag}{\tilde
  H'_0}$ such that the $\{\tilde H'_i\}$ as above is the chain of
trees of prime ideals, determined by $\nu^\dag_-$. In particular,
(\ref{eq:oddiseven}) and (\ref{heightfixed}) hold with $\tilde H'_i$
replaced by $H'_i$.

Now (\ref{heightfixed}) and Proposition \ref{Hintilde} imply that for
\textit{any} extension $\nu^\dag_-$ we have $\tilde H'_i=H'_i$ for
$i>0$, so that the special case above is, in fact, the only case
possible. Furthermore, by (\ref{eq:oddiseven}) we have
$H'_{2\ell+1}=H'_{2\ell}$ for all $\ell\in\{1,\dots,r\}$. This implies
that for all such $\ell$ the valuation $\nu^\dag_{2\ell+1,0}$ is the
trivial valuation of
$\lim\limits_{\overset\longrightarrow{R'}}\kappa(H'_{2\ell})$; in particular,
\begin{equation}
k_{\nu^\dag_{2\ell+1,0}}=\lim\limits_{\overset\longrightarrow{R'}}\kappa(H'_{2\ell})
\label{eq:knudag2l+1}
\end{equation}
for all $\ell\in\{1,\dots,r-1\}$.

If $\tilde H'_0=H'_1=\tilde H'_1$ then the only possibility for
$\nu^\dag_{10}=\nu^\dag_1$ is the trivial valuation of
$\lim\limits_{\overset\longrightarrow{R'}}\kappa(H'_1)$; we have
\begin{equation}
k_{\nu^\dag_1}=k_{\nu^\dag_{10}}=\lim\limits_{\overset\longrightarrow{R'}}
\kappa(H'_1).\label{eq:knudag10}
\end{equation}
If $\tilde H'_0=H'_0$ then by the hypothesis (1) of the Proposition
and the excellence of $R$ the ring $\frac{{R'}^\dag_{H'_1}}{H'_0{R'}^\dag_{H'_1}}$ is
a regular one-dimensional local ring (in particular, unibranch), hence
the valuation $\nu^\dag_1=\nu^\dag_{10}$ centered at
$\lim\limits_{\overset\longrightarrow{R'}}
\frac{{R'}^\dag_{H'_1}}{H'_0{R'}^\dag_{H'_1}}$ is unique and
(\ref{eq:knudag10}) holds also in this case.

By induction on $i$, it follows from (\ref{eq:knudag2l}),
(\ref{eq:knudag2l+1}), the uniqueness of $\nu^\dag_{2\ell,0}$ and the
triviality of $\nu^\dag_{2\ell+1,0}$ for $\ell\ge1$ that $\nu_i^\dag$
is uniquely determined for all $i$ and
$k_{\nu^\dag_i}=\lim\limits_{\overset\longrightarrow{R'}}\kappa(H'_i)$. This
proves that in both cases (\ref{eq:tildeH=H0}) and
(\ref{eq:tildeH=H1}) the valuation
$\nu^\dag_-=\nu^\dag_1\circ\dots\circ\nu^\dag_{2r}$ is unique. The last
statement of the Proposition is immediate from definitions.
\end{proof}

A related necessary condition for the uniqueness of $\nu^\dag_-$ will be
proved in \S\ref{locuni1}.

\section{Extending a valuation centered in an excellent local domain
  to its henselization.}
\label{henselization}

Let $\tilde R$ denote the henselization of $R$, as above. The
completion homomorphism $R\rightarrow\hat R$ factors through the
henselization: $R\rightarrow\tilde R\rightarrow\hat R$. In this
section, we will show that $H_1$ is a minimal prime of $\tilde R$,
that $\nu$ extends uniquely to a valuation $\tilde\nu_-$ of rank $r$
centered at $\frac{\tilde R}{H_1}$, and that $H_1$ is the unique
prime ideal $P$ of $\tilde R$ such that $\nu$ extends to a valuation
of $\frac{\tilde R}P$. Furthermore, we will prove that $H_{2\ell+1}$
is a minimal prime of $P_\ell\tilde R$ for all $\ell$ and that these
are precisely the prime $\tilde\nu$-ideals of $\tilde R$.

Studying the implicit prime ideals of $\tilde R$ and the extension of $\nu$ to
$\tilde R$ is a logical intermediate step before attacking the formal
completion, for the following reason. As we will show in the next section, if
$R$ is already henselian in (\ref{eq:defin1}) then $\P'_\beta\hat
R'_{H'_{2\ell+1}}\cap\hat R=\P_\beta\hat R$ for all $\beta$ and $R'$
and thus we have
$H_{2\ell+1}=\bigcap\limits_{\beta\in\Delta_{\ell}}\left({\cal
    P}_\beta\hat R\right)$.

We state the main result of this section. In the case when $R^e$ is an
\'etale extension of $R$, contained in $\tilde R$, we use
(\ref{eq:defin3}) with $R^\dag=R^e$ as our definition of the implicit
prime ideals.
\begin{theorem}\label{hensel0} Let $R^e$ be a local \'etale extension of
  $R$, contained in $\tilde R$. Then:

(1) The ideal $H_{2\ell+1}$ is prime for $0\le l\le r$; it is a minimal prime of
$P_\ell R^e$. In particular, $H_1$ is a minimal prime of $R^e$. We
have $H_{2\ell}=H_{2\ell+1}$ for $0\le l\le r$.

(2) The ideal $H_1$ is the unique prime $P$ of $R^e$ such that there
    exists an extension $\nu^e_-$ of $\nu$ to $\frac{R^e}P    $; the
    extension $\nu^e_-$ is unique. The graded algebra
    $\gr_{\nu^e_-}\frac{R^e}{H_1}$ is scalewise birational to $\gr_\nu
    R$; in particular, $rk\ \nu^e_-=r$.

(3) The ideals $H_{2\ell+1}$ are precisely the prime $\nu^e_-$-ideals of
    $R^e$.
\end{theorem}
\begin{proof} By assumption, the ring $R^e$ is a direct limit of
local, strict \'etale extensions of $R$ which are essentially of finite
type. All the assertions (1)--(3) behave well under taking direct limits, so it
is sufficient to prove the Theorem in the case when $R^e$ is essentially of
finite type over $R$. From now on, we will restrict attention to this case.

The next step is to describe explicitly those local blowings up
$R\rightarrow R'$ for which $R'$ is $\ell$-stable. Their interest to us is that, according to Proposition \ref{largeR2},
if $R'$ is $\ell$-stable then for all $R''\in{\cal T}(R')$ and all $\beta\in\frac{\Delta_\ell}{\Delta_{\ell+1}}$, we have the
equality
\begin{equation}
{\cal P}''_\beta(R''\otimes_RR^e)\cap R^e={\cal P}_\beta R^e;\label{eq:contracts}
\end{equation}
in particular, the limit in (\ref{eq:defin3}) is attained, that is, we have the equality
\begin{equation}
H_{2\ell+1}=\bigcap\limits_{\beta\in\Delta_\ell}\left(\left({\cal P}'_\beta\left(R^e\otimes_RR'\right)_{M'}\right)
\bigcap R^e\right).\label{eq:defin4}
\end{equation}
\begin{lemma}\label{lift} Let $\frac{R}{P_\ell}\to T$ be a finitely generated extension of $\frac {R}{P_\ell}$, contained in
$\frac{R_\nu}{\bf m_\ell}$. Let
$$
{\bf q}=\frac{\bf m_\nu}{\bf m_\ell}\cap T.
$$
There exists a $\nu$-extension $R\to R'$ of $R$ such that $\frac{R'}{P'_\ell}=T_{\bf q}$.
\end{lemma}

\begin{proof} Write $T=\frac R{P_\ell}\left[\overline a_1,\ldots,\overline
a_k\right]$, with $\overline a_i\in\frac{R_\nu}{\bf m_\ell}$, that
is, $\nu_{\ell+1}\left(\overline a_i\right)\geq 0,\ 1\leq i\leq
k$. We can lift the $\overline a_i$ to elements $a_i$ in $R_\nu$
such that $\nu\left(a_i\right)\geq0$. Let us consider the ring
$R''=R\left[a_1,\ldots,a_k\right]\subset R_\nu$ and its
localization $R'=R''_{{\bf m}_\nu\cap R''}$. The ideal $P'_\ell$
is the kernel of the natural map $R'\rightarrow\frac{R_\nu}{\bf
m_\ell}$. Thus both $\frac{R'}{P'_\ell}$ and $T_{\bf q}$ are equal
to the $\frac R{P_l}$-subalgebra of $\frac{R_\nu}{\bf m_\ell}$,
obtained by adjoining $\overline a_1,\ldots,\overline a_k$ to
$\frac R{P_l}$ inside $\frac{R_\nu}{\bf m_\ell}$ and then
localizing at the preimage of the ideal $\frac{\bf m_\nu}{\bf
m_\ell}$. This proves the Lemma.
\end{proof}

Let us now go back to our \'etale extension $R\to R^e$.
\begin{lemma}\label{anirred1} Fix an integer $l\in\{0,\dots,r\}$. There exists a
  local blowing up $R\rightarrow R'$ along $\nu$ having the following
  property: let $P'_\ell$ denote the $\ell$-th prime $\nu$-ideal of $R'$. Then
  the ring $\frac{R'}{P'_\ell}$ is analytically irreducible; in particular, $\frac{R'}{P'_\ell}\otimes_R R^e$ is an integral
  domain.
\end{lemma}
\begin{remark} We are not claiming that there exists $R'\in\cal T$ such that $\frac{R'}{P'_\ell}$ is analytically
irreducible for all $\ell$ (and we do not know how to prove such a
claim), only that for each $\ell$ there exists an $R'$, which may
depend on $\ell$, such that $\frac{R'}{P'_\ell}$ is analytically
irreducible. On the other hand, below we will prove that there exists an $\ell$-stable $R'\in\cal T$. According to
Definition \ref{stable} (2) and Proposition \ref{largeR1}, such a stable $R'$ has the property that
$\kappa\left(P''_\ell\right)\otimes_R\left(R''\otimes_RR^e\right)_{M''}$ is a domain for all $R''\in{\cal T}(R')$. For a given $R''$, this property is weaker than the analytic irreducibility of $R''/P''_\ell$. The latter is equivalent to saying that 
$\kappa(P''_\ell)\otimes_R(R''\otimes_RR^\sharp)_{M''}$ is a domain for every local \'etale extension $R^\sharp$ of $R''$.
\end{remark}
\begin{proof}\textit{(of Lemma \ref{anirred1})} Since $R$ is an excellent local ring, every homomorphic image of $R$ is Nagata
\cite{Mat} (Theorems 72 (31.H), 76 (33.D) and 78 (33.H)). Let $\pi:\frac R{P_\ell}\rightarrow S$ be the normalization of
$\frac {R}{P_\ell}$. Then $S$ is a finitely generated $\frac {R}{P_\ell}$-algebra contained in $\frac{R_\nu}{\bf m_\ell}$,
to which we can apply Lemma \ref{lift}. We obtain a $\nu$-extension $R\to R'$ such that the ring
$\frac{R'}{P'_\ell}\cong\frac{R'}{P_\ell R'}$ is a localization of $S$ at a prime ideal, hence it is an excellent normal
local ring. In particular, it is analytically irreducible (\cite{Nag}, Theorem (43.20), p. 187 and Corollary (44.3), p. 189), as desired.
\end{proof}

Next, we fix $\ell\in\{0,\dots,r\}$ and study the ring $(T')^{-1}(\kappa(P'_\ell)\otimes_RR^e)$, in particular, the structure of
the set of its zero divisors, as $R'$ runs over ${\cal T}(R)$ (here $T'$ is as in Remark \ref{interchanging}). Since $R^e$
is separable algebraic, essentially of finite type over $R$, the ring $(T')^{-1}(\kappa(P'_\ell)\otimes_RR^e)$ is finite over
$\kappa(P'_\ell)$; this ring is reduced, but it may contain zero
divisors. In fact, it is a direct product of fields which are finite separable extensions of $\kappa(P'_\ell)$ because $R^e$
is separable and essentially of finite type over $R$.\par

Consider a chain $R\rightarrow R'\rightarrow R''$ of $\nu$-extensions in $\cal T$. Let
\begin{eqnarray}
\kappa(P_\ell)\otimes_RR^e&=&\prod\limits_{j=1}^nK_j\\
(T')^{-1}\left(\kappa\left(P'_\ell\right)\otimes_RR^e\right)&=&\prod\limits_{j=1}^{n'}K'_j\\
(T'')^{-1}\left(\kappa\left(P''_\ell\right)\otimes_RR^e\right)&=&\prod\limits_{j=1}^{n''}K''_j
\end{eqnarray}
be the corresponding decompositions as products of finite field extensions of $\kappa(P_\ell)$ (resp. $\kappa(P'_\ell)$,
resp. $\kappa(P''_\ell)$). We want to compare $(T')^{-1}\left(\kappa\left(P'_\ell\right)\otimes_RR^e\right)$ with
$(T'')^{-1}\left(\kappa\left(P''_\ell\right)\otimes_RR^e\right)$.
\begin{remark} The ring $\kappa\left(P'_\ell\right)\otimes_RR^e$ is itself a direct product of finite extensions of
$\kappa\left(P'_\ell\right)$; say $\kappa\left(P'_\ell\right)=\prod\limits_{j\in S'}K'_j$ for a certain set $S'$. In this
situation, localization is the same thing as the natural projection to the product of the $K'_j$ over a certain subset
$\{1,\dots,n'\}$ of $S'$. Thus the passage from $(T')^{-1}\left(\kappa\left(P'_\ell\right)\otimes_RR^e\right)$ to
$(T'')^{-1}\left(\kappa\left(P''_\ell\right)\otimes_RR^e\right)$ can be viewed as follows: first, tensor each $K'_j$ with
$\kappa\left(P''_\ell\right)$ over $\kappa\left(P'_\ell\right)$; then, in the resulting direct product of fields, remove
a certain number of factors.
\end{remark}
Let $\bar K'_1,\dots,\bar K'_{\bar n'}$ be the distinct isomorphism classes of finite extensions of
$\kappa\left(P'_\ell\right)$ appearing among $K'_1,\dots,K'_{n'}$, arranged in such a way that
$\left[\bar K'_j:\kappa\left(P'_\ell\right)\right]$ is non-increasing with $j$, and similarly for
$\bar K''_1,\dots,\bar K''_{\bar n''}$.
\begin{lemma}\label{decrease} We have the inequality
\begin{equation}
\left(\left[\bar K''_1:\kappa\left(P''_\ell\right)\right],\dots,\left[\bar K''_{\bar n''}:\kappa\left(P''_\ell\right)\right],
n''\right)\le
\left(\left[\bar K'_1:\kappa\left(P'_\ell\right)\right],\dots,\left[\bar K'_{\bar n'}:\kappa\left(P'_\ell\right)\right],
n'\right)\label{eq:lex}
\end{equation}
in the lexicographical ordering. Furthermore, either $R'$ is $\ell$-stable or there exists $R''\in\cal T$ such that strict
inequality holds in (\ref{eq:lex}).
\end{lemma}
\begin{proof} Fix a $q\in\{1,\dots,\bar n'\}$ and consider the tensor product $\bar
K'_q\otimes_R\kappa\left(P''_\ell\right)$. Since $\bar K'_q$ is separable over $\kappa\left(P'_\ell\right)$, the ring
$\bar K'_q\otimes_R\kappa\left(P''_\ell\right)=\prod\limits_{j\in S''_q}K''_j$ is a product of fields. Moreover, two cases
are possible:

(a) there exists a non-trivial extension $L$ of $\kappa\left(P'_\ell\right)$ which embeds both into
$\kappa\left(P''_\ell\right)$ and $\bar K'_q$. In this case
\begin{equation}
\left[K''_j:\kappa\left(P''_\ell\right)\right]<\left[\bar K'_q:\kappa\left(P'_\ell\right)\right]\quad\text{ for all }j\in
S''_q.\label{eq:strict}
\end{equation}
(b) there is no field extension $L$ as in (a). In this case $\bar K'_q\otimes_R\kappa\left(P''_\ell\right)$ is a field,
so
\begin{equation}
\#S''_q=1\label{eq:card1}
\end{equation}
and
\begin{equation}
\left[K''_j:\kappa\left(P''_\ell\right)\right]=\left[\bar K'_q:\kappa\left(P'_\ell\right)\right]\quad\text{ for }j\in
S''_q.\label{eq:equal}
\end{equation}
Now, if there exists $q\in\{1,\dots,\bar n'\}$ for which (a) holds, take the smallest such $q$. Then
(\ref{eq:strict})--(\ref{eq:equal}) imply that strict inequality holds in (\ref{eq:lex}). On the other hand, if (b) holds
for all $q\in\{1,\dots,\bar n'\}$ then (\ref{eq:card1}) and (\ref{eq:equal}) imply that
\begin{equation}
\left(\left[\bar K''_1:\kappa\left(P''_\ell\right)\right],\dots,\left[\bar K''_{\bar n''}:\kappa\left(P''_\ell\right)\right]
\right)=
\left(\left[\bar K'_1:\kappa\left(P'_\ell\right)\right],\dots,\left[\bar K'_{\bar n'}:\kappa\left(P'_\ell\right)\right]
\right)\label{eq:lex1}
\end{equation}
and $n''\le n'$, so again (\ref{eq:lex}) holds.

Finally, assume that $R'$ is not $\ell$-stable. If there exists
$R''\in\cal T$ and $q\in\{1,\dots,\bar n'\}$ for which (a) holds, then by the above we have strict inequality in
(\ref{eq:lex}) and there is nothing more to prove. Assume there are no such $R''$ and $q$. Then
$(T')^{-1}(\kappa(P'_\ell)\otimes_RR^e)$ is not a domain, so $n'>1$.

Take $R''\in{\cal T}(R')$ such that $\left(\frac{R''}{P''_l}\otimes_RR^e\right)_{M''}$ is an integral domain; such an $R''$
exists by Lemma \ref{anirred1}. Then $n''=1<n'$, as desired.
\end{proof}
\begin{corollary}\label{anirred2} There exists a stable $R'\in\cal T$. The limit in
(\ref{eq:defin3}) is attained for this $R'$.
\end{corollary}
\begin{proof} In view of Proposition \ref{largeR1}, it is sufficient to prove that there exists $R'\in\cal T$
which which is $\ell$-stable for all $\ell\in\{0,1,\dots,r\}$. First, we fix
$\ell\in\{0,1,\dots,r\}$. Lemma \ref{decrease} implies that there exists $R'\in\mathcal{T}(R)$ which is $\ell$-stable.

By Proposition \ref{largeR1}, repeating the procedure above for each $\ell$ we can successively enlarge $R'$ in such a
way that it becomes stable.

The last statement follows from Proposition \ref{largeR2}.
\end{proof}

We are now in the position to prove Theorem \ref{hensel0}.

By Theorem \ref{primality1} (1), $H_{2\ell-1}$ is prime. By Proposition \ref{contracts}, $H_{2\ell+1}$
maps to $P_\ell$ under the map $\pi^e:\spec\ R^e\rightarrow\spec\ R$.
Since this map is \'etale, its fibers are zero-dimensional, which
shows that $H_{2\ell+1}$ is a minimal prime of $P_\ell$. This proves (1) of Theorem \ref{hensel0}.

By Proposition \ref{Hintilde}, for $0\le i\le2r$, $\tilde H_i$ is a prime ideal of $R^e$, containing
$H_i$. Since the fibers of $\pi^e$ are zero-dimensional, we must have $\tilde H_i=H_i$, so $\tilde H_{2\ell}=\tilde
H_{2\ell+1}=H_{2\ell}=H_{2\ell+1}$ for $0\le\ell\le r$. In particular, $\tilde H_0=H_1$. This shows that the unique prime
$\tilde H_0$ of $R^e$ such that there exists an extension $\nu^e_-$ of $\nu$ to $\frac{R^e}{\tilde H_0}$ is $\tilde H_0=H_1$.
Now (2) of the Theorem is given by Proposition \ref{uniqueness1}.

(3) of Theorem \ref{hensel0} is now immediate. This completes the proof of Theorem \ref{hensel0}.
\end{proof}

We note the following corollary of the proof of (2) of Theorem
\ref{hensel0} and Corollary \ref{blup1}. Let $\Phi^e=\nu^e_-(R^e\setminus\{0\})$, take an element $\beta\in\Phi^e$ and let
${\cal P}^e_\beta$ denote the $\nu^e_-$-ideal of $R^e$ of value $\beta$.
\begin{corollary}\label{blup} Take an element $x\in {\cal P}^e_\beta$. There
exists a local blowing up $R\rightarrow R'$ such that
$\beta\in\nu(R')\setminus\{0\}$ and $x\in {\cal P}'_\beta{R'}^e$.
\end{corollary}

\section{The Main Theorem: the primality of implicit ideals.}
\label{prime}

In this section we study the ideals $H_j$ for $\hat R$ instead of $\tilde R$. The main result of this section is
\begin{theorem}\label{primality} The ideal $H_{2\ell-1}$ is prime.
\end{theorem}
\begin{proof} For the purposes of this proof, let $H_{2\ell-1}$
denote the implicit ideals of $\hat R$ and $\tilde H_{2\ell-1}$ the
implicit prime ideals of the henselization $\tilde R$ of $R$.

Let $S$ be a local domain. By \cite{Nag} (Theorem (43.20), p. 187) there exists bijective maps between the set of minimal
prime ideals of the henselization $\tilde S$ and the maximal ideals of the normalization $S^n$. If, in addition, $S$ is excellent, the two above sets also admit a natural bijection to the set of minimal primes of $\hat S$ \cite{Nag} (Corollary (44.3), p. 189). If $S$ is a henselian local domain, its only minimal prime is the (0) ideal, hence by the above the same is true of $\hat S$. Thus $\hat S$ is also a domain.

This shows that any excellent henselian local domain is analytically
irreducible, hence $\tilde H_{2\ell-1}\hat R$ is prime for all
$\ell\in\{1,\dots,r+1\}$. Let $\tilde\nu_-$ denote the unique
extension of $\nu$ to $\frac{\tilde R}{\tilde H_1}$, constructed in
the previous section. Let $H^*_{2\ell-1}\subset\frac{\tilde
R}{\tilde H_1}$ denote the implicit ideals associated to the
henselian ring $\frac{\tilde R}{\tilde H_1}$ and the valuation
$\tilde\nu_-$.
\medskip

\noi\textit{Claim.} We have $H^*_{2\ell-1}=\frac{H_{2\ell-1}}{\tilde H_1}$.
\medskip

\noi\textit{Proof of the claim:} For $\beta\in\Gamma$, let $\tilde
P_\beta$ denote the $\tilde\nu_-$-ideal of $\frac{\tilde R}{\tilde
H_1}$ of value $\beta$. For all $\beta$, we have
$\frac{P_\beta}{\tilde H_1}\subset\tilde P_\beta$, and the same
inclusion holds for all the local blowings up of $R$, hence
$\frac{H_{2\ell-1}}{\tilde H_1}\subset H^*_{2\ell-1}$. To prove the
opposite inclusion, we may replace $\tilde R$ by a finitely
generated strict \'etale extension $R^e$ of $R$. Now let
$\Phi^e=\nu^e_-\left(R^e\setminus\{0\}\right)$ and take an element
$\beta\in\Phi^e\cap\Delta_{\ell -1}$. By Corollary \ref{blup}, there
exists a local blowing up $R\rightarrow R'$ such that $x\in
P'_\beta{R'}^e$. Letting $\beta$ vary over
$\Phi^e\cap\Delta_{\ell-1}$, we obtain that if $x\in H^*_{2\ell-1}$
then $x\in\frac{H_{2\ell-1}}{\tilde H_1}$, as desired. This
completes the proof of the claim.
\medskip

The Claim shows that replacing $R$ by $\frac{\tilde R}{\tilde H_1}$ in Theorem
\ref{primality} does not change the problem. In other words, we
may assume that $R$ is a henselian domain and, in particular, that
$\hat R$ is also a domain. Similarly, the ring $\frac
R{P_\ell}\otimes_R\hat R\cong\frac{\hat R}{P_\ell}$ is a domain, hence
so is its localization $\kappa(P_\ell)\otimes_R\hat R$.

Since $R$ is a henselian excellent ring, it is algebraically closed
in $\hat R$ (\cite{Nag}, Corollary (44.3), p. 189 and Corollary
\ref{notnormal} of the Appendix); of course, the same holds for
$\frac R{P_\ell}$ for all $\ell$. Then $\kappa(P_\ell)$ is
algebraically closed in $\kappa(P_\ell)\otimes_R\hat R$. This shows
that the ring $R$ is stable. Now the Theorem follows from Theorem
\ref{primality1}. This completes the proof of Theorem
\ref{primality}.
\end{proof}

\section{Towards a proof of Conjecture \ref{teissier}, assuming local
  uniformization in lower dimension}
\label{locuni1}

Let the notation be as in the previous sections. In this section,
we assume that the Local Uniformization Theorem holds and propose an approach to proving Conjecture \ref{teissier}. We prove a Corollary of Conjecture \ref{teissier} which gives a sufficient condition for $\hat\nu_-$ to be unique, which also turns out to be necessary under the additional assumption that $\hat\nu_-$ is minimal. We will assume that all the $R'\in\mathcal T$ are birational to each other, so that all the fraction fields
$K'=K$ and the homomorphisms $R'\rightarrow R''$ are local
blowings up with respect to $\nu$. Finally, we assume that $R$
contains a field $k_0$ and a local $k_0$-subalgebra $S$
essentailly of finite type, over which $R$ is strictly \'etale. In
particular, all the rings in sight are equicharacteristic.

First, we state the Local Uniformization Theorem in the precise form
in which we are going to use it. \begin{definition}\label{lut} We say that
  \textbf{the embedded Local Uniformization theorem holds in }
  $\mathcal T$ if the following conditions are satisfied.

Take an integer $\ell\in\{1,\dots,r-1\}$. Let
$\mu_{\ell+1}:=\nu_{\ell+1}\circ\nu_{\ell+2}\circ\dots\circ\nu_r$. Consider
a tree $\{H'\}$ of prime ideals of $\frac{\hat R'}{P'_\ell}$ such that
$H'\cap\frac{R'}{P'_\ell}=(0)$ and a tight extension
$\hat\mu_{2\ell+2}$ of $\mu_{\ell+1}$ to
$\lim\limits_{\overset\longrightarrow{R'}}\frac{\hat R'}{H'}$.

(1) There exists a local blowing up $\pi:R\rightarrow R'$ in $\mathcal
T$, which induces an isomorphism at the center of $\nu_\ell$, such
that $\frac{R'}{P'_\ell}$ is a regular local ring.

(2) Assume that $\frac{R'}{P'_\ell}$ is a regular local ring. Then
there exists in $\mathcal T$ a sequence $\pi:R\rightarrow R'$ of local
blowings up along non-singular centers not containing the center of
$\nu_l$ such that $\frac{\hat R'}{H'}$ is a regular local ring.
\end{definition}
It is well known (\cite{A}, \cite{L}, \cite{Z}) that the embedded
Local Uniformization theorem holds if $R$ is an excellent local domain
such that either $char\ k=0$ or $\dim\ R\le3$ (to be precise, (1) of
Definition \ref{lut} is well known and (2) is an easy consequence of
known results). While the Local Uniformization theorem in full
generality is still an open problem, it is widely believed to hold for
arbitrary quasi-excellent local domains. Proving this is an active
field of current research in algebraic geometry. Proving local
uniformization for rings of arbitrary characteristic is one of the
intended applications of Conjecture \ref{teissier}. Note that in Definition \ref{lut} we require only
local uniformization of rings of dimension strictly less than $\dim\
R$; the idea is to use induction on $\dim\ R$ to prove local
uniformization of rings of dimension $\dim\ R$.

We begin by stating a strengthening of Conjecture \ref{teissier} (using Remark
\ref{rephrasing}):
\begin{conjecture}\label{teissier1} The valuation $\nu$ admits at least one
  tight extension $\hat\nu_-$. This tight extension $\hat\nu_-$ can be
  chosen to have the following additional property: for rings $R'$
  sufficiently far in the tree $\mathcal T$ we have the equality of
  semigroups $\hat\nu_-\left(\frac{\hat R'}{\tilde
      H'_0}\setminus\{0\}\right)=\nu(R'\setminus\{0\})$ and for
  $\beta\in\nu(R'\setminus\{0\})$ the $\hat\nu_-$-ideal of value $\beta$
  is $\frac{\mathcal P_\beta\hat R'}{\tilde H'_0}$. In particular, we
  have the equality of graded algebras $\gr_\nu
  R'=\gr_{\hat\nu_-}\frac{\hat R'}{\tilde H'_0}$.
\end{conjecture}
Below, we give an explicit construction of a valuation $\hat\nu_-$ whose existence is asserted in the Conjecture by describing the trees of ideals $\tilde H'_i$, $0\le i\le 2r$ and, for each $i$, a valuations $\hat\nu_i$ of the residue field $k_{\nu_{i-1}}$, such that $\hat\nu_-=\hat\nu_1\circ\dots\hat\nu_{2r}$. More precisely, for $\ell\in\{0,\dots,r-1\}$, we will construct, recursively in the descending order of $\ell$, a tree $J'_{2\ell+1}$ of prime ideals of $\frac{\hat R'}{H'_{2\ell}}$, $R'\in\mathcal{T}$, such that
$J'_{2\ell+1}\cap\frac{R'}{P_\ell}=(0)$, and an extension $\hat\mu_{2\ell+2}$ of $\mu_{\ell+1}$ to
$\lim\limits_{\overset\longrightarrow{R'\in\mathcal{T}}}\frac{\hat R'}{J'_{2\ell+1}\hat R'}$; the valuation
$\hat\mu_2$ will be our candidate for the desired tight extension $\hat\nu_-$ of $\mu_1=\nu$. Unfortunately, two steps in this construction still remain
conjectural, namely, proving that $\hat\mu_{2\ell+2}$ is, indeed, a valuation, and that it is tight (this is essentially the content of Conjectures \ref{strongcontainment} and \ref{containment} below). Once these conjectures are proved, our recursive construction will be complete and Conjecture \ref{teissier1} will follow by setting $\hat\nu_-=\hat\mu_2$.

Let us now describe the construction in detail. According to Corollary \ref{htstable1}, we may
assume that $ht\ H'_i$ is constant for each $i$ after replacing $R$ by
some other ring sufficiently far in $\mathcal T$. From now on, we will
make this assumption without always stating it explicitly.

By (1) of Definition \ref{lut}, applied successively to the trees of ideals
$$
P'_\ell\subset R',\quad\ell\in\{1,\dots,r-1\},
$$
there exists $R''\in\cal T$ such that $\frac{R''}{P''_\ell}$ is regular for all $\ell\in\{1,\dots,r-1\}$. Without loss of generality,
we may also assume that $R''$ is stable.

For $\ell\in\{1,\dots,r-1\}$ and $R''\in\mathcal{T}$, let
$\mathcal{T}_\ell(R'')$ denote the subtree of $\mathcal T$,
consisting of all the local blowings up of $R''$ along ideals not
contained in $P''_\ell$ (such local blowings up induce an
isomorphism at the point $P''_\ell\in\spec\ R''$). Below, we will sometimes
work with trees of rings and ideals indexed by
$\mathcal{T}_\ell(R'')$ for suitable $\ell$ and $R''$ (instead of
trees indexed by all of $\cal T$); the precise tree with which we
are working will be specified in each case.

For $\ell=r-1$, we define $J'_{2r-1}:=H'_{2r-1}$ and $\hat\mu_{2r}:=\hat\nu_{2r,0}$; according to Proposition
\ref{nu0unique}, $\hat\nu_{2r,0}=\hat\nu_{2r}$ is the unique extension of $\nu_r$ to
$\lim\limits_{\overset\longrightarrow{R'}}\frac{\hat R'}{H'_{2r-1}\hat R'}$.

Next, assume that $\ell\in\{1,\dots,r-1\}$, that the tree $J'_{2\ell+1}$ of prime ideals of $\frac{\hat
R'}{H'_{2\ell}\hat R'}$ and a tight extension $\hat\mu_{2\ell+2}$ of $\mu_{\ell+1}$ to
$\lim\limits_{\overset\longrightarrow{R'}}\frac{\hat R'}{J'_{2\ell+1}}$ are already constructed for
$R'\in\mathcal{T}$ and that $J'_{2\ell+1}\cap\frac{R'}{P_\ell}=(0)$. It remains to construct the
ideals $J'_{2\ell-1}\subset\frac{\hat R'}{H'_{2\ell-2}\hat R'}$ and a tight extension
$\hat\mu_{2\ell}$ of $\mu_\ell$ to $\lim\limits_{\overset\longrightarrow{R'}}\frac{\hat R'}{J'_{2\ell-1}}$ for
$R'\in\mathcal{T}$.

We will assume, inductively, that for all $R'\in\mathcal{T}$ the quantity $ht\ J'_{2\ell+1}$ is constant and the following conditions hold:
\begin{enumerate}
\item We have the equality of semigroups $\hat\mu_{2\ell+2}\left(\frac{\hat R'}{J'_{2\ell+1}}\setminus\{0\}\right)\cong
\mu_{\ell+1}\left(\frac{R'}{P'_\ell}\setminus\{0\}\right)$.
\item For all $\beta\in\mu_{\ell+1}\left(\frac{R'}{P'_\ell}\setminus\{0\}\right)$ the $\hat\mu_{2\ell+2}$-ideal of $\frac{\hat R'}{J'_{2\ell+1}}$ of  value $\beta$ is the extension to $\frac{\hat R'}{J'_{2\ell+1}}$ of the $\mu_{\ell+1}$-ideal of $\frac{R'}{P'_\ell}$ of value $\beta$.
\item In particular, we have a canonical isomorphism $gr_{\hat\mu_{2\ell+2}}\frac{\hat R'}{J'_{2\ell+1}}\cong
gr_{\mu_{\ell+1}}\frac{R'}{P'_\ell}$ of graded algebras.
\end{enumerate}
By (2) of Definition \ref{lut} applied to the prime ideals $J'_{2\ell+1}\subset\frac{\hat R'}{H'_{2\ell}\hat R'}$, there exists
$R'\in\cal T$ such that both $\frac{R'}{P'_\ell}$ and $\frac{\hat R'}{J'_{2\ell+1}}$ are regular. The fact that
$\frac{R'}{P'_\ell}$ is regular implies that so is $\frac{\hat R'}{P'_\ell\hat R'}$. In particular, $\frac{\hat
R'}{P'_\ell\hat R'}$ is a domain, so $H'_{2\ell}=P'_\ell\hat R'$. Take a regular system of parameters
$$
\bar u'=(\bar u'_1,\dots,\bar u'_{n_\ell})
$$
of $\frac{R'}{P'_\ell}$. Let $k'$ denote the common residue field of $R'$,
$\frac{R'}{P'_{\ell-1}}$ and $\frac{R'}{P'_\ell}$. Fix an isomorphism $\frac{R'}{P'_\ell}\cong k'[[\bar u']]$. Renumbering the variables, if necessary, we may assume that there exists $s_\ell\in\{1,\dots,n_\ell\}$ such that $\bar u'_1,\dots,\bar u'_{s_\ell}$ are $k'$-linearly independent modulo
$({m'}^2+J'_{2\ell+1})\frac{\hat R'}{H'_{2\ell}}$. Since $\frac{\hat R'}{J'_{2\ell+1}}$ is regular, the ideal $J'_{2\ell+1}$ is
generated by a set of the form $\bar v'=(\bar v'_{s_\ell+1},\dots,\bar v'_{n_\ell})$, where
$$
\bar v'_j=\bar u'_j-\bar\phi_j(\bar u'_1,\dots,\bar u'_{s_\ell}),\ \bar\phi_j(\bar u'_1,\dots,\bar u'_{s_\ell})\in k'[[\bar
u'_1,\dots,\bar u'_{s_\ell}]].
$$
Let $\bar w'=(\bar w'_1,\dots,\bar w'_{s_\ell})=(\bar
u'_1,\dots,\bar u'_{s_\ell})$. Let $z'$ be a minimal set of
generators of $\frac{P'_\ell}{P'_{\ell-1}}$. Let $k'_0$ be a
quasi-coefficient field of $R'$ (that is, a subfield of $R'$ over
which $k'$ is formally \'etale; such a quasi-coefficient field
exists by \cite{Mat}, moreover, since $R'$ is algebraic over a
finite type algebra over a field by hypotheses, $k'$ is finite
over $k'_0$). By the hypotheses on $R$ and since
$\frac{R'}{P'_\ell}$ is a regular local ring and $\bar u'$ is a minimal
set of generators of its maximal ideal $\frac{m'}{P'_\ell}$, there
exists an ideal $I\subset k'_0[z']$ such that
$\frac{R'}{P'_{\ell-1}}$ is an \'etale extension of
$\frac{k'_0[z',\bar u']_{(z',\bar u')}}I$. By assumptions, we have $ht\
P'_{\ell-1}<ht\ P'_\ell$, so $0<ht\ P'_\ell-ht\
P'_{\ell-1}=ht(z')-ht\ I$, in other words,
\begin{equation}
ht\ I<ht(z').\label{eq:Inotmaximal}
\end{equation}
Next, we prove two general lemmas about ring extensions.
\medskip

\begin{notation} Let $k_0$ be a field and $(S,m,k)$ a local noetherian
$k_0$-algebra. For a field extension
\begin{equation}
k_0\hookrightarrow L\label{eq:k0inktilde}
\end{equation}
such that $k\otimes_{k_0}L$ is a domain, let $S(L)$ denote the
localization of the ring $S\otimes_{k_0}L$ at the prime ideal
$m(S\otimes_{k_0}L)$.
\end{notation}

\begin{lemma}\label{noetherian} Let $k_0$, $(S,m,k)$ and $L$ be as above. The ring $S(L)$ is noetherian.
\end{lemma}
\begin{proof} If the field extension (\ref{eq:k0inktilde})
is finitely generated, the Lemma is obvious. In the general case,
write $L=\lim\limits_{\overrightarrow{i}}L_i$ as a direct limit of
its finitely generated subextensions. For each $L_i$, let $k_i$
denote the residue field of $S(L_i)$; $k_i$ is nothing but the
field of fractions of $k\otimes_{k_0}L_i$. Write $\hat
S=\frac{k[[x]]}H$, where $x$ is a set of generators of $m$ and $H$
a certain ideal of $k[[x]]$. Then
$\widehat{S(L_i)}\cong\frac{k_i[[x]]}{Hk_i[[x]]}$. Given two
finitely generated extensions $L_i\subset L_j$ of $k_0$, contained
in $L$, we have a commutative diagram
$$
\begin{matrix}
S(L_j)&\overset{\pi_j}\rightarrow&\widehat{S(L_j)}&\cong&\frac{k_j[[x]]}{Hk_j[[x]]}&\\
\psi_{ij}\uparrow&\ &\uparrow&\ &\uparrow\phi_{ij}\\
S(L_i)&\overset{\pi_i}\rightarrow&\widehat{S(L_i)}&\cong&\frac{k_i[[x]]}{Hk_i[[x]]}
\end{matrix}
$$
where $\phi_{ij}$ is the map induced by the natural inclusion
$k_i\hookrightarrow k_j$ and the identity map of $x$ to itself.
Let $k_\infty=\lim\limits_{\overrightarrow{i}}k_i$. Then, for each
$i$, we have the obvious faithfully flat map
$\rho_i:\widehat{S(L_i)}\rightarrow\frac{k_\infty[[x]]}{Hk_\infty[[x]]}$,
defined by the natural inclusion $k_i\hookrightarrow k_\infty$ and
the identity map of $x$ to itself; the maps $\rho_i$ commute with
the $\phi_{ij}$. Thus, we have constructed a faithfully flat map
$\rho_i\circ\pi_i$ from each element of the direct system $S(L_i)$ to the fixed noetherian ring
$\frac{k_\infty[[x]]}{Hk_\infty[[x]]}$; moreover, the maps
$\rho_i\circ\pi_i$ are compatible with the homomorphisms
$\psi_{ij}$ of the direct system. This implies that the ring
$S(L)=\lim\limits_{\overrightarrow{i}}S(L_i)$ is noetherian.
\end{proof}

\begin{lemma}\label{IS(t)} Let $(S,m,k)$ be a local noetherian ring. Let $t$ be an arbitrary collection of independent variables.
Consider the rings $S[t]$ and $S(t):=S[t]_{mS[t]}$. Let $I$ be an
ideal of $S$. Then
\begin{equation}
IS(t)\cap S[t]=IS[t].\label{eq:IcapSt=I}
\end{equation}
\end{lemma}
\begin{proof} First, assume the collection $t$ consists of a
single variable. Consider elements $f,g\in S[t]$ such that
\begin{equation}
f\notin mS[t]\label{eq:fnotinmSt}
\end{equation}
 and
 \begin{equation}
 fg\in IS[t].\label{eq:fginISt}
 \end{equation}
 Proving the equation (\ref{eq:IcapSt=I}) amounts to proving that
\begin{equation}
g\in IS[t].\label{eq:ginISt}
\end{equation}
We prove (\ref{eq:ginISt}) by contradiction. Assume that $g\notin IS[t]$. Then there exists $n\in\mathbb N$ such that
$g\notin(I+m^n)S[t]$. Take the smallest such $n$, so that
\begin{equation}
g\in\left(I+m^{n-1}\right)S[t]\setminus(I+m^n)S[t].\label{eq:setminus}
\end{equation}
Write $f=\sum\limits_{j=0}^qa_jt^j$ and
$g=\sum\limits_{j=0}^lb_jt^j$. Let
\begin{eqnarray}
l_0:&=&\max\{j\in\{0,\dots,l\}\ |\ b_j\notin I+m^n\}\quad\text{ and}\\
q_0:&=&\max\{j\in\{0,\dots,q\}\ |\ a_j\notin m\}.
\end{eqnarray}
Let $c_{l_0+q_0}$ denote the $(l_0+q_0)$-th coefficient of $fg$.  We have
$$
c_{l_0+q_0}=\sum\limits_{i+j=l_0+q_0}a_ib_j=a_{q_0}b_{l_0}+\sum\limits_{\begin{array}{c}i+j=l_0+q_0\\
i>q_0\end{array}}a_ib_j+\sum\limits_{\begin{array}{c}i+j=l_0+q_0\\ j>l_0\end{array}}a_ib_j.
$$
By definition of $l_0$ and $q_0$ and (\ref{eq:setminus}) we have:
\begin{eqnarray}
a_{q_0}b_{l_0}&\notin&I+m^n\quad\text{and}\\
\sum\limits_{\begin{array}{c}i+j=l_0+q_0\\ i>q_0\end{array}}a_ib_j+\sum\limits_{\begin{array}{c}i+j=l_0+q_0\\
j>l_0\end{array}}a_ib_j&\in&I+m^n.
\end{eqnarray}
Hence $c_{l_0+q_0}\notin I+m^n$, which contradicts
(\ref{eq:fginISt}). This completes the proof of Lemma \ref{IS(t)}
in the case when $t$ is a single variable. The case of a general
$t$ now follows by transfinite induction on the collection $t$.
\end{proof}

\begin{lemma}\label{contractsto0} There exist sets of representatives
$$
u'=(u'_1,\dots,u'_{n_\ell})
$$
of $\bar u'$ and $\phi_j$ of $\bar\phi_j$, $s_\ell<j\le n_\ell$, in $\frac{\hat R'}{H'_{2\ell-2}\hat R'}$, having the
following properties. Let
\begin{eqnarray}
w'=(w'_1,\dots,w'_{s_\ell})&=&(u'_1,\dots,u'_{s_\ell}),\\
v'=(v'_{s_\ell+1},\dots,v'_{n_\ell})&=&
(u'_{s_\ell+1}-\phi_{s_\ell+1},\dots,u'_{n_\ell}-\phi_{n_\ell})\label{eq:defv}.
\end{eqnarray}
Let $J'_{2\ell-1}=\frac{H'_{2\ell-1}}{H'_{2\ell-2}}+(v')\subset\frac{\hat R'}{H'_{2\ell-2}}$. Then
\begin{equation}
w'\subset\frac{R'}{P'_{\ell-1}}\label{eq:winR}
\end{equation}
and
\begin{equation}
J'_{2\ell-1}\cap\frac{R'}{P'_{\ell-1}}=(0).\label{eq:contractsto0}
\end{equation}
\end{lemma}
\begin{proof}\textit{(of Lemma \ref{contractsto0})} There is no problem choosing $w'$ to satisfy (\ref{eq:winR}).

As for (\ref{eq:contractsto0}), we first prove the Lemma under the
assumption that $k$ is countable. We choose the representatives $u'$
arbitrarily and let $\bar\phi_j(u')\in k[[u']]$ denote the formal
power series obtained by substituting $u'$ for $\bar u'$ in
$\bar\phi_j$. Any representative $\phi_j$ of $\bar\phi_j$,
$s_\ell<j\le n_\ell$ has the form $\phi_j=\bar\phi_j(u')+h_j$ with
$h_j\in(z')\frac{\hat R'}{H'_{2\ell-2}}$. We define the $h_j$
required in the Lemma recursively in $j$. Take
$j\in\{s_\ell+1,\dots,n_\ell\}$. Assume that
$h_{s_{\ell+1}},\dots,h_{j-1}$ are already defined and that
\begin{equation}
(v'_{s_\ell+1},\dots,v'_{j-1})\cap\frac{R'}{P'_{\ell-1}}=(0),\label{eq:j-1inter0}
\end{equation}
where we view $\frac{R'}{P'_{\ell-1}}$ as a subring of $\frac{\hat
R'}{H'_{2\ell-1}}$. Since the ring $\frac{R'}{P'_{\ell-1}}$ is
countable, there are countably many ideals in $\frac{\hat
R'}{H'_{2\ell-1}+(v'_{s_\ell+1},\dots,v'_{j-1})}$, not contained in
$\frac{(z')\hat R'}{H'_{2\ell-1}+(v'_{s_\ell+1},\dots,v'_{j-1})}$,
which are minimal primes of ideals of the form $(f)\frac{\hat
R'}{H'_{2\ell-1}+(v'_{s_\ell+1},\dots,v'_{j-1})}$, where $f$ is a
non-zero element of $\frac{m'}{P'_{\ell-1}}$. Let us denote these
ideals by $\{I_q\}_{q\in\mathbb N}$; we have
\begin{equation}
ht\ I_q=1\quad\text{ for all }q\in\mathbb N.\label{eq:ht=1}
\end{equation}
We note that
\begin{equation}
\frac{(z')\hat R'}{H'_{2\ell-1}+(v'_{s_\ell+1},\dots,v'_{j-1})}\not\subset I_q\quad\text{ for all }q\in\mathbb
N.\label{eq:znotin}
\end{equation}
Indeed, by (\ref{eq:Inotmaximal}) and (\ref{eq:j-1inter0}) we have $ht\ \frac{(z')\hat
R'}{H'_{2\ell-1}+(v'_{s_\ell+1},\dots,v'_{j-1})}\ge1$. In view of
(\ref{eq:ht=1}), containment in (\ref{eq:znotin}) would imply
equality, which contradicts the definition of $I_q$.

Since $H'_{2\ell-1}\subsetneqq H'_{2\ell}$ and $J'_{2\ell+1}\subsetneqq\frac{m'\hat R'}{H'_{2\ell}}$, we have
$$
\dim\frac{\hat R'}{H'_{2\ell-1}+(v'_{s_\ell+1},\dots,v'_{j-1})}\ge(ht\ H'_{2\ell}-ht\ H'_{2\ell-1})+ht\
\frac{m'\hat R'}{H'_{2\ell}}-(j-s_\ell-1)\ge
$$
\begin{equation}
(ht\ H'_{2\ell}-ht\ H'_{2\ell-1})+ht\ \frac{m'\hat R'}{H'_{2\ell}}-ht\ J'_{2\ell+1}+1\ge3.
\end{equation}
Let $\tilde u_j$ denote the image of $u'_j-\bar\phi_j(u')$ in $\frac{\hat
R'}{H'_{2\ell-1}+(v'_{s_\ell+1},\dots,v'_{j-1})}$. Next, we construct an element $\tilde h_j\in\frac{(z')\hat
R'}{H'_{2\ell-1}+(v'_{s_\ell+1},\dots,v'_{j-1})}$ such that
\begin{equation}
\tilde u_j-\tilde h_j\notin\bigcup\limits_{q=1}^\infty I_q.\label{eq:notinIq}
\end{equation}
The element $\tilde h_j$ will be given as the sum of an infinite
series $\sum\limits_{t=0}^\infty h_{jt}^t$ in $(z')\frac{\hat
R'}{H'_{2\ell-1}+(v'_{s_\ell+1},\dots,v'_{j-1})}$, convergent in
the $m'$-adic topology, which we will now construct recursively in
$t$. Put $h_{j0}=0$. Assume that $t>0$, that $h_{j0},\dots,h_{j,t-1}$ are already defined and that for
$q\in\{1,\dots,t-1\}$ we have $u'_j-\bar\phi_j(u')-\sum\limits_{l=0}^qh_{jl}\notin\bigcup\limits_{l=1}^qI_l$
and $h_{jq}\in(z')\bigcap\left(\bigcap\limits_{l=1}^{q-1}I_l\right)$.
If $u'_j-\bar\phi_j(u')-\sum\limits_{l=0}^{t-1}h_{jl}\notin I_t$, put $h_{jt}=0$. If
$u'_j-\bar\phi_j(u')-\sum\limits_{l=0}^{t-1}h_{jl}\in I_t$, let $h_{jt}$ be any element of
$(z')\bigcap\left(\bigcap\limits_{l=1}^{t-1}I_l\right)\setminus
I_t$ (such an element exists because $I_t$ is prime, in view of
(\ref{eq:znotin})). This completes the definition of $\tilde h_j$.
Let $h_j$ be an arbitrary representative of $\tilde h_j$ in
$\frac{\hat R'}{H'_{2\ell-2}}$.

We claim that
\begin{equation}
\left(H'_{2\ell-1}+(v'_{s_\ell+1},\dots,v'_j)\right)\cap\frac{R'}{P'_{\ell-1}}=(0).\label{eq:jinter0}
\end{equation}
Indeed, suppose the above intersection contained a non-zero element $f$. Then any minimal prime $\tilde I$ of the
ideal $(v'_j)\frac{\hat R'}{H'_{2\ell-1}+(v'_{s_\ell+1},\dots,v'_{j-1})}$ is also a minimal prime of
$(f)\frac{\hat R'}{H'_{2\ell-1}+(v'_{s_\ell+1},\dots,v'_{j-1})}$. Since
$v_j\notin\frac{(z')\hat R'}{H'_{2\ell-1}+(v'_{s_\ell+1},\dots,v'_{j-1})}$, we have $\tilde
I\not\subset\frac{(z')\hat R'}{H'_{2\ell-1}+(v'_{s_\ell+1},\dots,v'_{j-1})}$. Hence $\tilde I=I_q$ for some
$q\in\mathbb N$. Then $v_j\in I_q$, which contradicts (\ref{eq:notinIq}).

Carrying out the above construction for all
$j\in\{s_\ell+1,\dots,n_\ell\}$ produces the elements $\phi_j$
required in the Lemma. This completes the proof of Lemma
\ref{contractsto0} in the case when $k$ is countable.

Next, assume that $k$ is uncountable. Let $u'$ be chosen as above.

By assumption, $\frac{R'}{P'_{\ell-1}}$ contains a
$k_0$-subalgebra $S$ essentially of finite type, over which $\frac{R'}{P'_{\ell-1}}$ is strictly \'etale. Take a
countable subfield $L_1\subset k_0$ such that the algebra $S$ is defined already over $L_1$ (this
means that $S$ has the form
\begin{equation}
S=(S'_1\otimes_{L_1}k_0)_{m'_1(S'_1\otimes_{L_1}k_0)},\label{eq:R'/P'}
\end{equation}
where $(S'_1,m'_1,k'_1)$ is a local $L_1$-algebra essentially of
finite type). Next, let $L_1\subset L_2\subset...$ be an
increasing chain of finitely generated field extensions of $L_1$,
contained in $k_0$, having the following property. Let
$(S'_q,m'_q,k'_q)$ denote the localization of
$S'_1\otimes_{k'_1}L_q$ at the maximal ideal
$m'_1(S'_1\otimes_{k'_1}L_q)$. We require that
$$
k'_\infty:=\bigcup\limits_{q=1}^\infty k'_q
$$
contain all the coefficients of all the formal power series
$\bar\phi_{s_\ell+1},\dots,\bar\phi_{n_\ell}$ and such that the
ideal $\frac{H'_{2\ell-1}}{H'_{2\ell-2}}$ is generated by elements
of $\frac{k'_\infty[[z']]}{I_\infty}[[u']]$, where $I_\infty$ is
the kernel of the natural homomorphism
$k'_\infty[[z']]\rightarrow\frac{\hat R'}{H'_{2\ell-2}}$. Let
$H'_{2\ell-1,\infty}=\frac{H'_{2\ell-1}}{H'_{2\ell-2}}\cap
\frac{k'_\infty[[z']]}{I_\infty}[[u']]$. We have constructed an increasing chain
$S'_1\subset S'_2\subset...$ of local $L_1$-algebras essentially of finite type such that $k'_q$ is the
residue field of $S'_q$. Then
$S'_\infty:=\bigcup\limits_{q=1}^\infty S'_q$ is a local
noetherian ring whose completion is
$\frac{k'_\infty[[z',u']]}{\left(P'_{\ell-1}\cap
S'_\infty\right)k'_\infty[[z',u']]}$. Let $m'_\infty$ denote the
maximal ideal of $S'_\infty$. The above argument in the countable
case shows that there exist representatives
$\phi_{s_\ell+1},\dots,\phi_{n_\ell}$ of
$\bar\phi_{s_\ell+1},\dots,\bar\phi_{n_\ell}$ in $\frac{\hat
S'_\infty}{H'_{2\ell-2}\cap\hat S'_\infty}$ such that, defining
$v'=(v'_{s_\ell+1},\dots,v'_{n_\ell})$ as in (\ref{eq:defv}), we
have
\begin{equation}
\left((v')+H'_{2\ell-1,\infty}\right)\cap S'_\infty=(0).\label{eq:S'infty0}
\end{equation}
Let $L_\infty=\bigcup\limits_{q=1}^\infty L_q$ and let $t$ denote
a transcendence base of $k_0$ over $L_\infty$. Let the notation be
as in Lemma \ref{noetherian} with $k_0$ replaced by $L_\infty$.
For example, $S'_\infty(L_\infty(t))$ will denote the localization
of the ring $S'_\infty\otimes_{L_\infty}L_\infty(t)$ at the prime
ideal ideal $m'_\infty(S'_\infty\otimes_{L_\infty}L_\infty(t))$.

By (\ref{eq:S'infty0}),
\begin{equation}
\left((v')+H'_{2\ell-1,\infty}\right)\hat S'_\infty[t]\cap
S'_\infty[t]=(0).\label{eq:S'infty0t}
\end{equation}
Now Lemma \ref{IS(t)} and the fact that $S'_\infty[t]$ is a domain
imply that
\begin{equation}
\left((v')+H'_{2\ell-1,\infty}\right)\hat
S'_\infty(L_\infty(t))\cap
S'_\infty(L_\infty(t))=(0).\label{eq:capbarS=0}
\end{equation}
Next, let $\tilde L$ be a finite extension of $L_\infty(t)$,
contained in $k_0$; then $S'_\infty(\tilde L)$ is finite over
$S'_\infty(L_\infty(t))$. Since $\hat S'_\infty(\tilde L)$ is
faithfully flat over $\hat S'_\infty(L_\infty(t))$ and in view of
(\ref{eq:capbarS=0}), we have
$$
\left(\left((v')+H'_{2\ell-1,\infty}\right)\hat S'_\infty(\tilde
L)\cap S'_\infty(\tilde L)\right)\cap S'_\infty(L_\infty(t))=(0).
$$
Hence $ht\ \left((v')+H'_{2\ell-1,\infty}\right)\hat
S'_\infty(\tilde L)\cap S'_\infty(\tilde L)=0$. Since
$S'_\infty(\tilde L)$ is a domain, this implies that
\begin{equation}
\left((v')+H'_{2\ell-1,\infty}\right)\hat S'_\infty(\tilde L)\cap
S'_\infty(\tilde L)=(0).\label{eq:tildek=0}
\end{equation}
Since $k_0$ is algebraic over $L_\infty(t)$, it is the limit of
the direct system of all the finite extensions of $L_\infty(t)$
contained in it. We pass to the limit in (\ref{eq:tildek=0}). By
(\ref{eq:R'/P'}), we have $S=S'_\infty(k_0)$; we also note that
$\hat S=\frac{\hat R'}{H'_{2\ell-2}}$.

Since the natural maps $\hat S'_\infty(\tilde L)\rightarrow\hat
S'_\infty(k_0)$ are all faithfully flat, we obtain
\begin{equation}
\left((v')+H'_{2\ell-1,\infty}\right)\hat S'_\infty(k_0)\cap
S=(0).\label{eq:capS=0}
\end{equation}
Since $\hat S=\frac{\hat R'}{H'_{2\ell-2}}$ is also the formal
completion of $\hat S'_\infty(k_0)$, it is faithfully flat over
$\hat S'_\infty(k_0)$. Hence
\begin{equation}
J'_{2\ell-1}\cap\hat
S'_\infty(k_0)=\left((v')+H'_{2\ell-1,\infty}\right)\frac{\hat
R'}{H'_{2\ell-2}}\cap\hat
S'_\infty(k_0)=\left((v')+H'_{2\ell-1,\infty}\right)\hat
S'_\infty(k_0).\label{eq:capbarS}
\end{equation}
Combining this with (\ref{eq:capS=0}), we obtain
\begin{equation}
J'_{2\ell-1}\cap S=(0).\label{eq:capS(t)=0}
\end{equation}
Thus the ideal $J'_{2\ell-1}\cap\frac{R'}{P'_{\ell-1}}$ contracts
to $(0)$ in $S$. Since $\frac{R'}{P'_{\ell-1}}$ is \'etale over
$S$, this implies the desired equality (\ref{eq:contractsto0}).
This completes the proof of Lemma \ref{contractsto0}.
\end{proof}

Since $\frac{\hat R'}{H'_{2\ell}\hat R'}$ is a complete regular
local ring and $(w',v')$ is a set of representatives of a minimal
set of generators of its maximal ideal $\frac{m'\hat
R'}{H'_{2\ell}}$, there exists a complete local domain
$R'_\ell$ (not necessarily regular) such that $\frac{\hat
R'}{H'_{2\ell-1}}\cong R'_\ell[[w',v']]$. Consider the ring
homomorphism
\begin{equation}
R'_\ell[[w',v']]\rightarrow R'_\ell[[w']],\label{eq:homomorphisms}
\end{equation}
obtained by taking the quotient modulo $(v')$. By
(\ref{eq:homomorphisms}), the quotient of $\frac{\hat
R'}{H'_{2\ell-2}}$ by $J'_{2\ell-1}$ is the integral domain
$R'_\ell[[w']]$, hence $J'_{2\ell-1}$ is prime.

Consider a local blowing up $R'\rightarrow R''$ in $\mathcal T$. Because of the stability assumption on $R$, the ring $\frac{\hat
R''}{H''_{2\ell-2}}\otimes_R\kappa(P''_{l-1})$ is finite over $\frac{\hat R'}{H'_{2\ell-2}}\otimes_R\kappa(P'_{l-1})$; hence the ring $\lim\limits_{\overset\longrightarrow{R''\in\mathcal T}}\left(\frac{\hat R''}{H''_{2\ell-2}}\otimes_R\kappa(P''_{l-1})\right)$ is integral over $\frac{\hat R'}{H'_{2\ell-2}}\otimes_R\kappa(P'_{l-1})$. In particular, there exists a prime ideal in 
$$
\lim\limits_{\overset\longrightarrow{R''\in\mathcal T}}\left(\frac{\hat R''}{H''_{2\ell-2}}\otimes_R\kappa(P''_{l-1})\right),
$$
lying over $J'_{2\ell-1}\frac{\hat R'}{H'_{2\ell-2}}\otimes_R\kappa(P'_{l-1})$. Pick and fix one such prime ideal. Intersecting this ideal with $\frac{\hat R''}{H''_{2\ell-2}}$ for each $R''\in\mathcal T$, we obtain a tree $J''_{2\ell-1}$ of prime ideals of $\frac{\hat R''}{H''_{2\ell-2}}$, $R''\in\mathcal T$.

Our next task is to define the restriction of the valuation
$\hat\mu_{2\ell}$ to the ring $\frac{\hat R'}{J'_{2\ell-1}}$. By
the induction assumption, $\hat\mu_{2\ell+2}$ is already defined
on $\lim\limits_{\overset\longrightarrow{R'\in\mathcal{T}}}\frac{\hat R'}{J'_{2\ell+1}\hat R'}$.
For all \textit{stable} $R''\in\mathcal{T}$ we have the isomorphism $gr_{\hat\mu_{2\ell+2}}\frac{\hat R''}{J''_{2\ell+1}}\cong
gr_{\mu_{\ell+1}}\frac{R''}{P''_\ell}$ of graded algebras (in particular, $gr_{\hat\mu_{2\ell+2}}\frac{\hat
R''}{J''_{2\ell+1}}$ is scalewise birational to $gr_{\mu_{\ell+1}}\frac{R''}{P''_\ell}$ for any $R''\in\mathcal{T}$ and
$\hat\mu_{2\ell+2}$ has the same value group $\Delta_\ell$ as $\mu_{\ell+1}$).

Define the prime ideals $\tilde H''_{2\ell-2}=\tilde H''_{2\ell-1}$
to be equal to the preimage of $J''_{2\ell-1}$ in $\hat R''$
and $\tilde H''_{2\ell}=\tilde H''_{2\ell+1}$ the preimage of $J''_{2\ell+1}$ in $\hat R''$. By definition of tight extensions,
the valuation $\hat\nu_{2\ell+1}$ must be trivial. It remains to describe the valuation $\hat\mu_{2\ell}$ on $\frac{\hat R''}{J''_{2\ell-1}}$, $R''\in\mathcal T$. We will first define $\hat\nu_{2\ell}$ and then put $\hat\mu_{2\ell}=\hat\nu_{2\ell}\circ\hat\mu_{2\ell+2}$.

By definition of tight extensions, the value group of
$\hat\mu_{2\ell}$ must be equal to $\Delta_{\ell-1}$ and that of $\hat\nu_{2\ell}$ to $\frac{\Delta_{\ell-1}}{\Delta_\ell}$. For a positive element $\bar\beta\in\frac{\Delta_{\ell-1}}{\Delta_\ell}$, define the candidate for
$\hat\nu_{2\ell}$-ideal of $\frac{\hat R''_{\tilde H''_{2\ell}}}{\tilde H''_{2\ell-1}\hat R''_{\tilde H''_{2\ell}}}$ of value $\bar\beta$, denoted
by $\hat{\mathcal P}''_{\beta\ell}$, by the formula
\begin{equation}
\hat{\mathcal P}''_{\bar\beta\ell}=\frac{\P''_{\bar\beta}\hat R''_{\tilde H''_{2\ell}}}{\tilde H''_{2\ell-1}\hat R''_{\tilde H''_{2\ell}}}.\label{eq:validealmain}
\end{equation}
\begin{conjecture}\label{strongcontainment} The elements $\phi_j$ of Lemma \ref{contractsto0} can be chosen in such a way that the following condition holds. For each positive element $\beta\in\frac{\Delta_{\ell-1}}{\Delta_\ell}$ and each tree morphism $R'\rightarrow R''$ in $\mathcal T$, we have
$$
\hat{\mathcal P}''_{\beta\ell}\cap\hat R'_{\tilde H'_{2\ell}}=\hat{\mathcal P}'_{\beta\ell}.
$$
\end{conjecture}
\begin{conjecture}\label{containment} The elements $\phi_j$ of Lemma \ref{contractsto0} can be chosen in such a way that
\begin{equation}
\bigcap\limits_{\bar\beta\in\left(\frac{\Delta_{\ell-1}}{\Delta_\ell}\right)_+}
\left(\P'_{\bar\beta}+\tilde H'_{2\ell-1}\right)\hat R'_{\tilde
H'_{2\ell}}\subset\tilde H'_{2\ell-1}. \label{eq:restrictionmain}
\end{equation}
\end{conjecture}
For the rest of this section assume that Conjectures \ref{strongcontainment} and \ref{containment} are true.

For all $\bar\beta\in\left(\frac{\Delta_{\ell-1}}{\Delta_\ell}\right)_+$,
we have the natural isomorphism
$$
\lambda_{\bar\beta}:\frac{\P'_{\bar\beta}}{\P'_{\bar\beta+}}\otimes_{\kappa(P'_{\ell-1})}\kappa(\tilde
H'_{2\ell})\longrightarrow\frac{\hat{\mathcal P}'_{\bar\beta}}{\hat{\mathcal P}'_{\bar\beta+}}
$$
of $\kappa(\tilde H'_{2\ell})$-vector spaces. The following fact is an easy consequences of Conjecture \ref{strongcontainment}:
\begin{corollary}\label{integraldomain}\textbf{(conditional on Conjecture \ref{strongcontainment})} If the elements $\phi_j$ of Lemma \ref{contractsto0} can be chosen as in Conjecture \ref{strongcontainment} then the graded algebra
$$
\gr_{\nu_\ell}\frac{R'_{P'_\ell}}{P'_{\ell-1}R'_{P'_\ell}}\otimes_{\kappa(P'_{\ell-1})}\kappa(\tilde
H'_{2\ell})\cong\bigoplus\limits_{\bar\beta\in\left(\frac{\Delta_{\ell-1}}{\Delta_\ell}\right)_+}\frac{\hat{\mathcal P}'_{\bar\beta}}{\hat{\mathcal P}'_{\bar\beta+}}
$$
is an integral domain.
\end{corollary}
\medskip

For a non-zero element $x\in\frac{\hat R'_{\tilde H'_{2\ell}}}{\tilde H'_{2\ell-1}}$, let $\operatorname{Val}_\ell(x)=\left\{\left.\beta\in\nu_\ell\left(\frac{R'}{P'_{\ell-1}}\setminus\{0\}\right)\ \right|\ x\in\hat{\mathcal P}'_{\beta\ell}\right\}$.
We define $\hat\nu_{2\ell}$ by the formula
\begin{equation}\label{eq:mu2l}
\hat\nu_{2\ell}(x)=\max\ \operatorname{Val}_\ell(x).
\end{equation}
Since $\nu_\ell$ is a rank 1 valuation, centered in a local noetherian domain $\frac{R'}{P'_{\ell-1}}$, the semigroup $\nu_\ell\left(\frac{R'}{P'_{\ell-1}}\setminus\{0\}\right)$ has order type $\mathbb N$, so by (\ref{eq:restrictionmain}) the set $Val_\ell(x)$ contains a maximal element. This proves that the valuation $\hat\nu_{2\ell}$ is well defined by the formula (\ref{eq:mu2l}),
and that we have a natural isomorphism of graded algebras
$$
\gr_{\nu_\ell}\frac{R'_{P'_\ell}}{P'_{\ell-1}R'_{P'_\ell}}\otimes_{\kappa(P'_{\ell-1})}\kappa(\tilde
H'_{2\ell})\cong\gr_{\hat\nu_{2\ell}}\frac{\hat R_{\tilde
H'_{2\ell}}}{\tilde H'_{2\ell-1}}.
$$
Since the above construction is valid for all $R\in\mathcal T$, $\hat\nu_{2\ell}$ extends naturally to a valuation centered in the
ring $\lim\limits_{\overset\longrightarrow{R'}}\frac{\hat
R''}{\tilde H''_{2\ell-1}\hat R''}$ (by abuse of notation, this extension will also be denoted by $\hat\nu_{2\ell}$).

 The extension $\hat\mu_{2\ell}$ of $\mu_\ell$ to
$\lim\limits_{\overset\longrightarrow{R''\in\mathcal{T}(R')}}\frac{\hat
R''}{\tilde H''_{2\ell-1}\hat R'}$ is defined by
$\hat\mu_{2\ell}=\hat\nu_{2\ell}\circ\hat\mu_{2\ell+2}$.

This completes the proof of Conjecture \ref{teissier1} (assuming Conjectures \ref{strongcontainment} and \ref{containment}) by descending induction on $\ell$.\hfill$\Box$\medskip

The next Corollary of Conjecture \ref{teissier1} gives necessary conditions for $\hat\nu_-$ to be uniquely determined by
$\nu$; it also shows that the same conditions are sufficient for $\hat\nu_-$ to be the unique minimal extension of $\nu$, that is, to satisfy
\begin{equation}
\tilde H'_i=H'_i,\quad0\le i\le 2r.\label{eq:tilde=nothing}
\end{equation}
Suppose given a tree $\left\{\tilde H'_0\right\}$ of minimal prime ideals of $\hat R'$ (in
particular, $R'\cap\tilde H'_0=(0)$). If the valuation $\nu$ admits an extension to a valuation $\hat\nu_-$ of
$\lim\limits_{\overset\longrightarrow{R'}}\frac{\hat R'}{\tilde H'_0}$, then $\tilde H'_0$ is the 0-th prime ideal of $\hat R'$, determined by $\hat\nu_-$. Since $\tilde H'_0$ is assumed to be a \textit{minimal} prime, we have $\tilde H'_0=H'_0$ by Proposition \ref{Hintilde}.
\begin{remark} Let the notation be as in Conjecture \ref{teissier1}. Denote the tree of prime ideals $\{\tilde H'_0\}$ by $\{H'\}$ for short. Consider a homomorphism
\begin{equation}
R'\rightarrow R''\label{eq:R'toR''}
\end{equation}
in $\mathcal T$. Assume that the local rings $R'$ and $\frac{\hat R'}{H'}$ are regular, and let $V=(V_1,\dots,V_s)$ be a minimal set of generators of $H'$. Then $V$ can be extended to a regular system of parameters for $\hat R'$. We have an isomorphism $\hat R'\cong\frac{\hat R'}{H'}[[V]]$. The morphism (\ref{eq:R'toR''}) induces an isomorphism $\hat R'_{H'}\cong \hat R''_{H''}$, so that $V$ induces a regular system of parameters of $\hat R''_{H''}$. In particular, the $H''$-adic valuation of $\hat R''_{H''}$ coincides with the $H'$-adic valuation of $\hat R'_{H'}$. On the other hand, we do not know, assuming that $R''$ and $\frac{\hat R''}{H''}$ are regular and $ht\ H''=ht\ H'$, whether $V$ induces a minimal set of generators of $H''$; we suspect that the answer is ``no''.
\end{remark}
\begin{corollary}\label{uncond1}\textbf{(conditional on Conjecture \ref{teissier1})} If the valuation $\nu$ admits a unique extension to a valuation $\hat\nu_-$ of $\lim\limits_{\overset\longrightarrow{R'}}\frac{\hat R'}{H'_0}$, then the following conditions hold:

(1) $ht\ H'_1\le 1$

(2) $H'_i=H'_{i-1}$ for all odd $i>1$.
\medskip
Moreover, this unique extension $\hat\nu_-$ is minimal.

Conversely, assume that (1)--(2) hold. Then there exists a unique minimal extension $\hat\nu_-$ of $\nu$ to $\lim\limits_{\overset\longrightarrow{R'}}\frac{\hat R'}{H'_0}$.
\end{corollary}
\begin{proof} The fact that conditions (1), (2) and equations (\ref{eq:tilde=nothing}) determine $\hat\nu_-$ uniquely is nothing but Proposition \ref{uniqueness1}. Conversely, assume that there exists a unique extension $\hat\nu_-$ of $\nu$ to $\lim\limits_{\overset\longrightarrow{R'}}\frac{\hat R'}{H'_0}$. By Remark \ref{minimalexist}, there exist minimal extensions of $\nu$ to $\lim\limits_{\overset\longrightarrow{R'}}\frac{\hat R'}{H'_0}$, hence $\hat\nu_-$ must be minimal.

Next, by Conjecture \ref{teissier1}, there exists a tree of prime
ideals $\tilde H'$ with $H'\cap R'=(0)$ and a tight extension
$\hat\mu_-$ of $\nu$ to
$\lim\limits_{\overset\longrightarrow{R'}}\frac{\hat R'}{H'}$. The
ideals $H'$ are both the the 0-th and the 1-st ideals determined by
$\hat\mu_-$; in particular, we have
\begin{equation}
H'_0\subset H'_1\subset H'\label{eq:inH'}
\end{equation}
by Proposition \ref{Hintilde}. Now, take any valuation $\theta$,
centered in the regular local ring $\frac{R'_{H'}}{H'_0}$, such that
the residue field $k_\theta=\kappa(H')$. Then the composition
$\hat\mu_-\circ\theta$ is an extension of $\nu$ to
$\lim\limits_{\overset\longrightarrow{R'}}\frac{\hat R'}{H'_0}$,
hence
\begin{equation}
\hat\mu_-\circ\theta=\hat\nu_-\label{eq:mucirctheta}
\end{equation}
by uniqueness. For $i\ge1$, the $i$-th prime ideal, determined by
$\hat\mu_-\circ\theta=\hat\nu_-$ coincides with that determined by
$\hat\mu_-$. Since $\nu$ is minimal and $\hat\mu_-$ is tight, we
obtain condition (2) of the Corollary. Finally, if we had $ht\
H'>1$, there would be infinitely many choices for $\theta$,
contradicting \ref{eq:mucirctheta} and the uniqueness of
$\hat\nu_-$. Thus $ht\ H'\le1$. Combined with \ref{eq:inH'}, this
proves (1) of the Corollary. This completes the proof of Corollary
(\ref{uncond1}), assuming Conjecture \ref{teissier1}.
\end{proof}

\appendix{Regular morphisms and G-rings.}

In this Appendix we recall the definitions of regular homomorphism, G-rings and excellent and quasi-excellent rings. We also summarize some of their basic properties used in the rest of the paper.
\begin{definition}\label{regmor} (\cite{Mat}, Chapter 13, (33.A), p. 249) Let $\sigma:A\rightarrow B$
be a homomorphism of noetherian rings. We say that $\sigma$ is {\bf regular} if
it is flat, and for every prime ideal $P\subset A$, the ring $B\otimes_A\kappa(P)$ is
geometrically regular over $\kappa(P)$ (this means that for any finite field
extension $\kappa(P)\rightarrow k'$, the ring $B\otimes_Ak'$ is regular).
\end{definition}
\begin{remark} If $\kappa(P)$ is perfect, the ring $B\otimes_A\kappa(P)$ is geometrically regular over $\kappa(P)$ if and
only if it is regular.
\end{remark}
\begin{remark} It is known that a morphism of finite type is regular in the above sense if and only if it is smooth (that is, formally smooth in the sense of Grothendieck with respect to the discrete topology), though we do not use this fact in the present paper.
\end{remark}
Regular morphisms come up in a natural way when one wishes to pass to the formal completion of a local ring:
\begin{definition}\label{Gring} (\cite{Mat}, (33.A) and (34.A)) Let $R$ be a noetherian ring. For a maximal ideal $m$ of $R$, let $\hat R_m$ denote the $m$-adic completion of $R$. We say that $R$ is a {\bf G-ring} if for every maximal ideal
$m$ of $R$, the natural map $R\rightarrow\hat R_m$ is a regular homomorphism.
\end{definition}
The property of being a G-ring is preserved by localization and passing to rings essentially of finite type over $R$.
\begin{definition}\label{quasiexcellent} (\cite{Mat}, Definition 2.5, (34.A), p. 259) Let $R$ be a noetherian
ring. We say that $R$ is {\bf quasi-excellent} if the following two
conditions hold:

(1) $R$ is J-2, that is, for any scheme $X$, which is reduced and of
finite type over $\spec\ R$, $Reg(X)$ is open in the Zariski topology.

(2) For every maximal ideal $m\subset R$, $R_m$ is a G-ring.
\end{definition}
It is known \cite{Mat} that a \textit{local} G-ring is automatically J-2, hence automatically quasi-excellent. Thus for local rings ``G-ring'' and ``quasi-excellent'' are one and the same thing.
A ring is said to be \textbf{excellent} if it is quasi-excellent and
universally catenary, but we do not need the catenary condition in this paper.

Both excellence and quasi-excellence are preserved by localization and passing to rings of finite type over $R$
(\cite{Mat}, Chapter 13, (33.G), Theorem 77, p. 254). In
particular, any ring essentially of finite type over a field, $\mathbf Z$, $\mathbf Z_{(p)}$, $\mathbf Z_p$,
the Witt vectors or any other excellent Dedekind domain is
excellent. See \cite{Nag} (Appendix A.1, p. 203) for some examples of
non-excellent rings.

Rings which arise from natural constructions in algebra and geometry are
excellent. Complete and complex-analytic local rings are excellent (see
\cite{Mat}, Theorem 30.D) for a proof that any complete local ring is excellent).

Finally, we remark that the category of quasi-excellent rings is a natural one for doing algebraic geometry, since it is the
largest reasonable class of rings for which resolution of singularities can hold. Namely, let $R$ be a noetherian ring.
Grothendieck (\cite{EGA}, IV.7.9) proves that if all of the irreducible closed subschemes of $\spec\ R$ and all of their
finite purely inseparable covers admit resolution of singularities, then $R$ must be quasi-excellent.
Grothendieck's result means that the largest {\it class} of noetherian
rings, closed under homomorphic images and finite purely inseparable extensions, for which resolution of singularities
could possibly exist, is {\it quasi-excellent} rings.

We now summarize the specific uses we make of quasi-excellence in the present paper. We begin by recalling three results from \cite{Mat} and \cite{Nag}. As a point of terminology, we note that Nagata's ``pseudo-geometric'' rings are now commonly known as Nagata rings. Quasi-excellent rings are Nagata (\cite{Mat}, (33.H), Theorem 78).
\begin{theorem}\label{annormal}(\cite{Mat}, (34.C), Theorem 79) Let $R$ be an excellent normal local ring. Then $R$ is analytically normal (this means that its formal completion $\hat R$ is normal).
\end{theorem}
\begin{theorem}(\cite{Nag}, (43.20), p. 187) Let $R$ be a local integral domain, $\tilde R$ its Henselization and $R'$ its normalization. There is a natural one-to-one correspondence between the minimal primes of $\tilde R$ and the maximal ideals of $R'$.
\end{theorem}
\begin{proposition}\label{anirred}(\cite{Nag}, Corollary (44.3), p. 189) Let $R$ be a quasi-excellent analytically normal local ring. Then its Henselization $\tilde R$ is analytically irreducible and is algebraically closed in its formal completion.
\end{proposition}
From the above results we deduce
\begin{corollary}\label{notnormal} Let $(R,\mathbf m)$ be a Henselian excellent local domain. Then $R$ is analytically irreducible and is algebraically closed in $\hat R$.
\end{corollary}
\begin{proof} If, in addition, we assume $R$ to be normal, the result follows from Theorems \ref{annormal} and \ref{anirred}. In the general case, let $R'$ denote the normalization of $R$. Then $R'$ is a Henselian normal quasi-excellent local ring, so it satisfies the conclusions of the Corollary. Consider the commutative diagram
\begin{equation}\label{eq:CD}
\xymatrix{R\ar[d]_-\psi \ar[r]^-\phi & {\hat R}\ar[d]_-{\hat\psi}\\
R'\ar[r]^-{\phi'}&{\hat R'}}
\end{equation}
where $\hat R'$ stands for the formal completion of $R'$.
Since $R$ is Nagata, $R'$ is a finite $R$-module. Thus $\phi'$ coincides with the $\mathbf m$-adic completion of $R'$, viewed as an $R$-module. Hence $\hat R'\cong R'\otimes_R\hat R$. Since $\psi$ is injective and $\hat R$ is flat over $r$, the map $\hat\phi$ is also injective. Since $R'$ is analytically irreducible, $\hat R'$ is a domain, and therefore so is its subring $\hat R$. This proves that $R$ is analytically irreducible.

To prove that $R$ is algebraically closed in $\hat R$, take an element $x\in\hat R$, algebraic over $R$. Since all the maps in \ref{eq:CD} are injective, let us view all the rings involved as subrings of $\hat R'$. Since $R'$ is algebraically closed in $\hat R'$, we have $x\in R'$, in particular, we may write $x=\frac ab$ with $a,b\in R$. Now, since $(a)\hat R\subset(b)\hat R$ and $\hat R$ is faithfully flat over $R$, we have $(a)\subset(b)$ in $R$, so $x=\frac ab\in R$. This proves that $R$ is algebraically closed in $\hat R$. The Corollary is proved.
\end{proof}

Next we summarize, in a more specific manner, the way in which these results are applied in the present paper. The main applications are as follows.

(1) Let $R$ be an excellent local domain, $P$ a prime ideal of $R$ and $H_i\subset H_{i+1}$ two prime ideals of $\hat R$ such that
\begin{equation}
H_i\cap R=H_{i+1}\cap R=P.\label{eq:HicontractstoP}
\end{equation}
Then $\frac RP$ is also excellent. Definitions \ref{regmor}, \ref{Gring} and \ref{quasiexcellent} imply that the ring $\hat R\otimes_R\kappa(P)$ is geometrically regular over $P$, in particular, regular. Moreover, (\ref{eq:HicontractstoP}) implies that the ideal $\frac{H_{i+1}}{P\hat R}$ is a prime ideal of $\frac{\hat R}{P\hat R}$, disjoint from the natural image of $R\setminus P$ in $\frac{\hat R}{P\hat R}$. Thus the local ring $\frac{\hat R_{H_{i+1}}}{P\hat R_{H_{i+1}}}$ is a localization of $\hat R\otimes_R\kappa(P)$ at the prime ideal $H_{i+1}(\hat R\otimes_R\kappa(P))$ and so is a local ring, geometrically regular over $\kappa(P)$, in particular, a regular local ring and, in particular, a domain.

(2) Assume, in addition, that $H_i$ is a minimal prime of $P\hat R$. Since $\frac{\hat R_{H_{i+1}}}{P\hat R_{H_{i+1}}}$ is a domain, $H_i$ is the only minimal prime of $P\hat R$, contained in $H_{i+1}$. We have $P\hat R_{H_{i+1}}=H_i\hat R_{H_{i+1}}$.



\end{document}